\documentclass[11pt,book]{now}
\usepackage[active]{srcltx}
\usepackage{color,tbl,bm,nameref}

\graphicspath{{L:/e-client-aw/now/tcs-055/}}

\def\nn{\nonumber}

\makeatletter

\def\del#1{\relax}

\def\prbox{\par
    \vskip-\lastskip\vskip-\baselineskip\hbox to \hsize{\hfill\fboxsep0pt\fbox{\phantom{\vrule width5pt height5pt depth0pt}}}}

\def\eqnarray{\stepcounter{equation}\let\@currentlabel\theequation
\global\@eqnswtrue\m@th \global\@eqcnt\z@\tabskip\@centering\let\\\@eqncr
$$\halign to\displaywidth\bgroup\@eqnsel\hskip\@centering
  $\displaystyle\tabskip\z@{##}$&\global\@eqcnt\@ne
  \hskip 0.8\arraycolsep \hfil${##}$\hfil
  &\global\@eqcnt\tw@ \hskip 0.8\arraycolsep $\displaystyle\tabskip\z@{##}$\hfil
   \tabskip\@centering&\llap{##}\tabskip\z@\cr}

\def\endeqnarray{\@@eqncr\egroup
      \global\advance\c@equation\m@ne$$\global\@ignoretrue}

\def\anauthor#1{%
   \expandafter\csname author\romannumeral#1\endcsname
}

\def\anaff#1{{\par 
\@ifundefined{affil\romannumeral#1}{}{%
    \expandafter\csname affil\romannumeral#1\endcsname, }%
\@ifundefined{@auaddline\romannumeral#1}{}{%
   \expandafter\csname @auaddline\romannumeral#1\endcsname, }%
\@ifundefined{@aucity\romannumeral#1}{}{%
    \expandafter\csname @aucity\romannumeral#1\endcsname, }%
\@ifundefined{@auzip\romannumeral#1}{}{%
    \expandafter\csname @auzip\romannumeral#1\endcsname, }%
\@ifundefined{@aucountry\romannumeral#1}{}{%
    \expandafter\csname @aucountry\romannumeral#1\endcsname, }%
\@ifundefined{@auemail\romannumeral#1}{}{%
    \expandafter\csname @auemail\romannumeral#1\endcsname}%
\par
}}

\def\makeintropage{%
\preface{Beginning with the paper \textit{A Topological Approach to Evasiveness}
by Kahn, Saks, and Sturtevant~\cite{kss1984}, there have been a number of interesting
research papers that use topological methods to prove lower bounds on the complexity of
graph properties. This is a fascinating topic that lies at the interface between
mathematics and theoretical computer science. The goal of this text is to offer an
integrated version of the underlying proofs in this body of research. While there are a
number of very good expositions available on topological methods in decision-tree
complexity, all those that I have seen refer to other sources for the proofs of some
topological results (including the key fixed-point theorem of R. Oliver
\cite{oliver1975}). In this text I have attempted to give a completely self-contained
account.\\
\hspace*{14pt}I have not assumed that the reader has any prior background in algebraic
topology---all constructions from that subject are developed from scratch.  The only
prerequisite is a high level of comfort with discrete mathematics and linear
algebra. Indeed, though I will sometimes refer to subsets of $\mathbb{R}^n$ for
intuition, all the results in this text finally rest on manipulations of finite sets.\\
\hspace*{14pt}While I was preparing this work for publication, I learned about the new
book \textit{A Course in Topological Combinatorics} by Mark de Longueville
\cite{longueville2013}.  This book gives a similar treatment of topological methods for
proofs of complexity of graph properties, including a proof of Oliver's theorem.  Whereas
my text is more economical and is intended to offer as direct a route as possible to
\cite{kss1984} and its related results, de Longueville's book is broader in scope and
encompasses topological methods for other combinatorial problems.  I hope that the
community will find both works beneficial.\\
\hspace*{14pt}The general flow of the text is to begin with foundational material and
then to build up more complex results at a steady pace.  The capstone results, which
consist of three lower bounds on the complexity of graph properties, appear in the final
part of the text. My undergraduate advisor Richard Hain once said that the final goal of
mathematics is ``to tell a good story.'' That is what I have attempted to do here, and I
hope the reader will enjoy the result.}
 \newpage
\def\textregistered{\textcircled{{\fontsize{4}{4}\selectfont R}}} 
\noindent\hbox to \textwidth{%
    \vbox{\hsize=.6\textwidth \footnotesize
\@journaltitleprefix\\\@journaltitle\\
Vol.~\@volume, No.~\@issue\ (\@pubyear) \@firstpage--\@lastpage\\
\copyright\ 2013\ \@copyrightowner\\ DOI: \@DOI}\hss \vbox{\hsize=.4\textwidth
    \raggedleft
    \includegraphics[width=53bp]{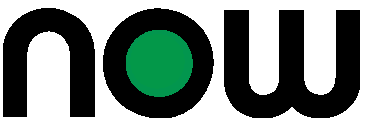}\\
    \includegraphics[width=81bp]{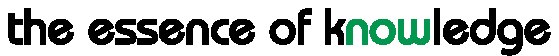}%
    \par
}}\par \vspace{24pt}
\def\thefootnote{*}
\begin{center}
{\fontsize{16}{20}\selectfont\sffamily\bfseries{{Evasiveness of Graph Properties and\\[5pt]
Topological Fixed-Point Theorems}}}\par
\end{center}
\vspace{20pt} \@theauthors\par \vspace{20pt} {\raggedright\@theinstitutes\par}
\vspace{12pt}
\@abstract
\newpage
}

\def\makefulltitlepage{%
\title{{\@articletitle}}
  \vbox to \textheight{\vskip-17\p@\vss\raggedleft%
    \noindent\vrule\@height4\p@\@width270\p@\par\vskip-5\p@%
    \noindent{\fontsize{22}{34}\selectfont\sffamily\bfseries\leftskip60pt plus1fill\rightskip\z@ \@title\par}
    \vskip9\p@%
    \noindent\vrule\@height4\p@\@width270\p@\par%
  \vss}%
  \newpage\ \newpage
}

\def\firstpage#1{\def\@firstpage{#1}}
\def\lastpage#1{\def\@lastpage{#1}}

\def\maketitlepage{%
  \vbox to\textheight{%
   \noindent\flushright{\fontsize{26}{34}\selectfont\sffamily\bfseries\leftskip0pt plus1fil\rightskip\z@ \@title \par}
    \vskip21\p@
    \noindent\vrule\@height2\p@\@width\textwidth\vskip30\p@
\sffamily\normalsize\raggedleft
{\fontsize{16}{18}\selectfont\bfseries {Carl A. Miller}}\\[18pt]
{ \fontsize{14}{16}\selectfont \itshape{University of Michigan}\\[3pt]
USA\\[6pt]
carlmi@umich.edu}\\[2pc]
    \vfill
    \flushright{\includegraphics[width=36.5mm]{now_logo}}%
    \flushright{\includegraphics[width=57mm]{essence_logo}}
    \flushright{\sffamily \fontsize{14}{16}\selectfont Boston -- Delft}%
}
}

\renewcommand\tableofcontents{%
    \if@twocolumn
      \@restonecoltrue\onecolumn
    \else
      \@restonecolfalse
    \fi
    \chapter*{\contentsname}
    \@mkboth{Contents}{Contents}
    \thispagestyle{plain}
    \vspace*{-10.6em}
    \@starttoc{toc}%
    \if@restonecol\twocolumn\fi
    }

\makeatother

\input xy
\xyoption{all}

\firstpage{337}

\lastpage{415}

\begin{document}

\pagenumbering{roman}

\isbn{978-1-60198-664-1}

\DOI{10.1561/0400000055}

\abstract{Many graph properties (e.g., connectedness, containing a complete
\hbox{subgraph}) are known to be difficult to check.  In a decision-tree model, the cost
of an algorithm is measured by the number of edges in the graph that it queries.  R. Karp
conjectured in the early 1970s that all monotone graph properties are evasive---that is,
any algorithm which computes a monotone graph property must check all edges in the worst
case. This conjecture is unproven, but a lot of progress has been made. Starting with the
work of Kahn, Saks, and Sturtevant in 1984, topological methods have been applied to
prove partial results on the Karp conjecture.  This text is a tutorial on these
topological \hbox{methods}. I give a fully self-contained account of the central proofs
from the paper of Kahn, Saks, and Sturtevant, with no prior knowledge of topology
assumed.  I also briefly survey some of the more recent results on \hbox{evasiveness}.}

\articletitle{Evasiveness of Graph Properties and Topological Fixed-Point Theorems}

\authorname1{Carl A. Miller}
\affiliation1{University of Michigan}
\author1address2ndline{Department of Electrical Engineering and Computer Science, 2260 Hayward St.}
\author1city{Ann Arbor}
\author1zip{MI 48109-2121}
\author1country{USA}
\author1email{carlmi@umich.edu}

\relax \makeatletter \input tcs.jtl\makeatother
 \volume{7}
 \issue{4}
 \copyrightowner{C. A. Miller}
 \pubyear{2011}

\maketitle

\setcounter{page}{1}

\newtheorem{notation}[theorem]{Notation}
\newtheorem{conj}[theorem]{Conjecture}

\def\mb{\mathbb}
\def\mf{\mathbf}

\def\im{\textnormal{im }}
\def\ker{\textnormal{ker }}
\def\sign{\textnormal{sign }}
\def\Tr{\textnormal{Tr}}
\def\mod{\textnormal{mod }}
\def\bar{\textnormal{bar}}
\def\weight{\textnormal{weight}}
\def\Perm{\textnormal{Perm}}
\def\Pow{\textnormal{Pow}}
\def\Sym{\textnormal{Sym}}
\def\coker{\textnormal{coker }}

\chapter{Introduction}\label{chap1}

Let $V$ be a finite set of size $n$, and let $\mathbf{G} ( V )$ denote the set of
undirected graphs on $V$. For our purposes, a \textbf{graph property} is simply a
function
\begin{eqnarray}
f \colon \mathbf{G} ( V ) \to \{ 0, 1 \}
\end{eqnarray}
which is such that whenever two graphs $Z$ and $Z'$ are isomorphic, $f ( Z ) = f ( Z' )$.
A graph $Z$ ``has property $f$'' if $f ( Z ) = 1$.

We can measure the cost of an algorithm for computing $f$ by counting the number of
edge-queries that it makes.  We assume that these edge-queries are adaptive (i.e., the
choice of query may depend on the outcomes of previous queries). An algorithm for $f$ can
thus be represented by a binary decision-tree (see Figure~\ref{dectreefigure}). The
\textbf{decision-tree complexity of $f$}, which we denote by $D ( f )$, is the least
possible depth for a decision-tree that computes $f$.  In other words, $D ( f )$ is the
number of edge-queries that an optimal algorithm for $f$ has to make in the worst case.

\begin{figure}[!t]
 \centerline{\includegraphics{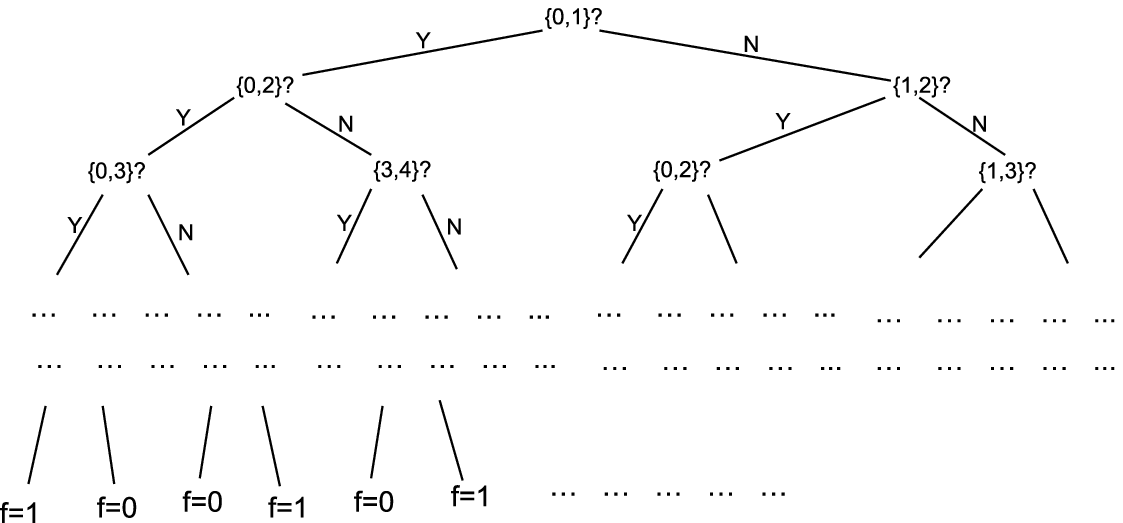}}
\fcaption{A binary decision tree.\label{dectreefigure}}
\end{figure}

Some graph properties are difficult to compute.  For example, let $h( Z ) = 1$ if and
only if $Z$ contains a cycle.  Suppose that an algorithm for $h$ makes queries to an
adversary whose goal is to maximize cost.  The adversary can adaptively construct a graph
$Y$ to foil the \hbox{algorithm}:~each time a pair $( i, j ) \in V \times V$ is queried,
the adversary answers ``yes,'' unless the inclusion of that edge would necessarily make
the graph $Y$ have a cycle, in which case he answers ``no.'' After $\binom{n}{2} - 1$
edge-queries by the algorithm have been made, the known edges will form a tree on the
elements of $V$.  The algorithm at this point will have no choice but to query the last
unknown edge to determine whether or not a cycle exists. We conclude from this argument
that $h$ is a graph property that has the maximal decision-tree complexity
$\binom{n}{2}$.  Such properties are called \textbf{evasive}.

A graph property is \textbf{monotone} if it is either always preserved by the addition of
edges (monotone-increasing) or always preserved by the deletion of edges
(monotone-decreasing). In 1973 the following conjecture was made \cite{rosenberg1973}.

{\makeatletter
\newtheoremstyle{nowthm}{4pt plus6pt minus4pt}{0pt}{\upshape}{0pt}{\bfseries}{}{.6em}
  {\rule{\textwidth}{.5pt}\par\vspace*{-1pt}\newline\thmname{#1}\thmnumber{\@ifnotempty{#1}{\hspace*{3.65pt}}{#2}$\!\!$}
  \thmnote{{\the\thm@notefont\bf (#3).}}}
\def\@endtheorem{\par\vspace*{-7.8pt}\noindent\rule{\textwidth}{.5pt}\vskip8pt plus6pt minus4pt}
\ignorespaces \makeatother

\begin{conj}[The Karp Conjecture]
All nontrivial monotone graph properties are evasive.
\end{conj}}

\noindent To date, this conjecture is unproven and no counterexamples are known. However
in 1984, a seminal paper was published by Kahn et~al.~\cite{kss1984} which proved the conjecture in some cases.  This
paper showed that evasiveness can be established through the use of topological
fixed-point theorems.  It has been followed by many more papers which exploited its
method to prove better results.

This text is a tutorial on the topological method of~\cite{kss1984}. My goal is to
provide background on the problem and to take the reader through all of the necessary
proofs. Let us begin with some history.

\section{Background}

Research on the decision-tree complexity of graph properties---including properties for
both directed and undirected graphs---dates back at least to the early 1970s
\cite{bbl1974,bollobas1976,hr1972,ht1974,kirkpatrick1974,mw1975,rosenberg1973}. Proofs
were given in early papers that certain specific graph properties are evasive (e.g.,
connectedness, containment of a complete subgraph of fixed size), and that other
properties at least have decision-tree complexity $\Omega (n^2)$. Although it was known
that there are graph properties whose decision-tree complexity is not $\Omega (n^2)$ (see
Example~18 in \cite{bbl1974}),  Aanderaa and Rosenberg conjectured that all
\textbf{monotone} graph properties have decision-tree complexity $\Omega (n^2)$
\cite{rosenberg1973}. This conjecture was proved by Rivest and Vuillemin \cite{rv1976}
who showed that all monotone graph properties satisfy $D (f) \geq n^2/16$.  Kleitman and
Kwiatkowski \cite{kk1980} improved this bound to $D (f) \geq n^2/9$.

Underlying some of these proofs is the insight that if a graph property $f$ has
nonmaximal decision-tree complexity, then the collection of graphs that satisfy $f$ have
some special structure.  For example, if $f$ is not evasive, then in the set of graphs
satisfying $f$ there must be an equal number of graphs having an odd number of edges and
an even number of edges. Rivest and Vuillemin \cite{rv1976} used the fact that if $f$ has
decision-tree complexity $\binom{n}{2} - k$, then the weight enumerator of $f$ (i.e., the
polynomial $\sum_j c_j t^j$, where $c_j$ is the number of $f$-graphs containing~$j$
edges) must be divisible by $(1 + t)^k$.

A topological method for the evasiveness problem was introduced in~\cite{kss1984}.
Suppose that $h$ is a monotone-increasing graph property on a vertex set $\{ 0, 1,
\ldots, n-1 \}$. Let $T$ be the collection of all graphs that do \textit{not} satisfy
$h$.  The set $T$ has the property that if $G$ is in $T$, then all of its subgraphs are
in $T$.  This is a close analogy to the property which defines simplicial complexes in
topology. Let $\{ x_{ab} \mid 0 \leq i < j < n\}$ be a labeled collection of linearly
independent vectors in some vector space $\mathbb{R}^N$. Each graph in $T$ determines a
simplex in $\mathbb{R}^N$: one takes the convex hull of the vectors $x_{ab}$
corresponding to the edges $\{ a, b \}$ that are in the graph.  The union of these hulls
forms a simplicial complex, $\Gamma_h$. The complex for ``connectedness'' on
four vertices (represented in three dimensions) is shown in
Figure~\ref{connectednessfigure}.

\begin{figure}[!t]
 \centerline{\includegraphics{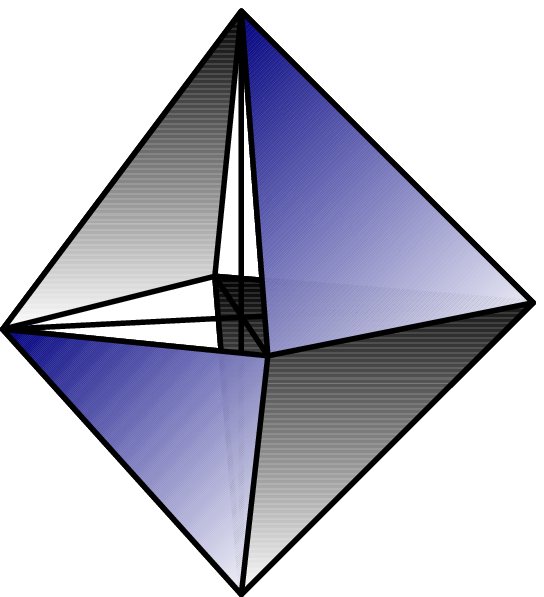}}
\fcaption{The simplicial complex for ``connectedness'' on four
vertices.\label{connectednessfigure}}
\end{figure}

A fundamental insight of \cite{kss1984} is that nonevasiveness can be translated to a
topological condition.  If $h$ is not evasive, then $\Gamma_h$ has a certain topological
property called \textbf{collapsibility}.  This property, which we will define formally
later in this text, essentially means that $\Gamma_h$ can be folded into itself and
contracted to a single point. This property implies the even--odd weight-balance
condition mentioned above, but it is stronger.  In particular, it allows for the
application of topological fixed-point theorems.

The following theorem is attributed to R.~Oliver.

{\makeatletter
\newtheoremstyle{nowthm}{4pt plus6pt minus4pt}{0pt}{\upshape}{0pt}{\bfseries}{}{.6em}
  {\rule{\textwidth}{.5pt}\par\vspace*{-1pt}\newline\thmname{#1}\thmnumber{\@ifnotempty{#1}{\hspace*{3.65pt}}{#2}$\!\!$}
  \thmnote{{\the\thm@notefont\bf (#3).}}}
\def\@endtheorem{\par\vspace*{-7.8pt}\noindent\rule{\textwidth}{.5pt}\vskip8pt plus6pt minus4pt}
\ignorespaces \makeatother

\begin{theorem}[Oliver \cite{oliver1975}]\label{fptquote}
Let $\Gamma$ be a collapsible simplicial complex. Let $G$ be a finite group which
satisfies the following condition:
\begin{itemize}
\item[(*)] There is a normal subgroup $G' \subseteq G$, whose
size is  a power of a prime, such that $G / G'$ is cyclic.
\end{itemize}
Then, any action of $G$ on $\Gamma$ has a fixed point.
\end{theorem}

When $\Gamma = \Gamma_h$, the fixed points of $G$ correspond to graphs, and this theorem
essentially forces the existence of certain graphs that do not satisfy $h$.  This theorem
is the basis for the following result of \cite{kss1984}:

\begin{theorem}[Kahn et~al.~\cite{kss1984}]\label{kss1quote}
Let $f$ be a monotone graph property on graphs of size $p^k$, where $p$ is prime.  If $f$
is not evasive, then it must be trivial.
\end{theorem}

\noindent The proof of this theorem essentially proceeds by demonstrating an appropriate
group action $G$ on the set of graphs of order $p^k$ such that the only $G$-invariant
graphs are the empty graph and the complete graph.

Thus evasiveness is known for all values of $n$ that are prime powers.  What about other
values of $n$? One could hope that if the decision-tree complexity is always
$\binom{p}{2}$ when the vertex set is size $p$, then the quantity $\binom{p}{2}$ is a
lower bound for the cases $p+1$, $p+2$, and so forth.  Unfortunately there is no known
way to show this.  However, all is not lost. The following general theorem is also proved
in \cite{kss1984}.

\begin{theorem}[Kahn et~al. \cite{kss1984}]\label{kss2quote}
Let $f$ be a nontrivial monotone graph property of order $n$.  Then,
\begin{eqnarray}
D (  f ) \geq \frac{n^2}{4} - o ( n^2 ).
\end{eqnarray}
\end{theorem}}

\noindent The paper \cite{kss1984} was then followed by several other papers on
evasiveness by other authors who used the topological approach to prove new results on
evasiveness \cite{bbkk2010,cks2002,king1990,kt2010,triesch1994,triesch1996,yao1988}. Some
of these papers found new group actions $G \circlearrowleft \Delta_h$ to exploit in the
nonprime cases.

The target results of this exposition are Theorems~\ref{kss1quote} and \ref{kss2quote}, and a theorem by Yao on evasiveness of bipartite graphs
\cite{yao1988}.  Now let us summarize what we need to do in order to get there.

\section{Outline of Text}

My goal in this exposition is to give a reader who does not know \hbox{algebraic}
topology a complete tutorial on topological proofs of evasiveness.  Therefore, a fair
amount of space will be devoted to building up concepts from algebraic topology.  I have
tended be economical in my \hbox{discussions} and to develop concepts only on an
as-needed basis. Readers who wish to learn more algebraic topology after this exposition
may want to consult good references such as \cite{hatcher2002,munkres}.

We begin, in \textit{\nameref{basicconceptschapter}}, by formalizing the class of
simplicial complexes and its relation to the class of graph properties. While
we have presented a simplicial complex in this introduction as a subset of
$\mathbb{R}^n$, it can also be defined simply as a collection of finite sets.  (This is
the notion of an \textbf{abstract simplicial complex}.)  Although the definition in terms
of subsets of $\mathbb{R}^n$ is helpful for intuition, the definition in terms of finite
sets is the one we will use in all proofs.

A critical construction in this monograph is the set of \textbf{homology
groups} of a simplicial complex.  These groups are algebraic objects which measure the
shape of the complex, and also~--- crucially for our purposes~--- help us understand the
behavior of the complex under automorphisms. \textit{\nameref{chaincomplexchapter}}
defines homology groups and provides some of the standard theory for them.

In \textit{\nameref{fptchapter}} we prove some topological results. The first is the
Lefschetz fixed-point theorem. One way to state this theorem is to say that any
automorphism of a collapsible simplicial complex has a fixed point.  However we instead
prove a theorem which applies to the more general class of
\textbf{$\mathbb{F}_p$-acyclic} complexes.  A simplicial complex is
$\mathbb{F}_p$-acyclic if its homology groups (over $\mathbb{F}_p$) are trivial.  When a
simplicial complex is $\mathbb{F}_p$-acyclic it behaves much like a collapsible complex
(and in particular, any automorphism has a fixed point). Finally, we prove a version of
Theorem~\ref{fptquote}.  The proof of the theorem depends on finding a tower of subgroups
\begin{eqnarray}
\{ 0 \} = G_0 \subset G_1 \subset G_2 \subset \cdots \subset G_n = G,
\end{eqnarray}
where each quotient $G_i / G_{i-1}$ is cyclic, and performing an inductive argument.

\textit{\nameref{resultschapter}} proves \hbox{Theorem}~\ref{kss1quote}, a
\hbox{bipartite} result of Yao \cite{yao1988}, and Theorem~\ref{kss2quote}. We
conclude with an informal discussion of a few of the more recent results
on decision-tree complexity of graph properties
\cite{bbkk2010,cks2002,king1990,kt2010,triesch1994,triesch1996}.

My primary sources for this exposition were
\cite{duandko,kss1984,munkres,smith1941,yao1988}.  A particular debt is owed to Du
and Ko~\cite{duandko}, which was my first introduction to the subject.

\section{Related Topics}

I will briefly mention two alternative lines of research that are related to the one I
cover here.  One can change the measure of complexity that one is using to measure graph
properties, and this leads to new problems requiring different methods.  A natural
variant is the \textbf{randomized decision-tree complexity.} Suppose that in our
decision-tree model, our algorithm is permitted to make random choices at each step about
which edges to check.  We define the cost of the algorithm on a particular input graph to
be the \textit{expected} number of edge queries, and the cost of the algorithm as a whole
to be the maximum of this quantity over all input graphs. The minimum of this quantity
over all algorithms is the randomized decision-tree complexity, $R ( f )$.

There is a line of research studying the randomized decision tree complexity of monotone
graph properties
\cite{ck2007,fkw2002,groger1992,hajnal1991,king1991,odonnel2005,yao1991}. While it is
easy to see that $R ( f )$ can be less than $\binom{n}{2}$, there are graph properties
for which $R ( f )$ is provably $\Omega ( n^2 )$ (such as the ``emptiness
property''---the property that the graph contains no edges). It is conjectured that $R (
f )$ is always $\Omega ( n^2 )$ for monotone graph properties, just as in the
deterministic model.  The best proved lower bound \cite{ck2007, hajnal1991} is $\Omega (
n^{4/3} \left( \log n \right)^{1/3})$.

Another variant of decision-tree complexity is \textbf{bounded-error quantum query
complexity}. A quantum query algorithm for a graph property uses a quantum ``oracle'' in
its computation. The oracle accepts a quantum state which is a superposition of
edge-queries to a graph, and it returns a quantum state which encodes the answers to
those queries.  The algorithm is permitted to use this oracle along with arbitrary
quantum operations to determine its result. The algorithm is permitted to make errors,
but the likelihood of an error must be below a fixed bound on all inputs.
(See~\cite{bcwz1999}.)

In the quantum case it is clear that a lower bound of $\Omega ( n^2 )$ does not hold:
Grover's algorithm~\cite{ambainis2004} can search a space of size $N$ in time $\Theta (
\sqrt{N} )$ using an oracle model.  With a modified version of Grover's algorithm, one
can compute the emptiness property in time $\Theta ( n )$. There are a number of other
monotone properties for which the quantum query complexity is known to be $o ( n^2 )$
(see \cite{ck2010} for a good summary on this topic). It is conjectured that all monotone
graph properties have quantum query complexity $\Omega ( n )$. The best proved lower
bound is $\Omega ( n^{2/3} )$, from an unpublished result attributed to Santha and Yao
\hbox{(see \cite{syz2004})}.

\section{Further Reading}

Other expositions about topological proofs of evasiveness can be found in \cite{duandko}
(in the context of computational complexity theory) and \cite{kozlov2008} (in the context
of algebraic topology), and also in Lovasz's lecture notes \cite{ly2002}. A reader who
wishes to learn more about algebraic topology can consult \cite{munkres}, or, for a more
advanced treatment, \cite{hatcher2002}.  For the particular subject of the topology of
complexes arising from graphs, there is an extensive treatment \cite{jonsson2008}, which
builds further on many of the concepts that I will discuss here.  And finally, for
readers who generally enjoy reading about applications of topology to problems in
discrete mathematics, the excellent book \cite{matousek2008} contains more material of
the same flavor.  It involves applications of a different topological result (the
Borsuk--Ulam theorem) to some problems in elementary mathematics.

\chapter{Basic Concepts}\label{basicconceptschapter}

\section{Graph Properties}\label{graphpropsection}

This part of the text covers some preliminary material.  We begin by formalizing some
basic terminology for finite graphs.

For our purposes, a \textbf{finite graph} is an ordered pair of sets $(V, E)$, in
which $V$ (the \textbf{vertex set}) is a finite set, and $E$ (the \textbf{edge set}) is
a set of $2$-element subsets of $V$.  For example, the pair
\begin{eqnarray}\label{graphoforder4}
\left( \left\{ 0, 1, 2, 3 \right\} , \left\{ \{ 0, 1 \} ,
\{ 0, 2 \} , \{ 1, 2 \}, \{ 2, 3 \} \right\} \right)
\end{eqnarray}
is a finite graph with four vertices, diagrammed in Figure~\ref{4graphfigure}.

\begin{figure}
 \centerline{\includegraphics{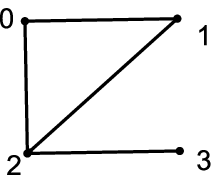}}
\fcaption{A graph on four vertices.\label{4graphfigure}}
\end{figure}

An \textbf{isomorphism} between two finite graphs is a one-to-one correspondence between
the vertices of the two graphs which matches up their edges.  In precise terms, if $G =
(V, E)$ and $G' = ( V' , E' )$ are two graphs, then an isomorphism between $G$ and $G'$
is a bijective function $f : V \to V'$ which is such that the set
\begin{eqnarray}
\left\{ \{ f ( v ) , f ( w ) \} \mid
\{ v , w \} \in E \right\}
\end{eqnarray}
is equal to $E'$.  For example, the graph in Figure~\ref{4graphfigure} is
isomorphic to the graph in Figure~\ref{4graphaltfigure} under the map $f \colon \{ 0, 1,
2, 3 \} \to \{ 0, 1, 2, 3 \}$ defined~by
\begin{eqnarray}
f( 0 ) = 1 &\quad & f ( 1 ) = 2 \\
f( 2 ) = 3 &\quad & f ( 3 ) = 0.
\end{eqnarray}

\begin{figure}[!b]
 \centerline{\includegraphics{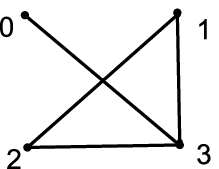}}
\fcaption{A graph that is isomorphic to the graph in
Figure~\ref{4graphfigure}.\label{4graphaltfigure}}
\end{figure}

We can now formalize the notion of a graph property.  Briefly
stated, a graph property is a function on graphs which is
compatible with graph isomorphisms.
Let
$V_0$ be a finite set, and let $\mathbf{G} \left( V_0 \right)$
denote the set of all graphs that have $V_0$ as their vertex set.
Then a function
\begin{eqnarray}
h \colon \mathbf{G} \left( V_0 \right) \to \{ 0, 1 \}
\end{eqnarray}
is a \textbf{graph property} (over $V_0$) if all pairs $(G, G')$ of isomorphic graphs
in $\mathbf{G} \left( V_0 \right)$ satisfy $h ( G ) =
h ( G' )$.

\begin{figure}[!b]
\label{graphsoforder3}
 \centerline{\includegraphics{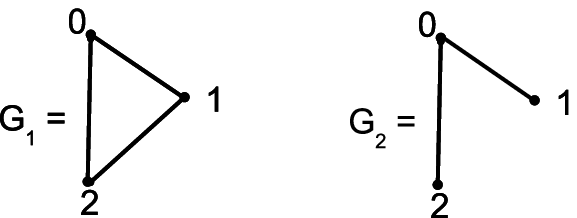}}
\fcaption{Two graphs of size $3$.\label{twographsfigure}}
\end{figure}

For example, consider the graphs in Figure~\ref{twographsfigure}, which are
members of $\mathbf{G} \left( \{ 0, 1, 2 \} \right)$. Then the function
\begin{eqnarray}
h_1 \colon \mathbf{G} \left( \{ 0, 1, 2 \} \right) \to \{ 0, 1 \}
\end{eqnarray}
defined by
\begin{eqnarray}
h_1 ( G ) = \left\{ \begin{array}{@{}l@{\quad}l@{}} 1 & \textnormal{if } G = G_1 \\
0 & \textnormal{if } G \neq G_1
\end{array} \right.
\end{eqnarray}
is a graph property.  However, the function $h_2$ defined by
\begin{eqnarray}
h_2 ( G ) = \left\{ \begin{array}{@{}l@{\quad}l@{}} 1 & \textnormal{ if } G = G_2 \\
0 & \textnormal{ if } G \neq G_2
\end{array} \right.
\end{eqnarray}
is {\it not} a graph property, since there exist graphs in $\mathbf{G} \left( \{ 0, 1, 2
\} \right)$ which are isomorphic to $G_2$ but not equal to $G_2$.

If $G, G' \in \mathbf{G} \left( V_0 \right)$ are graphs such that the edge set of $G'$ is
a subset of the edge set of $G$, then we say that $G'$ is a \textbf{subgraph} of $G$.
Note that this relationship gives us a partial ordering on the set $\mathbf{G} \left( V_0
\right)$. Let us say that a function $h \colon \mathbf{G} \left( V_0 \right) \to \{ 0, 1
\}$ is \textbf{monotone increasing} if it respects this ordering. In other words, $h$ is
monotone increasing if it satisfies $h ( G' ) \leq h ( G )$ for all pairs $(G', G)$ such
that $G'$ is a subgraph of $G$. Likewise, we say that the function $h$ is
\textbf{monotone decreasing} if it satisfies $h ( G' ) \geq h ( G)$ whenever $G'$ is a
subgraph of $G$.

If $h \colon \mathbf{G} \left( V_0 \right) \to \{ 0, 1 \}$ is a function, then a
\textbf{decision tree} for $h$ is a step-by-step procedure for computing the
value of $h$.  An example is the decision tree in Figure~\ref{decisiontreeexample}, which
computes the value of the function $h_2$ defined above. The diagram in
Figure~\ref{decisiontreeexample} describes an algorithm for computing $h_2$.  Each node
in the tree specifies an ``edge-query'', and each branch in the tree specifies how the
algorithm responds to the results of the edge query. For example, suppose that we wish to
apply the algorithm to compute the value of $h_2$ on the graph $G_1$ (from
(\ref{graphsoforder3}), above).  The algorithm would first query the edge $\{ 0, 1 \}$,
and it would find that this edge \textit{is} contained in $G_1$.  It would then follow
the ``Y'' branch from $\{ 0, 1 \}$, and query the edge $\{ 1, 2 \}$.  It would then
follow the ``Y'' branch from $\{ 1, 2 \}$, and determine that the value of $h_2$ is zero.

\begin{figure}[!t]
 \centerline{\includegraphics{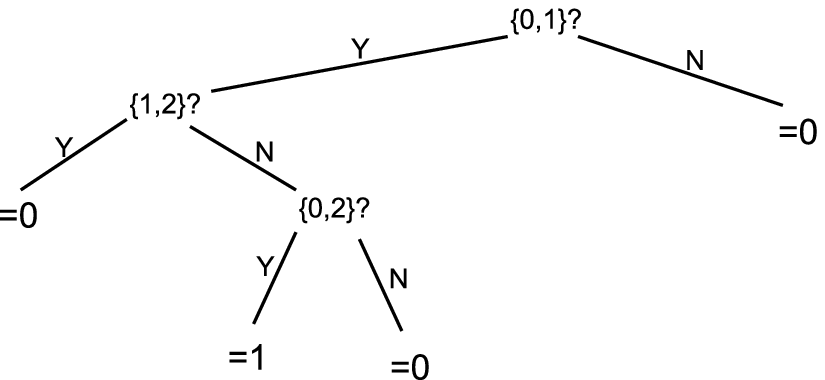}}
\fcaption{A decision tree for the graph property $h_2$.\label{decisiontreeexample}}
\end{figure}

The \textbf{decision-tree complexity} of a function $h \colon \mathbf{G} \left( V_0
\right) \to \{ 0, 1 \}$ is the smallest possible depth for a decision-tree which
correctly computes~$h$.  We denote this quantity by $D(h)$. For example, the depth of the
decision-tree in Figure~\ref{decisiontreeexample} is $3$.  It can be shown that any
decision-tree that computes $h_2$ must have depth at least $3$. Therefore, $D ( h_2 ) =
3$.

It is easy to prove that for any function $h \colon \mathbf{G} \left( V_0 \right) \to \{
0, 1 \}$, the inequality
\begin{eqnarray}
D ( h ) \leq \binom{ \left| V_0 \right|}{2}
\end{eqnarray}
is satisfied.  If the function $h$ satisfies
\begin{eqnarray}
D( h ) = \binom{ \left| V_0 \right|}{2}\!,
\end{eqnarray}
then we will say that the function $h$ is {\bf evasive}. Evasive functions are the
functions that are the most difficult to compute via a decision-tree.\footnote{The
concepts of ``decision-tree complexity'' and ``evasiveness'' can be defined for any
Boolean function.  See Chapter~\ref{resultschapter} of \cite{duandko} for a more detailed
treatment.}

\section{Simplicial Complexes}\label{simplicialcomplexsection}

Now we give a brief introduction to the notion of a simplicial complex.  We draw
on~\cite{munkres} for definitions and terminology.

There are at least two natural ways of defining simplicial complexes---one is as a
collection of finite sets, and another is as a collection of subsets of $\mathbb{R}^n$.
The first definition is the easiest to work with (and it will be the one we use the most
in this monograph). But the second definition is also important because
it provides some indispensible geometric intuition.  We will begin by building up the
second definition.

\begin{definition}
Let $N$ and $n$ be positive integers, with $n \leq N$.  Let $\mf{v}_0, \mf{v}_1, \ldots,
\mf{v}_n \in \mathbb{R}^N$ be vectors satisfying the condition that
\begin{eqnarray}
\{ ( \mf{v}_1 - \mf{v}_0 ) , ( \mf{v}_2 - \mf{v}_0 ), ( \mf{v}_3 - \mf{v}_0 ),
\ldots , ( \mf{v}_n - \mf{v}_0 ) \}
\end{eqnarray}
is linearly independent set.  Then the \textbf{$n$-simplex spanned by $\{ \mf{v}_0,
\mf{v}_1, \ldots, \mf{v}_n \}$} is the set
\begin{eqnarray}
\left\{ \sum_{i=0}^n c_i \mf{v}_i \mid \textnormal{ $0 \leq c_i \leq 1$ for all $i$, and
$\sum_{i=0}^n c_i = 1$ } \right\}\!.
\end{eqnarray}
\end{definition}

When we refer to an ``$n$-simplex'', we simply mean a set which can be defined in the
above form.  Note that a $1$-simplex is simply a line segment.  A $2$-simplex is a solid
triangle, and a $3$-simplex is a solid tetrahedron.

\begin{definition}
Let $N$ and $n$ be positive integers.  Let $v_0, \ldots, v_n \in \mb{R}^N$ be vectors
which span an $n$-simplex~$V$. Then the \textbf{faces} of~$V$ are the simplices in
$\mb{R}^N$ that are spanned by nonempty subsets of $\{ v_0, v_1, \ldots , v_n \}$.
\end{definition}

So, for example, the $2$-simplex in $\mathbb{R}^3$ shown in Figure~\ref{2simplexfigure}
has seven faces (including itself): three of dimension zero,
three of dimension~$1$, and one of dimension two.  In
general, an $n$-simplex has $\binom{n+1}{k+1}$ $k$-dimensional faces.

\begin{figure}[!b]
 \centerline{\includegraphics{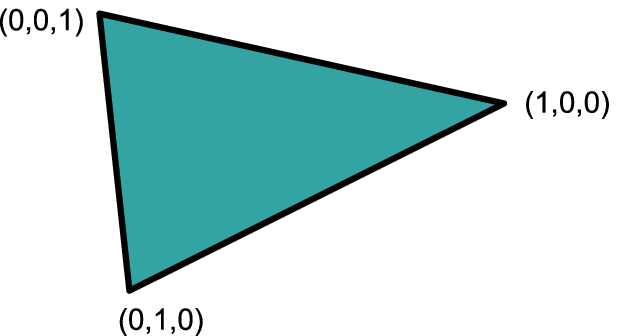}}
\fcaption{A $2$-simplex.\label{2simplexfigure}}
\end{figure}

\begin{definition}
Let $N$ be a positive integer.  A \textbf{simplicial complex} in $\mathbb{R}^N$ is a set
$S$ of simplicies in $\mathbb{R}^N$ which satisfies the following two conditions.
\begin{enumerate}
\item If $V$ is a simplex that is contained in $S$, then all faces of $V$ are also
contained in $S$.
\item If $V$ and $W$ are simplicies in $S$ such that $V \cap W \neq \emptyset$,
then $V \cap W$ is a face of both $V$ and $W$.
\end{enumerate}
\end{definition}

An example of a simplicial complex in $\mathbb{R}^2$ is shown in
Figure~\ref{2dimcomplexfig}.

\begin{figure}[!t]
 \centerline{\includegraphics{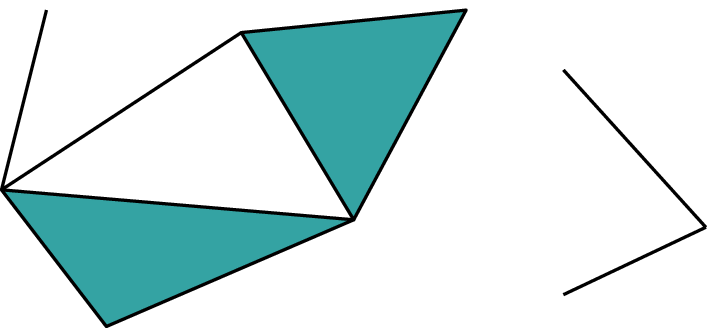}}
\fcaption{A simplicial complex in $\mathbb{R}^2$.\label{2dimcomplexfig}}
\end{figure}

Now, as mentioned earlier, there is another definition of simplicial complexes which
simply describes them as collections of finite sets.  Following \cite{munkres}, we will
use the term ``abstract simplicial complex'' to distinguish this definition from the
previous one.

\begin{definition}
An \textbf{abstract simplicial complex} is a set $\Delta$ of finite
nonempty sets which satisfies the following condition:
\begin{itemize}
\item If a set $Q$ is an element of $\Delta$, then all nonempty subsets
of $Q$ must also be elements of $\Delta$.
\end{itemize}
\end{definition}

Given a simplicial complex $S$ in $\mathbb{R}^N$, one can obtain an abstract simplicial
complex as follows.  Let $T$ be the set of all points in $\mathbb{R}^N$ which occur as
$0$-simplicies in $S$.  Let $\Delta_S$ be the set of all subsets $T' \subseteq T$ which
span simplicies that are in $S$.  Then, $\Delta_S$ is an abstract simplicial complex. (In
a sense, $\Delta_S$ records the ``gluing information'' for the simplicial complex~$S$.)

It is also easy to perform a reverse construction. Suppose that $\Delta$ is an abstract
simplicial complex.  Let
\begin{eqnarray}
U = \bigcup_{Q \in \Delta} Q
\end{eqnarray}
be the union of all of the sets that are contained in $\Delta$.  Let $N = \left| U
\right|$. Simply choose a set $V \subseteq \mathbb{R}^N$ consisting of $N$ linearly
independent vectors, and choose a one-to-one map $r \colon U \to V$.  Every set in
$\Delta$ determines a simplex in $\mathbb{R}^N$ (via $r$), and the collection of all of
these simplicies is a simplicial complex.

We define some terminology for abstract simplicial complexes.
\begin{definition}
Let $\Delta$ be an abstract simplicial complex.  Then,
\begin{itemize}
\item A \textbf{simplex in $\Delta$} is simply an element
of $\Delta$.  The \textbf{dimension} of a simplex $Q \in \Delta$, denoted $\dim ( Q )$,
is the \hbox{quantity} \hbox{$(\left| Q \right| - 1)$}. An \textbf{$n$-simplex} in
$\Delta$ is an element of $\Delta$ of \hbox{dimension}~$n$.
\item If $Q, Q' \in \Delta$
and $Q' \subseteq Q$, then we say that $Q'$ is a \textbf{face} of $Q$.
\item The \textbf{vertex set of $\Delta$} is the set
\begin{eqnarray}
\bigcup_{Q \in \Delta} Q.
\end{eqnarray}
Elements of this set are called \textbf{vertices of $\Delta$}.
\end{itemize}
\end{definition}

Here is an initial example of how abstract simplicial complexes arise.  Let $F$ be a
finite set.  Let $\mathcal{P} ( F )$ denote the power set of $F$.  Let $t \colon
\mathcal{P} ( F ) \to \{ 0, 1 \}$ be a function which is ``monotone increasing,'' in the
sense that any pair of sets $(A, B)$ such that $A \subseteq B \subseteq F$ satisfies $t (
A ) \leq t ( B) $.  Then, the set
\begin{eqnarray}
\left\{ C \mid \emptyset \subset C \subseteq F \textnormal{ and } t ( C ) = 0 \right\}
\end{eqnarray}
is an abstract simplicial complex.

Thus, a monotone increasing function on a power set determines an abstract simplicial
complex.  This connection is the basis for what we will discuss next.

\section{Monotone Graph Properties}\label{graphpropertysimplicial}

Now we will establish a relationship between monotone graph properties and simplicial
complexes.  We also introduce a topological concept (``collapsibility'') which has an
important role in this relationship.

Let $V_0$ be a finite set.  Using notation from \textit{\nameref{graphpropsection}}, let
$\mf{G} ( V_0 )$ denote the set of all graphs that have vertex set $V_0$. The elements of
$\mf{G} ( V_0 )$ are thus pairs of the form $(V_0 , E)$, where $E$ can be any subset of
the set
\begin{eqnarray}\label{thesetofalledges}
\left\{ \{ v, w \} \mid v , w \in V_0 \right\}.
\end{eqnarray}

Let $h \colon \mf{G}  ( V_0 ) \to \{ 0, 1 \}$ be a monotone increasing function.  Then
the \textbf{abstract simplicial complex associated with $h$}, denoted
$\Delta_h$, is the set of all nonempty subsets $E$ of set (\ref{thesetofalledges}) such
that
\begin{eqnarray}
h \left( ( V_0, E ) \right) = 0.
\end{eqnarray}

\begin{example}
Consider the set $\mathbf{G} \left( \{ 0, 1, 2, 3 \} \right)$ of graphs on the vertex set
$\{0, 1, 2, 3 \}$.  Define functions
\begin{eqnarray}
h_1 \colon \mathbf{G} \left( \{ 0, 1, 2, 3 \} \right) \to \{ 0, 1 \}, \\
h_2 \colon \mathbf{G} \left( \{ 0, 1, 2, 3 \} \right) \to \{ 0, 1 \}
\end{eqnarray}
by
\begin{eqnarray}
h_1 (G) & = & \left\{ \begin{array}{@{}l@{\quad}l@{}}
1 & \textnormal{if $G$ has at least three edges,} \\
0 & \textnormal{otherwise} \end{array}
\right.
\end{eqnarray}
and
\begin{eqnarray}
h_2 (G) & = & \left\{ \begin{array}{@{}l@{\quad}l@{}}
1 & \textnormal{if vertex ``$2$'' has at least one edge in $G$, } \\
0 & \textnormal{otherwise.} \end{array}
\right.
\end{eqnarray}
Then the simplicial complexes for $h_1$ and $h_2$ are shown in
Figures~\ref{graphprop1fig} and \ref{graphprop2fig}.\footnote{Note: Ignore the apparent
intersections in the interior of the diagram for $h_1$.  Imagine that the lines in the
diagram only intersect at the labeled points $\{ 0, 1 \}, \{ 0, 2 \} , \{ 1, 2 \} , \{ 0,
3 \} , \{ 2 , 3\}$, and $\{ 1, 3\}$.  (To really draw this diagram accurately, we would
need three dimensions.)}
\end{example}

\begin{figure}
 \centerline{\includegraphics{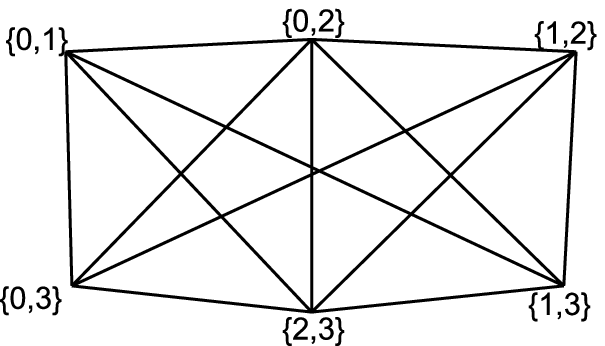}}
\fcaption{The simplicial complex of $h_1$.\label{graphprop1fig}}
\end{figure}

\begin{figure}
 \centerline{\includegraphics{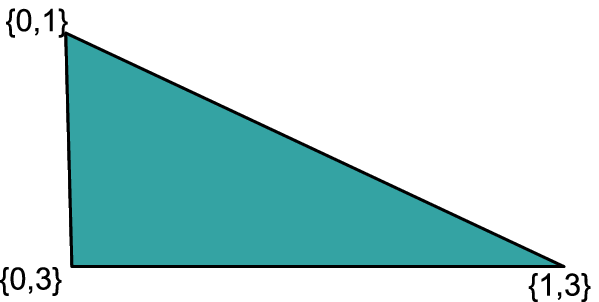}}
\fcaption{The simplicial complex of $h_2$.\label{graphprop2fig}}
\end{figure}

Thus we have a way of associating with any monotone-increasing graph
function
\begin{eqnarray}
h \colon \mathbf{G} ( V_0 ) \to \{ 0, 1 \}
\end{eqnarray}
an abstract simplicial complex $\Delta_h$.  The simplices of $\Delta_h$ correspond to
graphs on $V_0$.  The vertices of $\Delta_h$ correspond to \textit{edges} (not vertices!)
of graphs on $V_0$.

The association $[ h \mapsto \Delta_h ]$ is useful because it allows us to reinterpret
statements about graph functions in terms of simplicial complexes.  What we will do now
is to prove a theorem (for later use) which exploits this association.  The theorem
relates a condition on graph functions (``evasiveness,'' from
\textit{\nameref{graphpropsection}}) to a condition on simplicial complexes
(``collapsibility'').

We begin with some definitions.

\begin{definition}
Let $\Delta$ be an abstract simplicial complex, and let \hbox{$\alpha \in \Delta$} be a
simplex. Then $\Delta$ is a \textbf{maximal} simplex if it is not contained in any other
simplex in $\Delta$.
\end{definition}

\begin{definition}
Let $\Delta$ be an abstract simplicial complex, and let \hbox{$\beta \in \Delta$} be a
simplex. Then $\beta$ is called a \textbf{free face} of $\Delta$ if it is
\hbox{nonmaximal} and it is contained in only one maximal simplex in $\Delta$. If $\beta$
is a free face and $\alpha$ is the unique maximal simplex that contains it, then we will
say that \textbf{$\beta$ is a free face of $\alpha$}.
\end{definition}

\begin{definition}
An \textbf{elementary collapse} of an abstract simplicial complex is the operation of
choosing a single free face from the complex and deleting the face along with all the
faces that contain it.
\end{definition}

Here is an example of an elementary collapse: if
\begin{eqnarray}
\Sigma_1 = \left\{ \{ 0 \}, \{ 1 \} , \{ 2 \} , \{ 0, 1 \} , \{ 0, 2 \} , \{ 1, 2 \} , \{
0, 1, 2 \} \right\}\!,
\end{eqnarray}
then $\{0, 1\}$ is a free face of $\{0, 1, 2\}$ in $\Delta$.  By deleting the simplicies
$\{0,1\}$ and $\{0, 1, 2\}$, we obtain the complex
\begin{eqnarray}
\Sigma_2 = \left\{ \{ 0 \}, \{ 1 \} , \{ 2 \} , \{ 0, 2 \} , \{ 1, 2 \} \right\}\!.
\end{eqnarray}
The complex $\Sigma_2$ is an elementary collapse of the complex $\Sigma_1$. See
\hbox{Figure}~\ref{elementarycollapsefig}.

\begin{figure}[!b]
 \centerline{\includegraphics{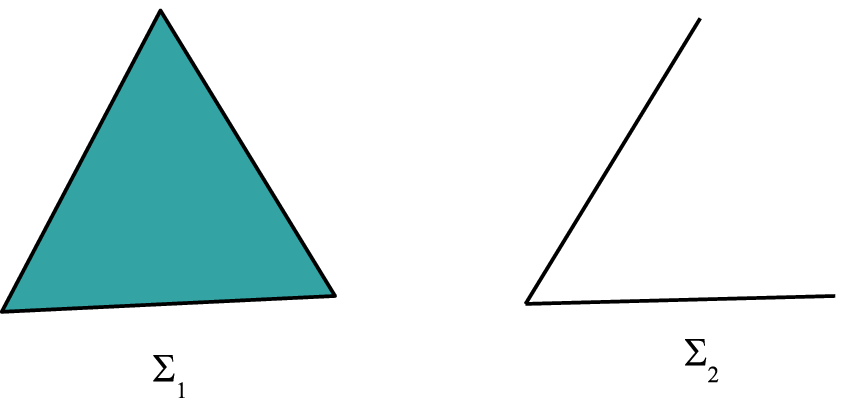}}
\fcaption{An elementary collapse.\label{elementarycollapsefig}}
\end{figure}

The previous example is an instance of what we will call a \textbf{primitive} elementary
collapse.  An elementary collapse is primitive if the free face that is deleted has
dimension one less than the maximal simplex in which it is contained.  In such a case,
the maximal simplex and free face itself are the only two simplices that are deleted.
(Not all elementary collapses are primitive.  An example of a nonprimitive elementary
collapse would be deleting all of the simplices $\{0\}$, $\{0,1\}$, $\{0,2\}$, and $\{0,
1, 2\}$ from $\Sigma_1$.)

\begin{definition}
Let $\Delta$ be an abstract simplicial complex.  Then $\Delta$ is
\textbf{collapsible} if there exists a sequence of elementary collapses
\begin{eqnarray}
\Delta , \Delta_1, \Delta_2 , \Delta_3 , \ldots, \Delta_n
\end{eqnarray}
such that $\left| \Delta_n \right| = 1$.
\end{definition}

In other words, $\Delta$ is collapsible if there exists a sequence of elementary
collapses which reduce $\Delta$ to a single $0$-simplex.

The abstract simplicial complexes $\Sigma_1$ and $\Sigma_2$ defined above are both
collapsible. An example of an abstract simplicial complex that is not collapsible is the
following:
\begin{eqnarray}
\Sigma = \left\{ \{ 0 \}, \{ 1 \} , \{ 2 \} , \{ 0, 1 \} , \{ 0, 2 \} , \{ 1, 2 \}
\right\}\!.
\end{eqnarray}
(This simplicial complex has no free faces, and therefore cannot be collapsed.)

The following theorem asserts that the simplicial complexes associated with
certain monotone-increasing graph functions are collapsible.  The theorem uses the
concept of ``evasiveness'' from \textit{\nameref{graphpropsection}}.

\begin{theorem}\label{collapsibilitytheorem}
Let $V_0$ be a finite set.  Let
\begin{eqnarray}
h \colon \mathbf{G} \left( V_0 \right) \to \{ 0 , 1 \}
\end{eqnarray}
be a monotone-increasing function which is not evasive.  If the complex $\Delta_h$ is not
empty, then it is collapsible.
\end{theorem}

\begin{proof}
The theorem has an elegant visual proof.  Essentially, what we do is to construct a
decision-tree for $h$ and then read off a collapsing-procedure for $\Delta_h$ from the
decision-tree.\footnote{Thanks to Yaoyun Shi, who suggested the nice visualization that
appears in this proof.}

\begin{figure}
 \centerline{\includegraphics{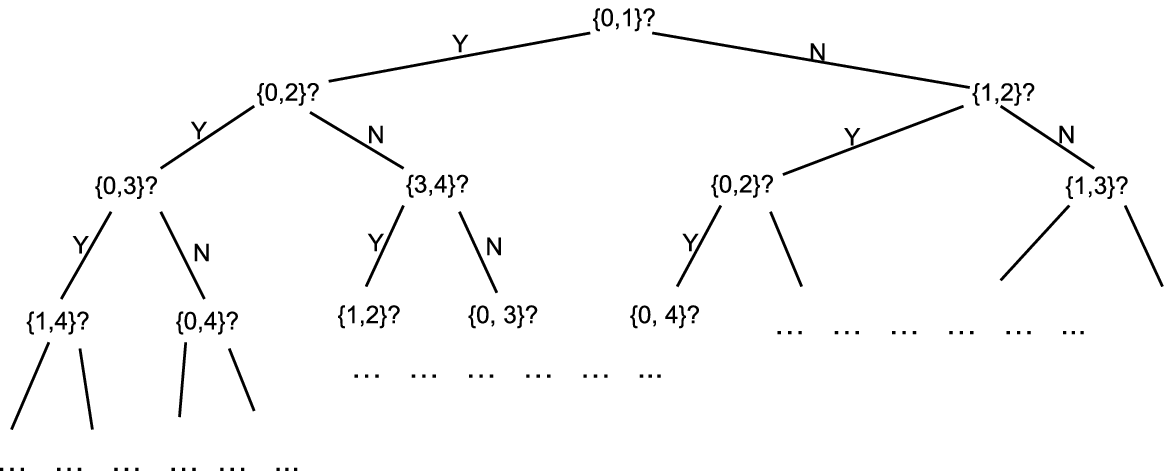}}
\fcaption{A decision tree.\label{adecisiontreefig}}
\end{figure}

Let $n = |V_0|$. Since we have assumed that the function $h$ is not evasive, there must
exist a decision tree of depth smaller than \hbox{$n(n-1)/2$} which decides $h$.  Let $T$
be such a tree.  (See Figure~\ref{adecisiontreefig}.) By modifying $T$ if necessary, we
can produce another decision-tree $T'$ which decides $h$ and which satisfies the
following conditions. (See Figure~\ref{decisiontreeTprimefig}.)
\begin{itemize}
\item The paths in $T'$ do not have repeated edges.  (That is,
no edge $\{ i , j \}$ appears more than once on any path in $T'$.)
\item Every path in $T'$ has length exactly $[n(n-1)/2 - 1]$.
\end{itemize}

\begin{figure}
 \centerline{\includegraphics{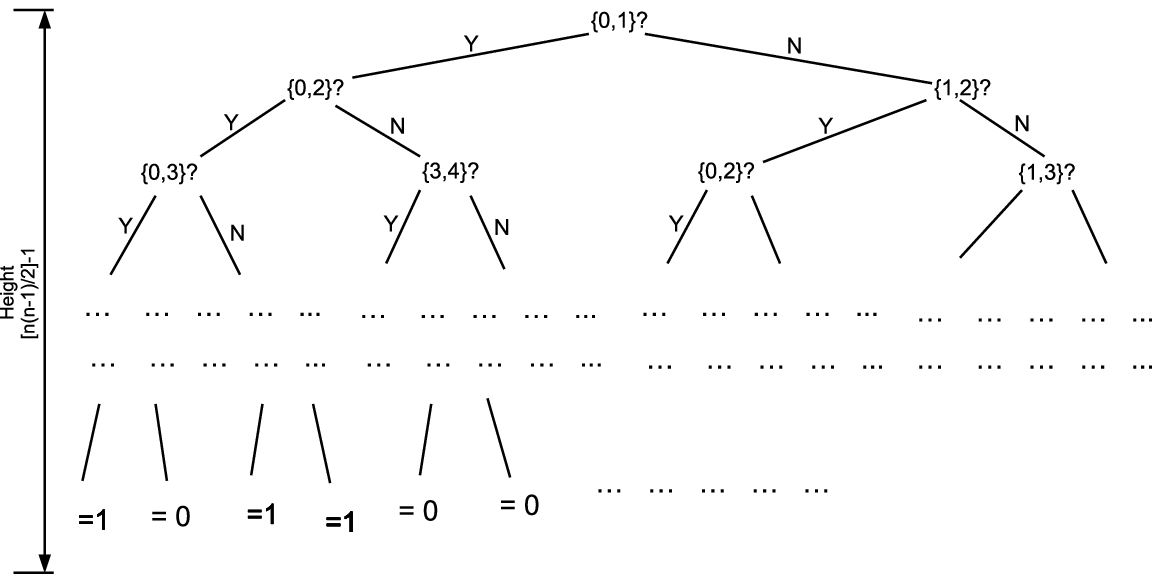}}
\fcaption{A decision tree of uniform height.\label{decisiontreeTprimefig}}
\end{figure}

We can define a natural total ordering on the leaves of tree $T'$. The ordering is
defined by asserting that for any parent-node in the tree, all leaves that can be reached
through the ``Y'' branch of the node are smaller than all the leaves that can be reached
through the ``N'' branch of the node.  Since any two leaves share a common ancestor, this
rule gives a total ordering.

For any leaf of tree $T'$, there are exactly two graphs which would cause
the leaf to be reached during computation.  Thus there is a
one-to-two correspondence between leaves of $T'$ and graphs
on $V_0$.  An example is shown in Figure~\ref{decisiontreedepth2fig}. Note that each leaf
is labeled with either with a ``$1$'' or a ``$0$'', depending on the value taken by the
function $h$ at the corresponding graphs. The simplicial complex $\Delta_h$ is composed
out of the graphs that appear at the ``$0$''-leaves of the tree.

\begin{figure}[!t]
 \centerline{\includegraphics{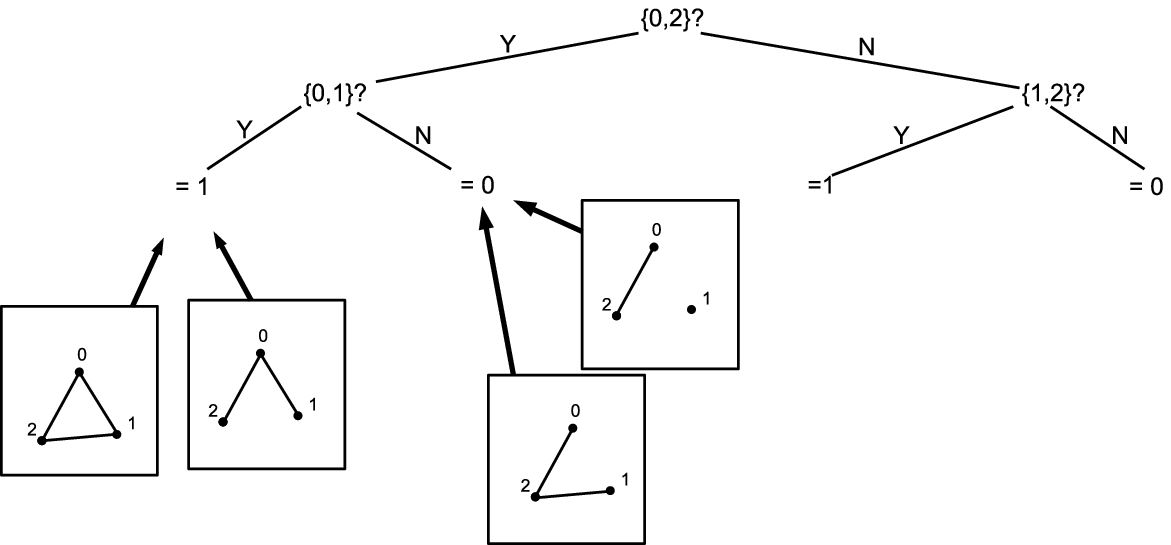}}
\fcaption{A decision tree of height $2$ for graphs of size
$3$.\label{decisiontreedepth2fig}}
\end{figure}

The ordering of the leaves of $T'$ provides a recipe for collapsing $\Delta_h$.  Simply
find the smallest (i.e., leftmost) ``$0$''-leaf that appears in tree $T'$.  This leaf
corresponds to a pair of simplices $\gamma_1, \gamma_2 \in \Delta_h$ with $\gamma_1
\subseteq \gamma_2$.  From the ordering of the leaves, we can deduce that $\gamma_1$ and
$\gamma_2$ are not contained in any simplices in $\Delta_h$ other than themselves.  Thus
$\gamma_1$ is a free face of $\Delta_h$.  We can therefore perform an elementary
collapse: let
\begin{eqnarray}
\Delta_1 = \Delta_h \smallsetminus \{ \gamma_1 , \gamma_2 \}.
\end{eqnarray}
Now find the second smallest $0$-leaf that appears in $T'$.  This leaf corresponds to
another pair of simplices $\gamma'_1, \gamma'_2 \in \Delta_h$ which are not contained in
any other simplices in $\Delta_h$, except possibly $\gamma_1$ or $\gamma_2$.  Perform
another elementary collapse:
\begin{eqnarray}
\Delta_2 = \Delta_1 \smallsetminus \{ \gamma'_1 , \gamma'_2 \}.
\end{eqnarray}

Continuing in this manner, we can obtain a sequence of elementary collapses
\begin{eqnarray}
\Delta_h, \Delta_1 , \Delta_2 , \Delta_3 , \ldots , \Delta_n
\end{eqnarray}
such that $\left| \Delta_n \right| = 1$.  Therefore, $\Delta_h$ is collapsible.
\end{proof}

\section{Group Actions on Simplicial Complexes}\label{groupactionsection}

Now we define the notion of a \textbf{simplicial isomorphism} between abstract simplicial
complexes.  This is a case of the more general notion of a simplicial map (see
\cite{munkres}).

\begin{definition}
Let $\Delta$ and $\Delta'$ be abstract simplicial complexes.  A simplicial isomorphism
from $\Delta$ to $\Delta'$ is a bijective map
\begin{eqnarray}
f \colon \Delta \to \Delta'
\end{eqnarray}
which is such that for any $Q_1, Q_2 \in \Delta$,
\begin{eqnarray}
Q_1 \subseteq Q_2 \hskip0.2in \Longleftrightarrow \hskip0.2in
f ( Q_1 ) \subseteq f (Q_2 ).
\end{eqnarray}
\end{definition}

In other words, a simplicial isomorphism between two abstract complexes $\Delta$,
$\Delta'$ is a one-to-one matching $f$ between the simplicies of~$\Delta$ and~$\Delta'$
which respects inclusion.  We note the following assertions, which can be proven easily
from this definition:
\begin{itemize}
\item If $f \colon \Delta \to \Delta'$ is a simplicial isomorphism, then
$f$ respects dimension (i.e., if $Q \in \Delta$ is an $n$-simplex, then
$f(Q)$ must be an $n$-simplex).
\item If $f \colon \Delta \to \Delta'$ is a simplicial isomorphism, then there is an associated map of
vertex sets
\begin{eqnarray}
\hat{f} \colon \bigcup_{Q \in \Delta} Q \to
\bigcup_{Q' \in \Delta'} Q'
\end{eqnarray}
defined by $f ( \{ v \} ) = \{ \hat{f} ( v ) \}$.  (Let us call this the \textbf{vertex
map} of $f$.)  The map $\hat{f}$ uniquely determines $f$.
\end{itemize}

Let $\Delta$ be an abstract simplicial complex. A simplicial automorphism of $\Delta$ can
be specified either as an inclusion preserving permutation of the elements of $\Delta$,
or simply as a permutation
\begin{eqnarray}
b \colon \bigcup_{Q \in \Delta} Q \to  \bigcup_{Q \in \Delta} Q
\end{eqnarray}
of the vertex set of $\Delta$ satisfying
\begin{eqnarray}
Q \in \Delta \Longrightarrow b ( Q ) \in \Delta.
\end{eqnarray}
When we speak of a \textbf{group action} $G \circlearrowleft \Delta$, we mean an action
of a group~$G$ on $\Delta$ by simplicial automorphisms.

In \textit{\nameref{fptchapter}} we will be concerned with determining the ``fixed
points'' of a group action on an abstract simplicial complex. As we will see, describing
this set requires some care. One could simply take the set $\Delta^G$ of $G$-invariant
simplices.  But this set is not always subcomplex of $\Delta$.  Consider the
two-dimensional complex $\Sigma$ in Figure~\ref{groupactionfig}, which
consists of the sets $\{ 0, 1, 2 \}$ and $\{ 0, 2, 3 \}$ and all of their proper nonempty
subsets. If we let $f \colon \Sigma \to \Sigma$ be the simplicial automorphism which
transposes $\{ 1 \}$ and $\{3 \}$ and leaves $\{ 0 \}$ and $\{ 2 \}$ fixed, then
$\Sigma^f$ is a subcomplex of $\Sigma$.  However, if we let $h \colon \Sigma \to \Sigma$
be the simplicial automorphism which transposes $\{ 0 \}$ and $\{ 2 \}$ and leaves $\{ 1
\}$ and $\{ 3 \}$ fixed, then $\Delta^h$ is not a subcomplex of $\Sigma$, since it
contains the set $\{ 0, 2 \}$ but does not contain its subsets $\{ 0 \}$ and $\{ 2 \}$.

\begin{figure}[!b]
 \centerline{\includegraphics[scale=1.03]{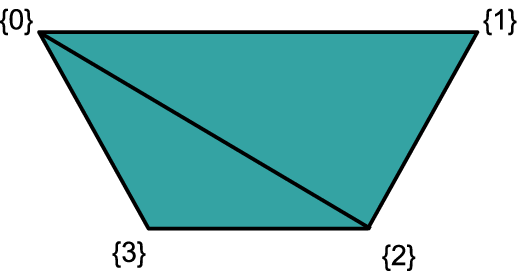}}
\fcaption{The complex $\Sigma$.\label{groupactionfig}}
\end{figure}

It is helpful to look at group actions on abstract simplicial \hbox{complexes} in terms
of the geometric representation introduced in
\textit{\nameref{simplicialcomplexsection}}. Let $\mf{e}_0, \mf{e}_1, \ldots, \mf{e}_n$
be the standard basis vectors in $\mathbb{R}^{n+1}$.  These vectors span an $n$-simplex
\begin{eqnarray}
\delta = \left\{ \sum_{i=0}^n c_i \mf{v}_i \mid 0 \leq c_i \leq 1 , \sum_{i=0}^n c_i = 1
\right\}\!.
\end{eqnarray}
If $f \colon \{ 0, 1, \ldots, n \} \to \{ 0, 1, \ldots, n \}$ is a permutation with
orbits $B_1, \ldots, B_m \subseteq \{ 0, 1, \ldots, n \}$, then $f$ induces a bijective
map on $\delta$.  The invariant set $\delta^f$ consists of those linear combinations
$\sum c_i \mf{v}_i$ satisfying the condition that $c_i = c_j$ whenever $i$ and $j$ lie in
the same orbit.  The set $\delta^f$ is an $(m-1)$-simplex which is spanned by the vectors
\begin{eqnarray}
\left\{ \frac{ \sum_{i \in B_k } \mf{v}_i }{\left| B_k \right|} \mid k = 1, 2, \ldots, m
\right\}\!.
\end{eqnarray}
This motivates the following definition.

\enlargethispage{4pt}

\begin{definition}
Let $\Delta$ be a finite abstract simplicial complex with vertex set $V$, and let $G
\circlearrowleft \Delta$ be a group action.  Let $A_1, \ldots, A_m \subseteq V$ denote
the orbits of the action of $G$ on $V$.  Then, let $\Delta^{[G]}$ denote the set of all
subsets $T \subseteq \{ A_1, \ldots , A_m \}$ satisfying
\begin{eqnarray}
\bigcup_{S \in T} S \in \Delta.
\end{eqnarray}
\end{definition}

\noindent It is easy to see that the set $\Delta^{[G]}$ is always a simplicial complex.
In the case of the complex $\Sigma$ from Figure~\ref{groupactionfig}, if we let $H$ be
the group generated by the automorphism $h$ which transposes $\{ 0 \}$ and $\{ 2 \}$, the
complex $\Sigma^{[H]}$ is one-dimensional and consists of
three zero simplices and two one-simplices.  (See
Figure~\ref{groupaction2fig}.) The vertices of $\Sigma^{[H]}$ are the orbits $\{ 1 \}$,
$\{ 3 \}$, and $\{ 0, 2 \}$.

\begin{figure}[!b]
 \centerline{\includegraphics{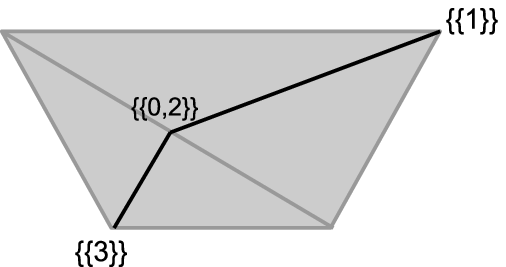}}
\fcaption{The complex $\Sigma^{[H]}$.\label{groupaction2fig}}
 \vspace*{-3pt}
\end{figure}

This complex $\Delta^{[G]}$ will be important in \textit{\nameref{fptchapter}}.

\chapter{Chain Complexes}\label{chaincomplexchapter}

 \vspace*{-12pt}

In this part of the text we will introduce some algebraic objects which are crucial for
measuring the behavior of simplicial complexes.  The \hbox{central} objects of concern
are \textbf{chain complexes} and \textbf{homology groups}. We will define these objects
and develop some important tools for dealing with them.

\section{Definition of Chain Complexes}\label{chaincomplexsection}

A \textbf{complex of abelian groups} is a sequence of abelian groups
\begin{eqnarray}
Z_0 , Z_1, Z_2, \ldots
\end{eqnarray}
together with group homomorphisms $d_i \colon Z_i \to Z_{i-1}$ for each $i > 0$,
\hbox{satisfying} the condition
\begin{eqnarray}
d_{i-1} \circ d_i = 0
\end{eqnarray}
(or equivalently, $\im d_i \subseteq \ker d_{i-1}$). The groups $Z_i$ and the maps $d_i$
are often expressed in a diagram like so:
\begin{eqnarray}
\xymatrix{\cdots \ar[r] & Z_3 \ar[r]^{d_3}
& Z_2 \ar[r]^{d_2}
& Z_1 \ar[r]^{d_1}
& Z_0}
\end{eqnarray}
We abbreviate the complex as $Z_\bullet$.

A chain complex is a particular complex of abelian groups that is obtained from a
simplicial complex. The definition of chain complex that we will use requires first
choosing a total ordering of the vertices of the abstract simplicial complex in question.
If the vertices of the abstract simplicial complex happen to be elements of a totally
ordered set (such as the set of integers), then our choice is already made for us.
Otherwise, it is necessary before applying our definition to specify what ordering of
vertices we are using.  The particular choice of ordering is not terribly important, but
it must be made consistently.

We introduce some new notation which takes this ordering issue into account.

\begin{notation}
Let $V$ be a totally ordered set, and let $\Delta$ be an abstract simplicial complex
whose vertices are all elements of $V$.  For any sequence of distinct elements $v_0, v_1,
\ldots, v_n \in V$ such that
\begin{eqnarray}
\left\{ v_0, \ldots , v_n \right\} \in \Delta
\end{eqnarray}
and
\begin{eqnarray}
v_0 < v_1 < v_2 < \ldots < v_n,
\end{eqnarray}
let
\begin{eqnarray}
[v_0, v_1, \ldots , v_n]
\end{eqnarray}
denote the $n$-simplex $\left\{ v_0, \ldots , v_n \right\}$ in $\Delta$.
\end{notation}

\noindent This notation allows us to cleanly handle the ordering on the vertices of an
abstract simplicial complex.  Note that if we say, ``$[ v_0, v_1, \ldots , v_n ]$ is a
simplex in $\Delta$'', we are implying both that $\{ v_0 , \ldots , v_n \}$ is an
element of $\Delta$ \textit{and} that the sequence $v_0, v_1, \ldots, v_n$ is in
ascending order.

Now we will define the sequence of groups which make up a chain complex.

\begin{definition}
Let $V$ be a totally ordered set, and let $\Delta$ be an abstract simplicial complex
whose vertices are elements of $V$.  Let $n$ be a nonnegative integer. Then, the
\textbf{$n$th chain group of $V$ over $\mathbb{R}$}, denoted $K_n ( \Delta, \mathbb{R}
)$, is the set of all formal $\mathbb{R}$-linear combinations of $n$-simplices in
$\Delta$.
\end{definition}

\begin{example}\label{triangleexample}
Let $\Sigma$ be the simplicial complex
\begin{eqnarray}
\Sigma = \left\{ \{ 0 \}, \{ 1 \} , \{ 2 \} , \{ 0, 1 \},
\{ 1, 2 \}, \{ 0, 2 \} \right\}.
\end{eqnarray}
Then, $\Sigma$ has three zero-simplices ($[0]$, $[1]$, and $[2]$) and three one-simplices
($[0,1]$, $[1,2]$, and $[0,2]$).  The chain group $K_0 \left( \Sigma, \mathbb{R} \right)$
is a three-dimensional real vector space, and its elements can be expressed
in the form
\begin{eqnarray}
r_1  [ 0 ] + r_2 [ 1 ]  + r_3 [2],
\end{eqnarray}
where $r_1$, $r_2$, and $r_3$ denote real numbers.  The chain group $K_1 \left( \Sigma,
\mathbb{R} \right)$ is a three-dimensional real vector space, and its
elements can be expressed in the form
\begin{eqnarray}
r_4  [0, 1] + r_5  [1, 2] + r_6  [0, 2 ],
\end{eqnarray}
where $r_4$, $r_5$, and $r_6$ denote real numbers.
\end{example}

In general, if $\Delta$ is an abstract simplicial complex, then $K_n \left( \Delta,
\mathbb{R} \right)$ is a real vector space whose dimension is equal to the number
$n$-simplicies in $\Delta$.  (If $\Delta$ has no $n$-simplicies, then $K_n \left( \Delta
, \mathbb{R} \right)$ is a zero vector space.)

\begin{definition}\label{boundarymapdef}
Let $V$ be a totally ordered set, and let $\Delta$ be an abstract simplicial complex
whose vertices are elements of $V$.  Let $n$ be a positive integer.  Then the
\textbf{boundary map} on the $n$th chain group of $\Delta$ (over $\mathbb{R}$) is the
unique $\mathbb{R}$-linear homomorphism
\begin{eqnarray}
d_n \colon K_n \left( \Delta , \mathbb{R} \right)
\to K_{n-1} \left( \Delta , \mathbb{R} \right)
\end{eqnarray}
defined by the equations
\begin{eqnarray}
\label{boundarymapdefeqn}
d_n \left( [ v_0, v_1, \ldots, v_n ] \right) =
\sum_{i = 0}^n (-1)^i [ v_0, v_1, \ldots, v_{i-1} , v_{i+1} ,
\ldots, v_n ]
\end{eqnarray}
(where $[v_0, v_1, \ldots, v_n]$ can be taken to be any $n$-simplex in $\Delta$.)
\end{definition}

\begin{example}\label{solidtriangleexample}
Let
\begin{eqnarray}
\Sigma' = \left\{ \{ 0 \} , \{ 1 \}, \{ 2 \} , \{ 0, 1 \} , \{ 1, 2 \} , \{ 0, 2 \} , \{
0, 1, 2 \} \right\}\!.
\end{eqnarray}
Then the boundary map
\begin{eqnarray}
d_2 \colon K_2 \left( \Sigma' , \mathbb{R} \right)
\to K_1 \left( \Sigma' , \mathbb{R} \right)
\end{eqnarray}
is defined by the equation
\begin{eqnarray}
d_2 \left( [ 0, 1, 2 ] \right) & = & [1, 2] - [0, 2] + [0, 1].
\end{eqnarray}
The boundary map
\begin{eqnarray}
d_1 \colon K_1 \left( \Sigma' , \mathbb{R} \right)
\to K_0 \left( \Sigma' , \mathbb{R} \right)
\end{eqnarray}
is defined by the equations
\begin{eqnarray}
d_1 \left( [ 0, 1 ] \right) & = & [0] - [1] \\
d_1 \left( [ 0, 2 ] \right) & = & [0] - [2] \\
d_1 \left( [ 1, 2 ] \right) & = & [1] - [2].
\end{eqnarray}
\end{example}


Note that in equation~(\ref{boundarymapdefeqn}), the simplicies that appear on the right
side are precisely the $(n-1)$-simplex faces of the simplex $[v_0, v_1, \ldots, v_n]$.
Geometrically, if $U \subseteq \mathbb{R}^N$ is an $n$-simplex, then the codimension-$1$
faces of $U$ make up the boundary (or exterior) of the set $U$.  This gives us an idea of
why $d_n$ is called a ``boundary'' map.

\begin{proposition}\label{doublezeroprop}
Let $\Delta$ be an abstract simplicial complex whose vertices are totally ordered.  Let
$n$ be an integer such that $n \geq 2$.  Then the map
\begin{eqnarray}
d_{n-1} \circ d_{n} \colon K_n \left( \Delta, \mathbb{R} \right)
\to K_{n-2} \left( \Delta, \mathbb{R} \right)
\end{eqnarray}
is the zero map.
\end{proposition}

\begin{proof}
Let $Q = [v_0, v_1, \ldots , v_n]$ be an $n$-simplex in $\Delta$.  Then,
applying Definition (\ref{boundarymapdef}) twice, we find
\begin{eqnarray*}
&&d_{n-1} \left( d_n \left( Q \right) \right)\\
& &\qquad =  \sum_{i=0}^n d_{n-1} \left( (-1)^i [ v_0, \ldots, v_{i-1} , v_{i+1} ,
\ldots v_n ] \right) \\
&&\qquad =  \sum_{i=0}^n \left( \sum_{j=0}^{i-1} (-1)^{i+j} [v_0, \ldots , v_{j-1} ,
v_{j+1} , \ldots
v_{i-1}, v_{i+1}, \ldots , v_n ] \right. \\
&&\qquad \quad +\! \left. \sum_{j=i+1}^n (-1)^{i+j-1} [v_0, \ldots , v_{i-1}, v_{i+1} ,
\ldots , v_{j-1}, v_{j+1}, \ldots, v_n ]\!\!\right)\!.
\end{eqnarray*}
All terms in this double-summation cancel, and thus we find that
\begin{eqnarray}
d_{n-1} ( d_n ( Q ) ) = 0.
\end{eqnarray}
Therefore by linearity, $d_{n-1} \circ d_n$ is the zero map.
\end{proof}

\noindent If $\Delta$ is an abstract simplicial complex with ordered vertices, then the
\textbf{chain complex of $\Delta$ over $\mathbb{R}$} is the set of $\mathbb{R}$-chain
groups of $\Delta$ together with their boundary maps:
\begin{eqnarray}
\xymatrix{\ldots \ar[r] & K_2 \left( \Delta, \mathbb{R} \right)
\ar[r]^{d_2} & K_1 \left( \Delta , \mathbb{R} \right) \ar[r]^{d_1} &
K_0 \left( \Delta , \mathbb{R} \right) \ar[r]^{d_0} & 0}
\end{eqnarray}

For any $n$, the \textbf{$n$th homology group} of $\Delta$ is defined by
\begin{eqnarray}
H_n \left( \Delta , \mathbb{R} \right) = (\ker d_n)/(\im d_{n+1})\!,
\end{eqnarray}
Consider the complex $\Sigma$ from Example~\ref{triangleexample}. The kernel of $d_0$ is
the entire space $K_0 ( \Delta , \mathbb{R} )$, while the image of $d_1$ is the set of
all linear combinations $r_1 [0 ] + r_2 [1 ] + r_3 [2]$ which are such that \hbox{$r_1 +
r_2 + r_3 = 0$}. The quotient $H_0 ( \Delta , \mathbb{R} ) = \ker d_0 / \im d_1$ is a
one-dimensional real \hbox{vector} space.  The homology group $H_1 ( \Delta , \mathbb{R}
) = \ker d_1  / \{ 0 \}$ is also a one-dimensional real vector space, spanned by the
element $[0, 1] - [0, 2] + [1, 2]$.  All other homology groups of $\Sigma$ are
zero-dimensional.

As we will see in \textit{\nameref{picturingsection}}, the homology groups are
interesting because they supply structural information about the complex $\Delta$.  As an
initial example, the reader is invited to prove the \hbox{following} fact as an
exercise:\vadjust{\pagebreak} for any finite abstract simplicial \hbox{complex}~$\Delta$,
the dimension of $H_0 ( \Delta , \mathbb{R} )$ is equal to the number of connected
components of $\Delta$.

Although we defined chain groups using $\mathbb{R}$ (the set of real numbers), it is
possible to define them using other algebraic structures in place of $\mathbb{R}$. Here
is a definition for chain groups over $\mathbb{F}_p$.  Proposition~\ref{doublezeroprop}
and the definition of homology groups carry over immediately to this case.

\begin{definition}
Let $V$ be a totally ordered set, and let $\Delta$ be an abstract simplicial complex
whose vertices are elements of $V$.  Then $K_n \left( \Delta , \mathbb{F}_p \right)$
denotes the vector space of formal $\mathbb{F}_p$-linear combinations of $n$-simplicies
in $V$. For each $n \geq 1$, the map
\begin{eqnarray}
d_n \colon K_n \left( \Delta , \mathbb{F}_p \right) \to
K_{n-1} \left( \Delta , \mathbb{F}_p \right)
\end{eqnarray}
is the unique $\mathbb{F}_p$-linear map defined by
\begin{eqnarray}
d_n \left( [ v_0, v_1, \ldots, v_n ] \right) = \sum_{i = 0}^n (-1)^i [ v_0, v_1, \ldots,
v_{i-1} , v_{i+1} , \ldots , v_n].\\[-13pt]\nn
\end{eqnarray}
\end{definition}

For the rest of this exposition we will be focusing on homology groups with coefficients
in $\mathbb{F}_p$, since these will eventually be the basis for our proofs of fixed-point
theorems. Much of what we will do in this text with $\mathbb{F}_p$-homology could be done
just as well with $\mathbb{R}$-homology, but there will be a key result
(Proposition~\ref{acyclicitypreservation}) which depends critically on the fact that we
are using coefficients in $\mathbb{F}_p$.

\section{Chain Complexes and Simplicial Isomorphisms}\label{chainmapsection}

Suppose that
\begin{eqnarray}
\xymatrix{\ldots
\ar[r] & I_{n+1} \ar[r]^{d_{n+1}}
& I_n \ar[r]^{d_n} & I_{n-1} \ar[r]^{d_{n-1}} & \ldots}
\end{eqnarray}
and
\begin{eqnarray}
\xymatrix{\ldots
\ar[r] & J_{n+1} \ar[r]^{d_{n+1}}
& J_n \ar[r]^{d_n} & J_{n-1} \ar[r]^{d_{n-1}} & \ldots}
\end{eqnarray}
are two complexes of abelian groups. A \textbf{map of complexes} $F \colon I_\bullet \to
J_\bullet$ is a family of homomorphisms
\begin{eqnarray}
F_n \colon I_n \to J_n
\end{eqnarray}
such that
\begin{eqnarray}
d_n \circ F_n = F_{n-1} \circ d_n.
\end{eqnarray}
Note that, as a consequence of this rule, the map $F_n$ must send the kernel of $d_n^I$
to the kernel of $d_n^J$. Moreover, the family $F$ induces maps on homology groups
\begin{eqnarray}
H_n ( I_\bullet ) \to H_n ( J_\bullet ).
\end{eqnarray}
for every $n$.

Let $p$ be a prime.  We are going to define the maps of chain complexes that are
associated with simplicial isomorphisms.  Some care must be taken in this
definition. Let $f \colon \Delta \to \Delta'$ be a simplicial isomorphism. An obvious way
to map $K_n \left( \Delta , \mathbb{F}_p \right)$ to $K_n \left( \Delta' , \mathbb{F}_p
\right)$ would be to naively apply $f$ like so: $\sum c_i Q_i \mapsto \sum c_i f ( Q_i
)$. However, this definition does not necessarily give a map of complexes, because
it is not necessarily compatible with the maps $d_i$.  The reader will recall that the
definition of $d_i$ depends on the ordering of the vertices of the simplicial complex in
question. The map $f$ may not be compatible with the ordering of the vertices of $\Delta$
and $\Delta'$.  In our definition of the maps $K_n \left( \Delta , \mathbb{F}_p \right)
\to K_n \left( \Delta , \mathbb{F}_p  \right)$, we need to take this ordering issue into
account.

Note that for any bijection $g \colon S_1 \to S_2$ between two totally ordered sets $S_1$
and $S_2$, there is a unique permutation $\alpha \colon S_2 \to S_2$ which makes the
composition $\alpha \circ g$ an order-preserving map.  Let us say that the \textbf{sign}
of the map $g$ is the sign of its associated permutation $\alpha$.\footnote{See
\cite{lang}, pp.~30--31 for a definition of the sign of a permutation.  Briefly: if
$\sigma : X \to X$ is a permutation of a finite set $X$, then we can write $\sigma =
\tau_1 \circ \tau_2 \circ \ldots \circ \tau_m$ for some $m$, where each of the maps
$\tau_i \colon X \to X$ is a permutation which transposes two elements. The sign of
$\sigma$ is $(-1)^m$.}

\begin{definition}\label{chainmapdef}
Suppose that $\Delta$ and $\Delta'$ are abstract simplicial complexes whose vertex
sets are totally ordered.  Suppose that $f \colon \Delta \to \Delta'$ is a simplicial
isomorphism and that $\hat{f}$ is its vertex map. Let $p$ be a prime, and let $n$ be a
nonnegative integer. The \textbf{$n$th chain map associated with $f$} (over
$\mathbb{F}_p$) is the unique $\mathbb{F}_p$-linear map
\begin{eqnarray}
F_n \colon K_n ( \Delta , \mathbb{F}_p ) \to K_n ( \Delta' , \mathbb{F}_p )
\end{eqnarray}
given by
\begin{eqnarray}
Q & \mapsto & \big( \sign (\hat{f}_{\mid Q}) \big)  f ( Q ).
\end{eqnarray}
for all $Q \in \Delta$.  Here, $( \sign ( \hat{f}_{\mid Q} ) )$ denotes the sign of the
bijection $( \hat{f} )_{\mid Q} \colon Q \to f ( Q )$.
\end{definition}

Let $\Sigma'$ be the complex from Example~\ref{solidtriangleexample}, and let $g \colon
\Sigma' \to \Sigma'$ be the automorphism given by the permutation $[0 \mapsto 1, 1
\mapsto 2, 2 \mapsto 0]$.  Then the chain maps $G_n$ associated with $g$ are
as shown below.
\begin{eqnarray*}
\begin{array}{ccccc}
G_0 ( [0] ) = [1] && G_1 ([0,1]) = [1,2] && \\
G_0 ( [1] ) = [2] && G_1 ([1,2]) = -[0,2] && G_2 ( [0, 1, 2] ) = [0, 1, 2] \\
G_0 ( [2] ) = [0] & & G_1 ( [0,2]) = -[0,1] &&
\end{array}
\end{eqnarray*}

\begin{proposition}
The chain maps $F_n$ of Definition~\ref{chainmapdef} determine a map of complexes,
\begin{eqnarray}
F \colon K_\bullet \left( \Delta , \mathbb{F}_p \right) \to K_\bullet \left ( \Delta' ,
\mathbb{F}_p \right)\!.
\end{eqnarray}
\end{proposition}

\begin{proof}
It suffices to show that for any $n > 0$, and any $n$-simplex $Q \in \Delta$,
\begin{eqnarray}
d_n ( F_n ( Q ) ) = F_n ( d_n ( Q ) ).
\end{eqnarray}
Let $n$ be a positive integer, and let $Q \in \Delta$ be an $n$-simplex. Write the
simplices $Q$ and $f(Q)$ as
\begin{eqnarray}
Q = [v_0, v_1, \ldots , v_n ], \qquad f(Q) = [w_0, w_1, \ldots, w_n ].
\end{eqnarray}
(Here, as usual, we assume that the sequences $v_0, \ldots , v_n$ and $w_0, \ldots , w_n$
are in ascending order.) The elements $d_n ( F_n ( Q ) )$ and $F_n ( d_n ( Q ) )$ are
linear combinations of faces of the simplex $[w_0, \ldots , w_n]$.  We need simply to
show that the coefficients in the expressions for $d_n ( F_n ( Q ) )$ and $F_n ( d_n ( Q
) )$ are the same.

Suppose that the face
\begin{eqnarray}
[v_0, v_1, \ldots , v_{i-1} , v_{i+1} , \ldots , v_n]
\end{eqnarray}
of $Q$ maps to the face
\begin{eqnarray}
[w_0, w_1, \ldots , w_{j-1} , w_{j+1} , \ldots , w_n]
\end{eqnarray}
under $f$.  Then, by applying the definitions of $d_n$ and $F_n$ we find that the
coefficient of $[w_0, w_1, \ldots , w_{j-1} , w_{j+1} , \ldots , w_n]$ in $d_n ( F_n ( Q
) )$ is
\begin{eqnarray}\label{coeff1}
 (-1)^j \big( \sign \hat{f}_{\mid Q} \big),
\end{eqnarray}
whereas the coefficient of $[w_0, w_1, \ldots , w_{j-1} , w_{j+1} , \ldots , w_n]$ in
$F_n ( d_n ( Q ) )$~is
\begin{eqnarray}\label{coeff2}
\big(\sign \hat{f}_{\mid \{ v_0, \ldots , v_{i-1} , v_{i+1} , \ldots v_n \}} \big)
(-1)^i.
\end{eqnarray}
It is a fact (easily proven from the definition of sign) that
\begin{eqnarray}
\big(\sign \hat{f}_{\mid \{ v_0, \ldots , v_{i-1} , v_{i+1} , \ldots v_n \}} \big) =
(-1)^{j-i} \big(\sign \hat{f}_{\mid Q} \big).
\end{eqnarray}
Therefore quantities~(\ref{coeff1}) and (\ref{coeff2}) are equal.  So the coefficients of
$[w_0, w_1, \ldots , w_{j-1} , w_{j+1} , \ldots , w_n]$ in $d_n ( F_n ( Q ) )$ and $F_n (
d_n ( Q ) )$ are the same.  This reasoning can be repeated to show that all of the
coefficients in $d_n ( F_n ( Q ) )$ and $F_n ( d_n ( Q ) )$ are the same.
\end{proof}

We have proven that if $f \colon \Delta \to \Delta'$ is a simplicial isomorphism, then
there is induced chain map (in fact, an isomorphism),
\begin{eqnarray}
F \colon K_\bullet \left( \Delta , \mathbb{F}_p \right) \to K_\bullet \left( \Delta' ,
\mathbb{F}_p \right)\!.
\end{eqnarray}
This chain map induces vector space isomorphisms
\begin{eqnarray}
H_n \left( \Delta , \mathbb{F}_p \right)
\to H_n \left( \Delta' , \mathbb{F}_p \right)
\end{eqnarray}
for every $n \geq 0$.  (We may denote these maps using the same symbol,~$F$.)


\section{Picturing Homology Groups}\label{picturingsection}

Before continuing any further with our technical discussion of chain complexes, let us
take a moment to explore some geometric interpretations for the concepts introduced so
far. For convenience, we will assume in the following discussion that $p$ is a prime
greater than or equal to~$5$.

Consider the the two-dimensional simplicial complex $\Gamma$ shown in
Figure~\ref{trianglefig1}. If $[v_0, v_1]$ is a $1$-simplex (where we assume the
existence of an ordering under which $v_0 < v_1$), then let us represent the chain
element $[v_0, v_1 ] \in K_1 ( \Gamma , \mathbb{F}_p )$ by drawing an arrow from $v_0$ to
$v_1$, and let us represent the negation $- [v_0, v_1 ] \in K_1 ( \Gamma , \mathbb{F}_p
)$ by drawing an arrow from $v_1$ to $v_0$. We can likewise use double-headed arrows to
represent\vadjust{\pagebreak} the elements $2 [v_0, v_1]$ and $-2 [v_0, v_1]$.  Sums of
such elements can be represented as collections of arrows. In this way we can draw
some of the elements of $K_1 ( \Gamma , \mathbb{F}_p )$ as diagrams like the one in
Figure~\ref{trianglefig1}.

\begin{figure}[!b]
 \centerline{\includegraphics{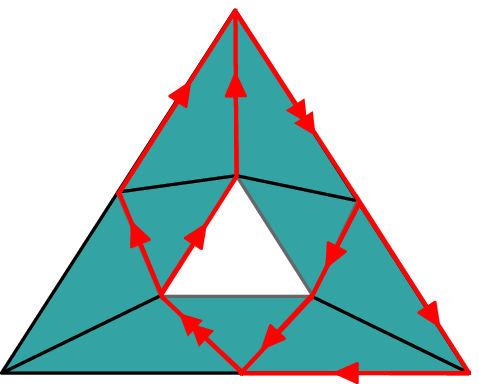}}
\fcaption{A complex $\Gamma$ and a chain element $a \in K_1 ( \Gamma , \mathbb{F}_p
)$.\label{trianglefig1}}
\end{figure}

\begin{figure}[!b]
 \centerline{\includegraphics{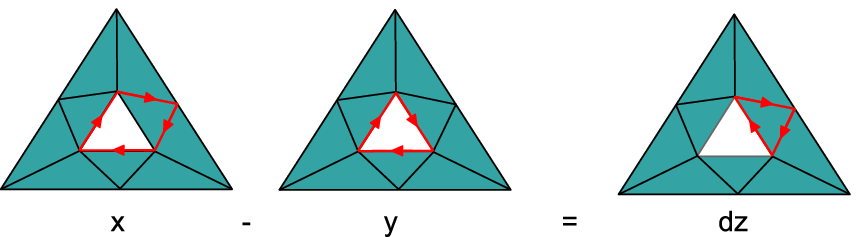}}
\fcaption{Two elements $x,y \in K_1 ( \Gamma , \mathbb{F}_p )$ which are contained in the
same coset of $H_1 ( \Gamma , \mathbb{F}_p )$.\label{trianglefig3}}
\end{figure}

An element $c \in K_1 ( \Gamma , \mathbb{F}_p )$ that is represented in this way will
satisfy $dc = 0$ if and only if for every vertex $v$ of $\Gamma$, the total multiplicity
of incoming arrows at $v$ is the same, mod $p$, as the total multiplicity of the outgoing
arrows at $v$.  The element $a$ represented in Figure~\ref{trianglefig1} is such a case.

Each element $c \in K_1 ( \Gamma , \mathbb{F}_p )$ satisfying $dc = 0$ represents an
element of the quotient $H_1 (  \Gamma , \mathbb{F}_p ) = \ker d_1 / \im d_2$, and thus
we can use this geometric interpretation to understand $H_1 ( \Gamma , \mathbb{F}_p )$.
Note that, although there are many diagrams that we could draw which satisfy the
balanced-multiplicity condition mentioned above, it will often occur that two diagrams
represent the same element of $H_1 ( \Gamma , \mathbb{F}_p )$. Figure~\ref{trianglefig3}
gives an example. In fact, any two elements $u, v \in \ker d_1$ will lie in the same
coset of $H_1 ( \Gamma , \mathbb{F}_p )$ if and only if the amount of flow around the
missing center triangle of $\Gamma$ is the same mod $p$ for both $u$ and $v$.  This makes
it easy to express the structure of $H_1 ( \Gamma , \mathbb{F}_p )$: if we let $\alpha
\in H_1 ( \Gamma , \mathbb{F}_p )$ be the coset containing the element $y$ from
Figure~\ref{trianglefig3}, then $H_1 ( \Gamma , \mathbb{F}_p )$ is a one-dimensional
$\mathbb{F}_p$-vector space that is spanned by $\alpha$.

Meanwhile, it is easy to see that $\ker d_2 = \{ 0 \}$ and hence \hbox{$H_2 ( \Gamma ,
\mathbb{F}_p ) = \{ 0 \}$}.  We thus have the following:
\begin{eqnarray}
H_0 ( \Gamma , \mathbb{F}_p ) & \cong & \mathbb{F}_p \\
H_1 ( \Gamma , \mathbb{F}_p ) & \cong & \mathbb{F}_p \\
H_i ( \Gamma , \mathbb{F}_p ) & \cong & \{ 0 \} \hskip0.2in
\textnormal{for all } i \geq 2.
\end{eqnarray}
This kind of reasoning can be used to describe the homology groups of any finite
simplicial complex $\Pi$ that is contained in $\mathbb{R}^2$. The dimension of $H_1 ( \Pi
, \mathbb{F}_p )$ for such a complex is always equal to the number holes enclosed by
$\Pi$.

\begin{figure}[!b]
 \centerline{\includegraphics[scale=0.99]{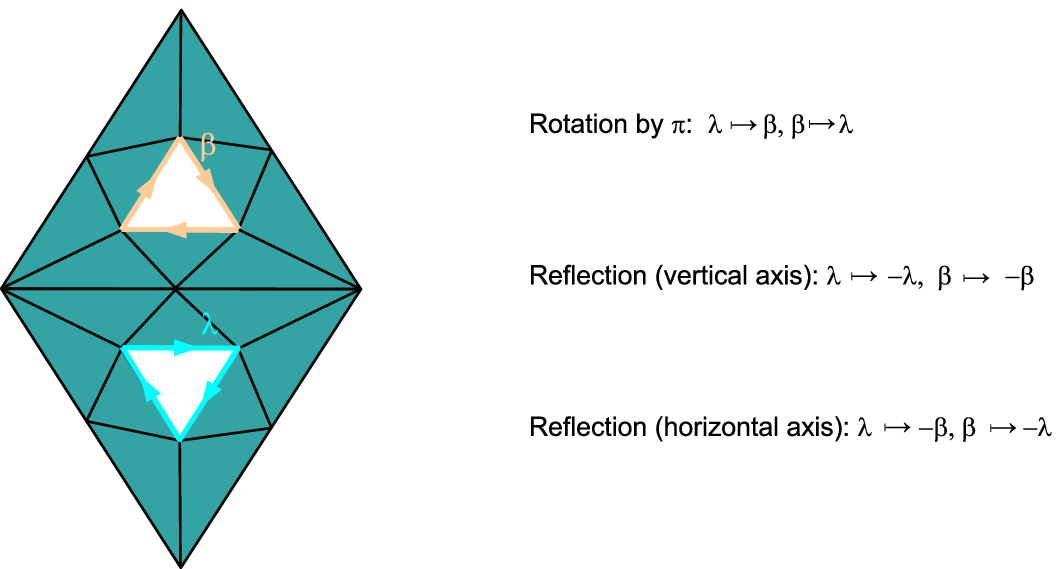}}
\fcaption{The complex $\Gamma'$ and the effect of three different
automorphisms.\label{2trianglefig}}
 \vspace*{-3pt}
\end{figure}

Such visualizations are also useful for understanding the behavior of homology groups
under automorphisms. Figure~\ref{2trianglefig} shows an example of a simplicial complex
$\Gamma'$ for which $H_1 ( \Gamma' , \mathbb{F}_p ) \cong \mathbb{F}_p^2$.  Any
automorphism of $\Gamma'$ induces a linear automorphism of $H_1 ( \Gamma' , \mathbb{F}_p
)$.  The figure describes a few such automorphism in terms of two chosen basis elements
$\lambda , \beta \in H_1 ( \Gamma' , \mathbb{F}_p )$.

\begin{figure}[!t]
 \centerline{\includegraphics[scale=0.95]{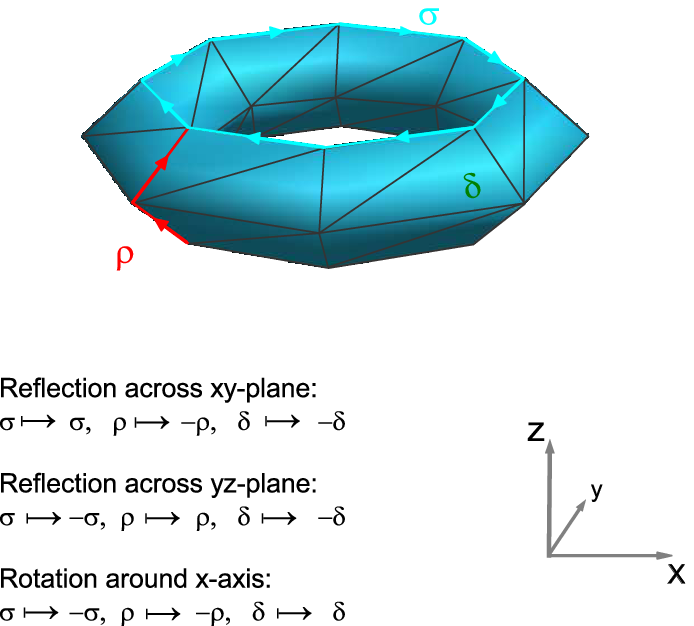}}
\fcaption{The complex $\Lambda$ and the effect of three different
automorphisms.\label{torusfigure}}
 \vspace*{-3pt}
\end{figure}

To observe nontrivial automorphisms of higher homology groups, we need to consider
simplicial complexes in three-dimensional space. Figure~\ref{torusfigure} shows a
simplicial complex $\Lambda$ in $\mathbb{R}^3$ which has the shape of a torus.  Let $z
\in K_2 ( \Lambda , \mathbb{F}_p )$ be a linear combination of\vadjust{\pagebreak} all
the $2$-simplices in $\Lambda$ in which  the coefficient of the simplex $[v_0, v_1, v_2]$
in $z$ is $(+1)$ if the vertices $v_0$, $v_1$, and $v_2$ appear in clockwise order on the
surface of the torus, and $(-1)$ if they appear in counterclockwise order.  When
Definition~\ref{boundarymapdef} is applied to compute $d z$, all terms cancel and we find
that $d z = 0$.  The element $z$ determines a coset $\delta \in H_2 ( \Lambda ,
\mathbb{F}_p)$, which spans the one-dimensional space $H_2 ( \Lambda, \mathbb{F}_p )$.

 \enlargethispage{12pt}

Figure~\ref{torusfigure} gives a basis $\{ \sigma , \rho \}$ for the
two-dimensional space $H_1  ( \Lambda , \mathbb{F}_p )$, and explains the
effect of various automorphisms on $H_1 ( \Lambda , \mathbb{F}_p )$ and $H_2 ( \Lambda ,
\mathbb{F}_p )$.

\section{Some Homological Algebra}

We resume developing concepts from an algebraic standpoint. It is helpful now to take
time to study homology groups in a more abstract setting, without reference to simplicial
complexes.  For  any complex of abelian groups
\begin{equation}
\xymatrix{\ldots
\ar[r] & K_{n+1} \ar[r]^{d_{n+1}}
& K_n \ar[r]^{d_n} & K_{n-1} \ar[r]^{d_{n-1}} & \ldots},
\end{equation}
the $n$th homology group of $K_\bullet$ is defined by
\begin{equation}
H_n ( K_\bullet , \mathbb{F}_p ) =
( \ker  d_n ) / ( \im  d_{n+1} ).
\end{equation}
In this part of the text we will state a result (Proposition~\ref{snakelemmaprop}) which
allows us to relate the homology groups\vadjust{\pagebreak} of $K_\bullet$ to the
homology groups of smaller complexes. This will be an essential building block in later
proofs.

Let us say that a sequence of maps of abelian groups
\begin{eqnarray}
\xymatrix{ \ldots \ar[r] & A_{n+1} \ar[r]^{f_{n+1}} &
A_n \ar[r]^{f_n} &
A_{n-1} \ar[r]^{f_{n-1}} & \ldots }
\end{eqnarray}
is \textbf{exact} if it satisfies the condition $\ker f_n = \im f_{n+1}$ for every $n$.
Thus, a sequence of the form
\begin{eqnarray}
\xymatrix{ 0 \ar[r] & P \ar[r]^f & Q \ar[r]^g & R \ar[r] & 0}
\end{eqnarray}
is\enlargethispage{12pt} exact if and only if $f$ is injective, $g$ is surjective, and
$\im f = \ker g$. (Note that this makes $R$ isomorphic to the quotient $Q / f ( P )$.)
Suppose that a sequence of maps of complexes
\begin{eqnarray}\label{seqofcomplexes}
\xymatrix{ 0 \ar[r] & X_\bullet \ar[r]^F & Y_\bullet \ar[r]^G & Z_\bullet \ar[r] & 0}
\end{eqnarray}
is such that
\begin{eqnarray}
\xymatrix{ 0 \ar[r] & X_n \ar[r]^{F_n} &
Y_n \ar[r]^{G_n} & Z_n \ar[r] & 0 }
\end{eqnarray}
is an exact sequence for every $n$.  Then we will say that~(\ref{seqofcomplexes}) is an
exact sequence of complexes.

I claim that if
\begin{eqnarray}
\xymatrix{
&  \vdots \ar[d] & \vdots \ar[d] & \vdots \ar[d] &  \\
0 \ar[r] & X_{n+1} \ar[r]^F \ar[d]^d  & Y_{n+1} \ar[r]^G \ar[d]^d & Z_{n+1} \ar[r] \ar[d]^d & 0 \\
0 \ar[r] & X_n \ar[r]^F \ar[d]  & Y_n \ar[r]^G \ar[d] & Z_n \ar[r] \ar[d] & 0 \\
&  \vdots & \vdots & \vdots &  \\
}
\end{eqnarray}

\noindent is an exact sequence of complexes, then
\begin{eqnarray}
H_n ( X_\bullet ) \to H_n ( Y_\bullet ) \to H_n ( Z_\bullet )
\end{eqnarray}
is an exact sequence.  This can be seen through a ``diagram-chasing'' argument. It is
obvious that
\begin{eqnarray}
\im \left[ H_n ( X_\bullet ) \to H_n ( Y_\bullet ) \right] \subseteq
\ker \left[ H_n ( Y_\bullet ) \to H_n ( Z_\bullet ) \right],
\end{eqnarray}
and so we only need to prove the reverse inclusion. Suppose that $(y + \im d_n^Y)$ is a
coset in $H_n ( Y_\bullet )$ that is killed by the map to $H_n ( Z_\bullet )$.  Then $G (
y ) \in \im d_{n+1}^Z$, so we can find $z' \in Z_{n+1}$ such that $dz' = G( y )$.
Choosing an arbitrary element $y' \in G^{-1} \{ z' \}$, we have $y - dy' \in \ker G$, and
therefore by exactness, $F (x) = y - dy'$ for some $x$.  Since $dy = 0$ and $d(dy') = 0$,
we have $F ( dx ) = d F (x ) = 0$ and therefore $dx = 0$.  Thus $(x + \im d_n^X)$ is a
coset in $H_n ( X_\bullet )$ which maps to $y + \im d_n^Y$, and the claim is proved.


While it might be tempting to assume that the maps $H_n ( X_\bullet ) \to H_n ( Y_\bullet
)$ are injective and the maps $H_n ( Y_\bullet ) \to H_n ( Z_\bullet )$ are surjective,
this is not generally true.  The homology groups of $X_\bullet$, $Y_\bullet$, and
$Z_\bullet$ have a more complex relationship which is expressed by the following
proposition.

\begin{proposition}\label{snakelemmaprop}
Let $X_\bullet$, $Y_\bullet$, and $Z_\bullet$ be complexes of abelian groups,
and let $F \colon X_\bullet \to Y_\bullet$ and $G \colon Y_\bullet \to Z_\bullet$ be
maps of complexes such that for any $n$, the sequence
\begin{eqnarray}
\xymatrix{ 0 \ar[r] & X_n \ar[r]^{F_n} &
Y_n \ar[r]^{G_n} & Z_n \ar[r] & 0 }
\end{eqnarray}
is an exact sequence.  Then, there exist homomorphisms
\begin{eqnarray}
\gamma_n \colon H_n \left( Z_\bullet  \right) \to
H_{n-1} \left( X_\bullet \right)
\end{eqnarray}

\noindent for every $n$ which are such that the sequence
\begin{eqnarray}\label{snakesequence}
\xymatrix{ \ldots \ar[r] & H_2 ( Y_\bullet ) \ar[r] &
H_2 ( Z_\bullet ) \ar[lldd]^{\gamma_2} \\
\\
H_1 ( X_\bullet ) \ar[r] & H_1 ( Y_\bullet )
\ar[r] & H_1 ( Z_\bullet) \ar[lldd]^{\gamma_1} \\
\\
H_{0} ( X_\bullet ) \ar[r] & H_{0} ( Y_\bullet)
\ar[r] & H_0 (  Z_\bullet ) \ar[r] & 0 }
\end{eqnarray}
is exact.
\end{proposition}

 \removelastskip\pagebreak

Since the proof of this proposition is fairly technical, we have placed it in
Appendix.  (See Proposition~\ref{realsnakelemma}.) The maps $\gamma_n$ can be
briefly described like so: let $\overline{G}_n \colon Z_n \to Y_n$ be a function (not
necessarily a homomorphism) which is such that $G_n \circ \overline{G}_n$ is the identity
map, and let $\overline{F}_n \colon  F ( X_n ) \to X_n$ be the inverse of $F$. Then, for
any coset
\begin{eqnarray}
z + \im d_{n+1}^Z \in H_n ( Z_\bullet),
\end{eqnarray}
the image under $\gamma_n \colon H_n ( Z_\bullet ) \to H_{n-1} ( X_\bullet )$ is given by
\begin{eqnarray}
\overline{F}_{n-1} ( d ( \overline{G}_n ( z ))) + \im d_n^X \in H_{n-1} ( X_\bullet ).
\end{eqnarray}
As we will see, the above proposition is very useful because it allows us to draw
conclusions about the homology groups of a complex $Y_\bullet$ based on the homology
groups of its subcomplexes and quotient complexes.

We close with a few additional constructions. Note that for any map of complexes $F
\colon I_\bullet \to J_\bullet$, there exist the complexes
\begin{eqnarray}
\xymatrix{\ldots
\ar[r] & \im F_{n+1} \ar[r]^{d_{n+1}}
& \im F_n \ar[r]^{d_n} & \im F_{n-1} \ar[r]^{d_{n-1}} & \ldots}
\end{eqnarray}
and
\begin{eqnarray}
\xymatrix{\ldots
\ar[r] & \ker F_{n+1} \ar[r]^{d_{n+1}}
& \ker F_n \ar[r]^{d_n} & \ker F_{n-1} \ar[r]^{d_{n-1}} & \ldots} .
\end{eqnarray}
We write these complexes as $(\im F)$ and $(\ker F)$, respectively.
Note that these complexes fit into an exact sequence
\begin{eqnarray}
0 \to \ker F \to I_\bullet \to \im F \to 0.
\end{eqnarray}

The $\textbf{direct sum}$ of $I_\bullet$ and $J_\bullet$, written
$I_\bullet \oplus J_\bullet$, is the complex
\begin{equation}
\xymatrix{ \ldots \ar[r] & I_{n+1} \oplus J_{n+1} \ar[r] & I_n \oplus J_n \ar[r] &
I_{n-1} \oplus J_{n-1} \ar[r] & \ldots },
\end{equation}
where the maps in this complex are simply the maps induced by $d_k \colon I_k \to
I_{k-1}$ and $d_k \colon J_k \to J_{k-1}$. Note that the homology groups of this complex
are simply $H_n ( I_\bullet ) \oplus H_n ( J_\bullet )$.

\section{Collapsibility Implies Acyclicity}\label{acyclicitysection}

Now we will offer our first application of Proposition~\ref{snakelemmaprop}. In
\textit{\nameref{graphpropsection}}, we defined the notion of {\it collapsibility} for
simplicial \hbox{complexes}.  In this part of the text we will see how the condition of
collapsibility for a simplicial complex $\Delta$ implies that the homology groups of
$\Delta$ are trivial.

We begin with a useful definition.

\begin{definition}
Let $\Delta$ be an abstract simplicial complex whose vertex-set is totally ordered. Let
$p$ be a prime, and let $n$ be a nonnegative integer.  Define the map
\begin{eqnarray}
s \colon K_0 \left( \Delta , \mathbb{F}_p \right) \to \mathbb{F}_p
\end{eqnarray}
by asserting that $s ( \gamma )$ is the sum of the coefficients of $\gamma$.  That is, if
\begin{eqnarray}
\gamma & = & c_1 Q_1 + c_2 Q_2 + \ldots + c_r Q_r,
\end{eqnarray}
with $c_i \in \mathbb{F}_p$ and $Q_i \in \Delta$, then
\begin{eqnarray}
s( \gamma ) = c_1 + c_2 + \ldots + c_r \in \mathbb{F}_p.
\end{eqnarray}
The \textbf{reduced $n$th homology group of $\Delta$} over $\mathbb{F}_p$, denoted
$\widetilde{H}_n \left( \Delta , \mathbb{F}_p \right)$, is the $n$th homology group of
the complex
\begin{eqnarray*}
\xymatrix{\ldots \ar[r] & K_2 \left( \Delta, \mathbb{F}_p \right) \ar[r]^{d_2} & K_1
\left( \Delta , \mathbb{F}_p \right)  \ar[r]^{d_1}& K_0 \left( \Delta , \mathbb{F}_p
\right) \ar[r]^{\hspace*{0.2in}s} & \mathbb{F}_p \ar[r] & 0 }
\end{eqnarray*}
\end{definition}

The reduced homology groups $\left\{ \widetilde{H}_n \left( \Delta , \mathbb{F}_p \right)
\right\}$ of an abstract simplicial complex $\Delta$ are the same as the ordinary
homology groups $\left\{ H_n \left( \Delta , \mathbb{F}_p \right) \right\}$, except that
the dimension of $\widetilde{H}_0 \left( \Delta , \mathbb{F}_p \right)$ is one less than
the dimension of $H_0 \left( \Delta , \mathbb{F}_p \right)$.  Note that the reduced
homology groups of the trivial complex $\{ \{ 0 \} \}$ are all zero.

\begin{definition}\label{acyclicitydefinition}
Let $\Delta$ be an abstract simplicial complex whose vertex-set is totally ordered.
Then, $\Delta$ is \textbf{$\mathbb{F}_p$-acyclic} if
\begin{eqnarray}
\dim_{\mathbb{F}_p } \widetilde{H}_n \left( \Delta , \mathbb{F}_p \right) = 0
\end{eqnarray}
for all nonnegative integers $n$.
\end{definition}

 \removelastskip\pagebreak

Stated differently, a complex is $\mathbb{F}_p$-acyclic if its $\mathbb{F}_p$-homology is
the same as that of single point. An example of an $\mathbb{F}_p$-acyclic simplicial
complex is this one, from Example~\ref{solidtriangleexample}.
\begin{eqnarray*}
\Sigma' = \left\{ \{ 0 \} , \{ 1 \}, \{ 2 \} , \{ 0, 1 \} , \{ 1, 2 \} , \{ 0, 2 \} , \{
0, 1, 2 \} \right\}\!.
\end{eqnarray*}
One can check by direct calculation that all of the reduced homology groups of this
simplicial complex are trivial.  On the other hand, the simplicial complex $\Sigma$ of
Example~\ref{triangleexample} is \textit{not} $\mathbb{F}_p$-acyclic, since
$\widetilde{H}_1 \left( \Sigma' , \mathbb{F}_p \right) \cong \mathbb{F}_p$.

Another way of expressing Definition~\ref{acyclicitydefinition} is this: an abstract
simplicial complex $\Delta$ is $\mathbb{F}_p$-acyclic if
\begin{eqnarray*}
\xymatrix{\ldots \ar[r] & K_2 \left( \Delta, \mathbb{F}_p \right) \ar[r]^{d_2} & K_1
\left( \Delta , \mathbb{F}_p \right) \ar[r]^{d_1} & K_0 \left( \Delta , \mathbb{F}_p
\right) \ar[r]^{\hskip0.2in s} & \mathbb{F}_p \ar[r] & 0 }
\end{eqnarray*}
is an exact sequence.

When a complex forms an exact sequence, let us refer to it as an \textbf{exact complex}.
The following algebraic lemma is useful for proving exactness of complexes.

\begin{lemma}\label{exactnesslemma}
Let
\begin{eqnarray}
0 \to X_\bullet \to Y_\bullet \to Z_\bullet \to 0
\end{eqnarray}
be an exact sequence of complexes of abelian groups.  Then,
\begin{enumerate}
\item \label{alglemmapart1}
If $X_\bullet$ and $Y_\bullet$ are exact complexes, then $Z_\bullet$ is an exact\\
complex.
\item \label{alglemmapart2}
If $Y_\bullet$ and $Z_\bullet$ are exact complexes, then $X_\bullet$ is an exact\\
complex.
\item \label{alglemmapart3}
If $X_\bullet$ and $Z_\bullet$ are exact complexes, then $Y_\bullet$ is an exact\\
complex.
\end{enumerate}
\end{lemma}

\begin{proof}
We prove (\ref{alglemmapart1}).  Suppose that $X_\bullet$ and $Y_\bullet$ are
exact complexes.  By Proposition~\ref{snakelemmaprop}, there is an exact sequence
\begin{eqnarray*}
&&\ldots \to H_{n+1} ( Z_\bullet) \to H_n ( X_\bullet) \to H_n ( Y_\bullet) \to H_n(
Z_\bullet)\\
&& \qquad \to H_{n-1} ( X_\bullet) \to H_{n-1} ( Y_\bullet ) \to \ldots
\end{eqnarray*}
The reader will observe that since the groups $\left\{ H_n ( X_\bullet ) \right\}$ and
$\left\{ H_n ( Y_\bullet ) \right\}$ are all zero, the groups $\left\{ H_n ( Z_\bullet )
\right\}$ must all be zero as well. Therefore~$Z_\bullet$ is an exact complex.

Assertions (\ref{alglemmapart2}) and (\ref{alglemmapart3}) follow similarly.
\end{proof}

Now we are ready to prove our main theorem.

\begin{theorem}
\label{acyclicitytheorem} Let $p$ be a prime.  Let $\Delta$ be an abstract simplicial
complex which has a total ordering on its vertex set. If $\Delta$ is collapsible, then
$\Delta$ is $\mathbb{F}_p$-acyclic.
\end{theorem}

\begin{proof}
Recall (from \textit{\nameref{graphpropertysimplicial}}) the definition of
\textbf{primitive elementary collapse}. For any elementary collapse $( \Sigma, \Sigma'
)$, there is a sequence of primitive elementary collapses which reduces \hbox{$\Sigma$ to
$\Sigma'$}:
\begin{eqnarray}
\Sigma, \Sigma_1 , \Sigma_2 , \ldots , \Sigma_t , \Sigma'
\end{eqnarray}
(This is an elementary fact which the reader is invited to prove as an exercise.)

Suppose that the complex $\Delta$ is collapsible. There exists a sequence of elementary
collapses which collapse $\Delta$ to a single $0$-simplex.  Therefore, there exists a
sequence of \textit{primitive} elementary collapses which collapse $\Delta$ to a single
$0$-simplex.  Let
\begin{eqnarray}
\Delta, \Delta_1 , \Delta_2, \ldots , \Delta_r
\end{eqnarray}
be such a sequence, with $\left| \Delta_r
\right| = 1$.

Let $Z_\bullet$ be the complex formed by the quotient groups
\begin{eqnarray}
Z_n = K_n (\Delta , \mathbb{F}_p)/ K_n \left( \Delta_1 , \mathbb{F}_p \right)
\end{eqnarray}
The structure of the complex $Z_\bullet$ is quite simple: it is isomorphic
to the following complex:
\begin{eqnarray}
\hspace*{-6pt}\xymatrix{\ldots \ar[r] & 0 \ar[r] & 0 \ar[r] & \mathbb{F}_p \ar[r]^{Id} &
\mathbb{F}_p \ar[r] & 0 \ar[r] & 0 \ar[r] & \ldots}\qquad
\end{eqnarray}

 \vfill\eject

\noindent There is an exact sequence of complexes
\begin{eqnarray}
\xymatrix{
& \vdots \ar[d] & \vdots \ar[d]  & \vdots \ar[d] & \\
0 \ar[r] & K_1 \left( \Delta_1 , \mathbb{F}_p \right)
\ar[r] \ar[d]
& K_1 \left( \Delta , \mathbb{F}_p \right)  \ar[r]  \ar[d]
& Z_1 \ar[r] \ar[d] & 0 \\
0 \ar[r] & K_0 \left( \Delta_1 , \mathbb{F}_p \right)
\ar[r] \ar[d]
& K_0 \left( \Delta , \mathbb{F}_p \right)  \ar[r]  \ar[d]
& Z_0 \ar[r] \ar[d] & 0 \\
0 \ar[r] & \mathbb{F}_p \ar[r] \ar[d] &
\mathbb{F}_p \ar[r] \ar[d]  & 0 \ar[r] \ar[d] & 0 \\
0 \ar[r] & 0 \ar[r] & 0 \ar[r] & 0 \ar[r] & 0 }
\end{eqnarray}
The complex $Z_\bullet$ is clearly exact.  So by Lemma~\ref{exactnesslemma},
the complex
\begin{eqnarray*}
\xymatrix{\ldots \ar[r] & K_2 \left( \Delta, \mathbb{F}_p \right) \ar[r]^{d_2} & K_1
\left( \Delta , \mathbb{F}_p \right) \ar[r]^{d_1} & K_0 \left( \Delta , \mathbb{F}_p
\right)\ar[r]^{\hspace*{0.2in}s} & \mathbb{F}_p \ar[r] & 0 }
\end{eqnarray*}
is exact iff
\begin{eqnarray*}
\hspace*{-3pt}\xymatrix{\ldots \ar[r] & K_2 (\Delta_1, \mathbb{F}_p) \ar[r]^{d_2} & K_1
(\Delta_1 , \mathbb{F}_p) \ar[r]^{d_1} & K_0 (\Delta_1 , \mathbb{F}_p)
\ar[r]^{\hspace*{0.2in}s} & \mathbb{F}_p \ar[r] & 0}
\end{eqnarray*}
is exact.  Therefore $\Delta$ is $\mathbb{F}_p$-acyclic iff $\Delta_1$ is $\mathbb{F}_p$-acyclic.

Similar reasoning shows that for any $i$, $\Delta_i$ is $\mathbb{F}_p$-acyclic iff
$\Delta_{i+1}$ is $\mathbb{F}_p$-acyclic.  The theorem follows by induction.
\end{proof}

\chapter{Fixed-Point Theorems}\label{fptchapter}

We are now ready to put the theory from \textit{\nameref{chaincomplexchapter}} to use to
study group actions $G \circlearrowleft \Delta$ on simplicial complexes.

\section{The Lefschetz Fixed-Point Theorem}\label{lftsection}

\begin{theorem}\label{acyclicimplieslft}
Let $\Delta$ be a finite abstract simplicial complex with ordered vertices.  Suppose that
$\Delta$ is $\mathbb{F}_p$-acyclic for some prime number $p$. Let $f \colon \Delta \to
\Delta$ be a simplicial automorphism. Then, there exists a simplex $Q \in \Delta$ such
that $f ( Q ) = Q$.
\end{theorem}

\begin{proof}
Let us introduce some notation: if $Y$ is a finite-dimensional vector space over
$\mathbb{F}_p$, and $h \colon Y \to Y$ is a linear endomorphism, then let $\Tr_h \left( Y
\right)$ denote the trace of $h$ on $Y$.  Note that the trace function is additive over
exact sequences.  That is, if
\begin{eqnarray}
0 \to X \to Y \to Z \to 0
\end{eqnarray}
is an exact sequence, and $h$ acts on $X$, $Y$, and $Z$ in a compatible manner, then
\begin{eqnarray}
\Tr_h (Y) = \Tr_h (X) + \Tr_h (Z).
\end{eqnarray}

Let $F$ denote the chain map associated with $f$.  Consider the
\hbox{values~of}
\begin{eqnarray}
\Tr_F \left( H_n \left( \Delta , \mathbb{F}_p \right) \right)
\end{eqnarray}
for $n = 0 , 1, 2, {\ldots}\,$. Since $\Delta$ is $\mathbb{F}_p$-acyclic, these are easy
to compute.  If $n > 0$, then $H_n \left( \Delta , \mathbb{F}_p \right)$ is a zero vector
space. The vector space $H_0 \left( \Delta , \mathbb{F}_p \right)$ is a one-dimensional
$\mathbb{F}_p$-vector space on which $F$ acts trivially.  Therefore,
\begin{eqnarray}
\Tr_F \left( H_0 \left( \Delta , \mathbb{F}_p \right) \right) & = & 1, \\
\Tr_F \left( H_n \left( \Delta, \mathbb{F}_p \right) \right) & = & 0 \hskip0.2in \textnormal{ for } n > 0.
\end{eqnarray}

Now we can carry out the proof using the additivity of the trace function. Suppose, for
the sake of contradiction, that there is \textit{no} simplex in $\Delta$ which is
stabilized by $F$.  Then, for any $n$, the chain map $F$ acts on $K_n \left( \Delta ,
\mathbb{F}_p \right)$ by permuting the basis elements in a fixed-point free manner,
possibly changing signs.  A matrix representation of this action would be a matrix with
entries from the set $\{ -1, 0, 1 \}$, having only zeroes on the main diagonal.  Thus we
see that
\begin{eqnarray}
\Tr_F \left( K_n \left( \Delta , \mathbb{F}_p \right) \right) = 0.
\end{eqnarray}
Observe the following chain of equalities.
\begin{eqnarray*}
0 & = & \sum_{n \geq 0} (-1)^n \Tr_F \left( K_n \left( \Delta , \mathbb{F}_p \right) \right) \\
& = & \Tr_F \left( K_0 \left( \Delta , \mathbb{F}_p \right) \right) +
\sum_{n \geq 1} (-1)^n \left[
\Tr_F \left( \im d_n \right)  +
\Tr_F \left( \ker d_n \right)
\right] \\
& = & \Tr_F \left( K_0 \left( \Delta , \mathbb{F}_p \right) \right) - \Tr_F \left( \im
d_1 \right)\\
&&  +\, \sum_{n \geq 1} (-1)^n \left[ \Tr_F \left( \ker d_n \right) -
\Tr_F \left( \im d_{n+1} \right)  \right] \\
& = & \Tr_F \left( H_0 \left( \Delta , \mathbb{F}_p \right) \right)
+ \sum_{n \geq 1} (-1)^n \Tr_F \left( H_n \left( \Delta ,
\mathbb{F}_p \right) \right) \\
& = & 1.
\end{eqnarray*}
We obtain a contradiction.  Therefore, there must exist a simplex $Q$ in $\Delta$ such
that $f ( Q ) = Q$.
\end{proof}

 \pagebreak

 \vspace*{-20pt}

\begin{corollary}\label{lefschetzcorollary}
Let $\Sigma$ be a finite abstract simplicial complex which is collapsible.  Let $g \colon
\Sigma \to \Sigma$ be a simplicial automorphism.  Then there must exist a simplex $T \in
\Sigma$ such that $g ( T ) = T$.
\end{corollary}

\begin{proof}
This follows immediately from the above theorem and Theorem~\ref{acyclicitytheorem}.
\end{proof}

Let us consider what Theorem~\ref{acyclicimplieslft} means geometrically. Suppose that
$\Theta$ is an ordinary simplicial complex in $\mathbb{R}^N$ (see
\textit{\nameref{simplicialcomplexsection}}).  Then a simplicial automorphism of $\Theta$
is simply a continuous permutation of the points of $\Theta$ which maps every $n$-simplex
of $\Theta$ to another $n$-simplex of $\Theta$ in an affine-linear manner.

Suppose that $V \subset \mathbb{R}^N$ is a single $n$-simplex spanned by $\mathbf{v}_0 ,
\mathbf{v}_1 , \ldots , \mathbf{v}_n \in \mathbb{R}^N$. Note that any affine-linear map
of $V$ onto itself must fix the point
\begin{eqnarray}
\sum_{i = 0}^n \left( \frac{1}{n+1} \right) \mathbf{v}_i \in V.
\end{eqnarray}
{\tra2Thus, any simplicial map which stabilizes $V$ must have a fixed point in~$V$.
Therefore, when we establish that a simplicial automorphism maps a particular simplex to
itself, we have in fact proved that it has a fixed point. This justifies
our calling Theorem~\ref{acyclicimplieslft} a ``fixed-point theorem.''}

 \enlargethispage{10pt}

Let $f \colon \Delta \to \Delta$ be a simplicial automorphism which satisfies the
assumptions of Theorem~\ref{acyclicimplieslft}. We can use the reasoning from the proof
of Theorem~\ref{acyclicimplieslft} to draw further conclusions about the set $\Delta^f$.
Note that the quantity
\begin{eqnarray}
\Tr_F ( K_n \left( \Delta , \mathbb{F}_p \right) )
\end{eqnarray}
is equal to the number of $n$-simplicies $Q \in \Delta$ that satisfy $f ( Q ) = Q$.  By
the reasoning from the proof of Theorem~\ref{acyclicimplieslft}, we have
\begin{eqnarray}
\sum_{n \geq 0} (-1)^n \Tr_F \left( K_n ( \Delta , \mathbb{F}_p ) \right) & = & 1.
\end{eqnarray}
This implies a different version of Theorem~\ref{acyclicimplieslft}. For any subset $S$
of a simplicial complex $\Delta$, let
\begin{eqnarray}
\chi \left( S \right) & = & \sum_{n \geq 0} (-1)^n \left| \left\{ Q \in S \mid \dim ( Q )
= n \right\} \right|\!.
\end{eqnarray}
The quantity $\chi ( S )$ is called the \textbf{Euler characteristic} of $S$.

 \removelastskip\pagebreak

\begin{theorem}
Let $\Delta$ be a finite abstract simplicial complex with ordered vertices, and suppose
that $\Delta$ is $\mathbb{F}_p$-acyclic for some prime number $p$. Let $f \colon \Delta
\to \Delta$ be a simplicial automorphism. Then,
\begin{eqnarray}
\chi ( \Delta^f )  & = & 1.
\end{eqnarray}
\end{theorem}


\section{A Nonabelian Fixed-Point Theorem}\label{nonabeliansection}

In this part of the text we will prove a nonabelian fixed-point theorem which is
attributed to R.~Oliver \cite{oliver1975}.

Let $\Delta$ be a collapsible abstract simplicial complex.  Let $G$ be a finite group
which acts on $\Delta$ via simplicial automorphisms. By
Corollary~\ref{lefschetzcorollary}, we know that for any element $g \in G$, there must be
a simplex $Q \in \Delta$ such that $g \left ( Q \right) = Q$. We will prove that, under
certain conditions, a stronger statement can be made: there must exist a single
simplex~$Q$ which is stabilized by all the elements of $G$.

Our method of proof for this result is essentially an inductive one.  We require that the
automorphism group $G$ has a certain filtration by subgroups,
\begin{eqnarray}
\{ 0 \} = G_0 \subset G_1 \subset G_2 \subset \ldots \subset G_r = G,
\end{eqnarray}
and we inductively deduce conditions on the $G_i$-fixed subsets of $\Delta$, for $i = 0,
1, 2, \ldots, r$.  The key to this argument is the first result that we will prove,
Proposition~\ref{acyclicitypreservation}, which tells us that the property of
``$\mathbb{F}_p$-acyclicity'' can be carried forward along this filtration.  The proof of
Proposition~\ref{acyclicitypreservation} is the most difficult part of the argument; once
that proposition is proved, the other elements of the argument fall into place easily.

 \enlargethispage{6pt}

For now, we will be focusing our attention on simplicial automorphisms $f \colon \Delta
\to \Delta$ for which $\Delta^f$ \textit{is} a subcomplex of $\Delta$.  That is, we will
be focusing on those maps $f$ satisfying the condition
\begin{eqnarray}
Q \in \Delta^f  \textnormal{ and } Q' \subseteq Q \Longrightarrow
Q' \in \Delta^f.
\end{eqnarray}
for any $Q, Q' \in \Delta$.  Geometrically, what this condition implies is that if $f$
stabilizes a simplex $Q$, then it also fixes all of the vertices of $Q$.

The \textbf{order} of a simplicial automorphism $f \colon \Delta \to \Delta$ is the least
$n \geq 1$ such that $f^n$ is the identity.  (If no such $n$ exists, then the order of
$f$ is $\infty$.)

\begin{proposition}\label{acyclicitypreservation}
Let $\Delta$ be a finite abstract simplicial complex with ordered vertices.   Let $p$ be
a prime, and suppose that $\Delta$ is $\mathbb{F}_p$-acyclic. Suppose that $f \colon
\Delta \to \Delta$ is an order-$p$ automorphism of $\Delta$ such that $\Delta^f$ is a
subcomplex of $\Delta$.  Then, the complex $\Delta^f$ must be $\mathbb{F}_p$-acyclic.
\end{proposition}

\begin{proof}
Suppose that $\Delta$ is $\mathbb{F}_p$-acyclic. We know, by
Theorem~\ref{acyclicimplieslft}, that the subcomplex $\Delta^f$ must be nonempty.  To
prove the proposition, we must show that the homology groups $H_n \left( \Delta^f ,
\mathbb{F}_p \right)$ are trivial for $n > 0$, and that $H_0 \left( \Delta^f ,
\mathbb{F}_p \right)$ is one-dimensional.

The proof that follows is based on the paper ``Fixed-point theorems for periodic
transformations'' by Smith~\cite{smith1941}.  The approach of the proof is to
define some special subcomplexes of $K_\bullet \left( \Delta , \mathbb{F}_p \right)$ and
then exploit relationships between these subcomplexes.

Let
\begin{eqnarray}
F \colon K_\bullet \left( \Delta , \mathbb{F}_p \right) \to
K_\bullet \left( \Delta , \mathbb{F}_p \right)
\end{eqnarray}
denote the chain map associated with $f$.  Note that since $F$ is a map of
complexes, any linear combination of the maps $F , F^2, F^3, \ldots$ is also a map of
complexes. Define
\begin{eqnarray}
\delta \colon K_\bullet \left( \Delta , \mathbb{F}_p \right) \to K_\bullet \left( \Delta
, \mathbb{F}_p \right)
\end{eqnarray}
by
\begin{eqnarray}
\delta = \mathbb{I} - F.
\end{eqnarray}
(Here $\mathbb{I}$ denotes the identity map.)  Define
\begin{eqnarray}
\sigma \colon K_\bullet \left( \Delta , \mathbb{F}_p \right) \to K_\bullet \left( \Delta
, \mathbb{F}_p \right)
\end{eqnarray}
by
\begin{eqnarray}
\sigma = \mathbb{I} + F + F^2 + \ldots + F^{p-1}.
\end{eqnarray}

The maps $\delta$ and $\sigma$ determine four subcomplexes of $K_\bullet \left( \Delta ,
\mathbb{F}_p \right)$:
\begin{eqnarray}
( \im \delta), (\ker\delta), (\im \sigma),\quad \textnormal{and}\quad (\ker \sigma).
\end{eqnarray}
We can describe these four complexes explicitly.  Let $\Delta' = \Delta \smallsetminus
\Delta^f$. Let $S \subseteq \Delta'$ be a set which contains exactly one element from
every $f$-orbit in $\Delta'$.  Then the following assertions hold (as the reader may
verify):
\begin{itemize}
\item The set
\begin{eqnarray*}
\left\{ \sum_{i=0}^{p-1} F^i \left( Q \right) \mid Q \in S \right\}
\end{eqnarray*}
is a basis\footnote{When we say that a set $T$ is a basis for a complex $X_\bullet$, we
mean that $T$ is a union of bases for the vector spaces $\{ X_i \}$.} for $\left( \im
\sigma \right)$.
\item The set
\begin{eqnarray*}
\left\{  F^i ( Q) - F^{i+1} ( Q ) \mid
Q \in S , 0 \leq i \leq p-2 \right\}
\end{eqnarray*}
is a basis for $\left( \im \delta \right)$.
\item The set
\begin{eqnarray*}
\left\{  F^i ( Q) - F^{i+1} ( Q ) \mid
Q \in S , 0 \leq i \leq p-2 \right\} \cup
\left\{ Q \mid Q \in \Delta^f \right\}
\end{eqnarray*}
is a basis for $\left( \ker \sigma \right)$.
\item The set
\begin{eqnarray*}
\left\{ \sum_{i=0}^{p-1} F^i \left( Q \right) \mid
Q \in S \right\}
 \cup
\left\{ Q \mid Q \in \Delta^f \right\}
\end{eqnarray*}
is a basis for $\left( \ker \sigma \right)$.
\end{itemize}
From these bases, we can see that there are the following isomorphisms of complexes:
\begin{eqnarray}\label{keyisomorphism1}
\left( \ker \sigma \right) \cong \left( \im \delta \right)
 \oplus K_\bullet \left( \Delta^f , \mathbb{F}_p \right)
\end{eqnarray}
and
\begin{eqnarray}\label{keyisomorphism2}
\left( \ker \delta \right) \cong \left( \im \sigma \right)
\oplus K_\bullet \left( \Delta^f , \mathbb{F}_p \right).
\end{eqnarray}
These imply isomorphisms of homology groups:
\begin{eqnarray}\label{homiso1}
H_n ( \ker \sigma ) \cong H_n ( \im \delta ) \oplus H_n ( \Delta^f , \mathbb{F}_p ), \\
\label{homiso2} H_n ( \ker \delta ) \cong H_n ( \im \sigma)
\oplus H_n ( \Delta^f , \mathbb{F}_p ).
\end{eqnarray}

Now, consider the exact sequences
\begin{eqnarray}
0 \to \left( \ker \sigma \right) \to K_\bullet \left( \Delta ,
\mathbb{F}_p \right) \to \left( \im \sigma \right) \to 0,  \\
0 \to \left( \ker \delta \right) \to K_\bullet \left( \Delta ,
\mathbb{F}_p \right) \to \left( \im \delta \right) \to 0
\end{eqnarray}
By Proposition~\ref{snakelemmaprop}, these imply the existence of two long exact
sequences:
\begin{eqnarray*}
&& \ldots \to H_{n+1} ( \im \sigma ) \to H_n ( \ker \sigma ) \to H_n ( \Delta ,
\mathbb{F}_p ) \to H_n ( \im \sigma )\\
&&\qquad \to H_{n-1} ( \ker \sigma ) \to \ldots  \\[6pt]
&& \ldots \to H_{n+1} ( \im \delta ) \to H_n ( \ker \delta ) \to H_n ( \Delta ,
\mathbb{F}_p ) \to H_n ( \im \delta )\\
&&\qquad \to H_{n-1} ( \ker \delta ) \to \ldots .
\end{eqnarray*}
Let us step through the terms in these sequences, starting from the left.  Let $c$ be the
dimension of the complex $\Delta$ (that is, the dimension of the largest simplex in
$\Delta$).  The exact sequences take the form
\begin{eqnarray*}
\ldots \longrightarrow 0 \longrightarrow H_c ( \ker \sigma )
\to H_c ( \Delta , \mathbb{F}_p ) \to H_c ( \im \sigma )
\to H_{c-1} ( \ker \sigma ) \to \ldots  \\[6pt]
\ldots \longrightarrow 0 \longrightarrow H_c ( \ker \delta ) \to H_c ( \Delta ,
\mathbb{F}_p ) \to H_c ( \im \delta ) \to H_{c-1} ( \ker \delta ) \to \ldots .
\end{eqnarray*}
Since $\Delta$ is acyclic, we know that $H_c \left( \Delta ,\mathbb{F}_p \right) = \{ 0
\}$, which clearly implies that both $H_c \left( \ker \sigma \right)$ and $H_c \left(
\ker \delta \right)$ are zero.  So the exact sequences take the form
\begin{eqnarray*}
\ldots \longrightarrow 0 \longrightarrow 0 \longrightarrow 0
\longrightarrow H_c ( \im \sigma )
\to H_{c-1} ( \ker \sigma ) \to \ldots  \\[6pt]
\ldots \longrightarrow 0 \longrightarrow 0 \longrightarrow 0 \longrightarrow H_c ( \im
\delta ) \to H_{c-1} ( \ker \delta ) \to \ldots
\end{eqnarray*}
But isomorphisms \eqref{homiso1} and \eqref{homiso2} imply that $H_c \left( \im \sigma
\right)$ and $H_c \left( \im \delta \right)$ are also zero.  So the exact sequences are
like so:
\begin{eqnarray*}
\ldots \longrightarrow 0 \longrightarrow 0 \longrightarrow 0
\longrightarrow 0
\longrightarrow H_{c-1} ( \ker \sigma ) \to \ldots  \\[6pt]
\ldots \longrightarrow 0 \longrightarrow 0 \longrightarrow 0 \longrightarrow 0
\longrightarrow H_{c-1} ( \ker \delta ) \to \ldots
\end{eqnarray*}
We can apply the same reasoning to show that all terms in the sequences with index
$(c-1)$ are likewise zero. Continuing in this manner, we eventually find that
\textit{all} the homology groups in the sequences that have a positive index are zero.
We are left with the exact sequences in the following form:
\begin{eqnarray}\label{finallongexactsequence}
\ldots \longrightarrow 0 \longrightarrow 0 \longrightarrow
H_0 \left( \ker \sigma \right) \to H_0 \left( \Delta , \mathbb{F}_p
\right) \to H_0 \left( \im \sigma \right) \longrightarrow 0\qquad\quad \\[6pt]
\ldots \longrightarrow 0 \longrightarrow 0 \longrightarrow H_0 \left( \ker \delta \right)
\to H_0 \left( \Delta , \mathbb{F}_p \right) \to H_0 \left( \im \delta \right)
\longrightarrow 0\qquad\quad
\end{eqnarray}

We have shown that all of the homology groups $H_n \left( \ker \sigma \right)$, $n > 0$
are trivial.  This implies by isomorphism~(\ref{homiso1}) that $H_n \left( \Delta^f ,
\mathbb{F}_p \right)$ is trivial for all $n > 0$.  Also, we know from
isomorphism~(\ref{homiso1}) and sequence~(\ref{finallongexactsequence}) that
\begin{eqnarray}
\dim H_0 \left( \Delta^f , \mathbb{F}_p \right) \leq \dim H_0 \left( \ker \sigma \right)
\leq \dim H_0 \left( \Delta , \mathbb{F}_p \right) = 1.
\end{eqnarray}
The dimension of $H_0 \left( \Delta^f , \mathbb{F}_p \right)$ cannot be zero (since
$\Delta^f$ is nonempty).  So $H_0 \left( \Delta^f , \mathbb{F}_p \right)$ must be
one-dimensional.  Therefore, $\Delta^f$ is\break $\mathbb{F}_p$-acyclic.
\end{proof}

\begin{corollary}\label{pgroupactioncorollary}
Suppose that $H$ is a group of order $p^m$, with $m \geq 1$, which acts on $\Delta$ in
such a way that $\Delta^h$ is a subcomplex of $\Delta$ for any $h \in H$.  Then,
$\Delta^H$ is $\mathbb{F}_p$-acyclic.
\end{corollary}

\begin{proof}
Since $\left| H \right| = p^m$, there exists a filtration of $H$ by normal subgroups,
\begin{eqnarray}
\{ 0 \} = H_0 \subset H_1 \subset \ldots \subset H_m = H
\end{eqnarray}
such that $H_i / H_{i-1} \cong \mathbb{Z} / p \mathbb{Z}$ for any $i \in \{ 1, 2, \ldots,
m \}$.  (See Chapter~I, Corollary~6.6 in \cite{lang}.)  For any $i \in \{ 1, 2, \ldots ,
m \}$, we can choose an element $a_i \in H_i$ which generates $H_i / H_{i-1}$.  Then,
\begin{eqnarray}
\Delta^{H_i} = \left( \Delta^{H_{i-1}} \right)^{a_i}.
\end{eqnarray}
By Proposition~\ref{acyclicitypreservation}, if $\Delta^{H_{i-1}}$ is
$\mathbb{F}_p$-acyclic, so is $\Delta^{H_i}$.  The corollary follows by induction.
\end{proof}

Now we are ready to prove the main theorem.

\begin{theorem}\label{nonabelianfixedpointtheorem}
Let $G$ be a finite group satisfying the following \hbox{condition}:
\begin{itemize}
\item There is a normal subgroup $G' \subseteq G$ such that
$\left| G' \right|$ is a prime power and $G / G'$ is cyclic.
\end{itemize}
Let $\Delta$ be a collapsible abstract simplicial complex on which $G$ acts, satisfying
the condition that $\Delta^g$ is a simplicial complex for any $g \in G$.  Then, $\chi (
\Delta^G ) = 1$.
\end{theorem}

\begin{proof}
We are given that $\left| G' \right| = p^m$ for some prime $p$ and $m \geq 0$.  Choose a
total ordering on the vertices of $\Delta$. By Theorem~\ref{acyclicitytheorem}, $\Delta$
is $\mathbb{F}_p$-acyclic. By Corollary~\ref{pgroupactioncorollary}, $\Delta^{G'}$ is
$\mathbb{F}_p$-acyclic.

Choose an element $b \in G$ which generates $G/G'$.  By Theorem~\ref{acyclicimplieslft},
the complex
\begin{eqnarray}
(\Delta^{G'})^b = \Delta^G
\end{eqnarray}
has Euler characteristic equal to $1$.
\end{proof}

Note that Theorem~\ref{nonabelianfixedpointtheorem} implies in particular that the
invariant subcomplex $\Delta^G$ is nonempty.

\section{Barycentric Subdivision}\label{barycentricsubdivisionsection}

In \textit{\nameref{nonabeliansection}} we proved
Theorem~\ref{nonabelianfixedpointtheorem}, which asserts that if a group action $G
\circlearrowleft \Delta$ satisfies certain requirements, then $\Delta^G$ must be
nonempty. The theorem as stated is unfortunately not general enough for our purposes.
Indeed the condition that all of the subsets $\{ \Delta^g \mid g \in G \}$ are
subcomplexes will not be satisfied by the simplicial complexes arising from graph
properties, except in trivial cases.  Therefore we need a theorem which can be applied to
group actions that do not satisfy this condition.

\begin{figure}[!b]
 \centerline{\includegraphics{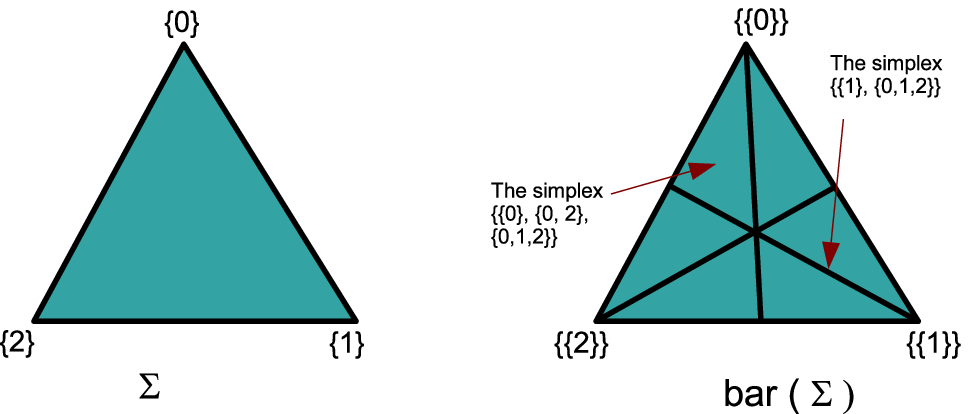}}
\fcaption{The complexes $\Sigma$ and $\bar ( \Sigma )$.\label{barycentricfig}}
\end{figure}

Barycentric subdivision is a process of dividing up the simplicies in a simplicial
complex into smaller simplicies.  Barycentric subdivision replaces an abstract simplicial
complex $\Delta$ with a larger complex $\Delta'$ that has similar properties.  The
advantage of this construction is that for any simplicial automorphism $g \colon \Delta
\to \Delta$, there is an induced automorphism $g \colon \Delta' \to \Delta'$ which
satisfies the condition that $(\Delta' )^g$ is an abstract simplicial complex. Working
within this larger complex will allow us to prove a generalization of
Theorem~\ref{nonabelianfixedpointtheorem}.

\begin{definition}\label{barycentricdef}
Let $\Delta$ be an abstract simplicial complex.  Then the \textbf{barycentric subdivision
of $\Delta$}, denoted $\bar ( \Delta )$, is the simplicial \hbox{complex}
\begin{eqnarray*}
\bar(\Delta ) = \big\{ \{ Q_1, Q_2, \ldots , Q_r \} \mid r \geq 1, Q_i \in \Delta , Q_1
\subset Q_2 \subset \ldots \subset Q_r \big\}.
\end{eqnarray*}
\end{definition}

Here is another way to phrase the above definition. Let $\Delta$ be an abstract
simplicial complex.  Then the subset relation $\subset$ gives a partial ordering on the
elements of $\Delta$.  The complex $\bar ( \Delta )$ is the set of all $\subset$-chains
in $\Delta$.

As an example, let
\begin{eqnarray}\label{baryexample}
\Sigma = \left\{ \{ 0 \} , \{ 1 \} , \{ 2 \} , \{ 0, 1 \} , \{ 1 ,2 \} , \{ 0, 2 \} , \{
0, 1, 2 \} \right\}\!.
\end{eqnarray}
The complex $\Sigma$ and its barycentric subdivision $\bar ( \Sigma )$ are shown in
Figure~\ref{barycentricfig}.

Geometrically, the operation $[ \Delta \mapsto \bar ( \Delta ) ]$ has the effect of
splitting every simplex of dimension $n$ in $\Delta$ into $(n+1)!$ simplicies of
\hbox{dimension~$n$}. Note that vertices in $\bar ( \Delta )$ are in one-to-one
correspondence with the simplicies of $\Delta$.  Figure~\ref{barycentric2fig} shows
another example of barycentric subdivision.

\begin{figure}
 \centerline{\includegraphics{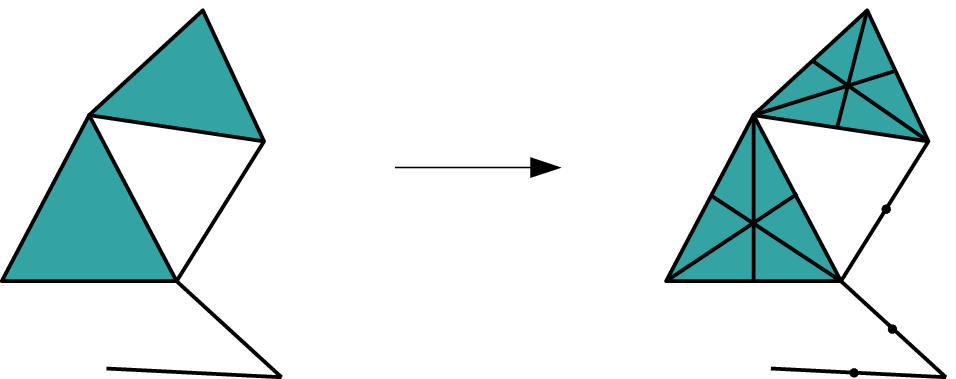}}
\fcaption{An example of barycentric subdivision.\label{barycentric2fig}}
\end{figure}

As the reader can observe, the simplicial complex $\bar ( \Delta )$ has some similarities
with the original simplicial complex $\Delta$. It can be shown that the
homology groups of $\bar ( \Delta )$ are isomorphic to those of $\Delta$, although we
will not need to prove that here.  The following propositions are proved in the Appendix
(as Proposition~\ref{baryeulerappendixprop} and Proposition~\ref{baryappendixprop}).

\begin{proposition}\label{barycollapseprop}
Let $\Delta$ be an abstract simplicial complex. If $\Delta$ is collapsible, the $\bar (
\Delta )$ is also collapsible.
\end{proposition}

\begin{proposition}\label{baryeulerprop}
Let $\Delta$ be a finite abstract simplicial complex. Then, $\chi ( \bar ( \Delta ) ) =
\chi ( \Delta )$.
\end{proposition}

Now, let us consider how this construction behaves under group actions. Let $f \colon
\Delta \to \Delta$ be a simplicial automorphism of $\Delta$.  Then there is an induced
simplicial automorphism,
\begin{eqnarray}
f \colon \bar ( \Delta ) \to \bar ( \Delta ).
\end{eqnarray}
The invariant subset $\bar ( \Delta )^f$ can be expressed like so:
\begin{eqnarray*}
\bar ( \Delta )^f  = \big\{\{ Q_1, Q_2, \ldots , Q_r \} \mid r \geq 1, Q_i \in \Delta^f ,
Q_1 \subset Q_2 \subset \ldots \subset Q_r \big\}.
\end{eqnarray*}

 \pagebreak

\noindent It is easy to see that this set is always a simplicial complex.  Thus the
following lemma holds true:
\begin{lemma}\label{barylemma}
Let $\Delta$ be an abstract simplicial complex, and let
$G \circlearrowleft \Delta$ be a group action on $\Delta$.  Then,
for any $g \in G$, the set
\begin{eqnarray}
\left( \bar ( \Delta ) \right)^g
\end{eqnarray}
is a subcomplex of $\bar ( \Delta )$.
\end{lemma}
Lemma~\ref{barylemma} can be observed in the example complex $\Sigma$ which we discussed
above~\eqref{baryexample}. As we can see in Figure~\ref{barycentricfig}, any permutation
of the set $\{ 0, 1, 2 \}$ fixes a subcomplex of the complex $\bar ( \Sigma)$.

With the aid of barycentric subdivision, we can now prove the following fixed-point
theorem.

\begin{theorem}\label{fptinit}
Let $\Delta$ be a collapsible abstract simplicial complex.  Let $G \circlearrowleft
\Delta$ be a group action on $\Delta$.  Suppose that $G$ has a normal subgroup $G'$ which
is such that $\left| G' \right|$ is a prime power and $G / G'$ is cyclic.  Then, the set
$\Delta^G$ is nonempty.
\end{theorem}

\begin{proof}
By Proposition~\ref{barycollapseprop}, $\bar ( \Delta )$ is collapsible. By
Theorem~\ref{nonabelianfixedpointtheorem}, $\chi ( \bar ( \Delta )^G ) = 1$.  Therefore
$\bar ( \Delta )^G$ is nonempty, and thus $\Delta^G$ is likewise nonempty.
\end{proof}

Now let $\Delta^{[G]}$ denote the complex constructed in
\textit{\nameref{groupactionsection}}. The set $\Delta^{[G]}$ is very similar to
$\Delta^G$; indeed, there is a one-to-one inclusion preserving map
\begin{eqnarray}\label{naturaliso}
i \colon \Delta^{[G]} \to \Delta^G
\end{eqnarray}
which is given simply by mapping any $S \in \Delta^{[G]}$ to the union of the elements of
$S$. (The main difference between $\Delta^G$ and $\Delta^{[G]}$ is that $\Delta^{[G]}$ is
a simplicial complex, whereas $\Delta^G$ generally is not.)

The map~(\ref{naturaliso}) induces a simplicial isomorphism
\begin{eqnarray}\label{barygroupiso}
\bar \left( \Delta^{[G]} \right) \to
\bar ( \Delta )^G
\end{eqnarray}
Figure~\ref{barygroupfig} illustrates the relationship between $\Delta^{[G]}$ and $\bar (
\Delta )^G$. \hbox{Isomorphism}~(\ref{barygroupiso}) enables our final generalization of
\hbox{Theorem}~\ref{nonabelianfixedpointtheorem}.

\begin{figure}[!t]
 \centerline{\includegraphics{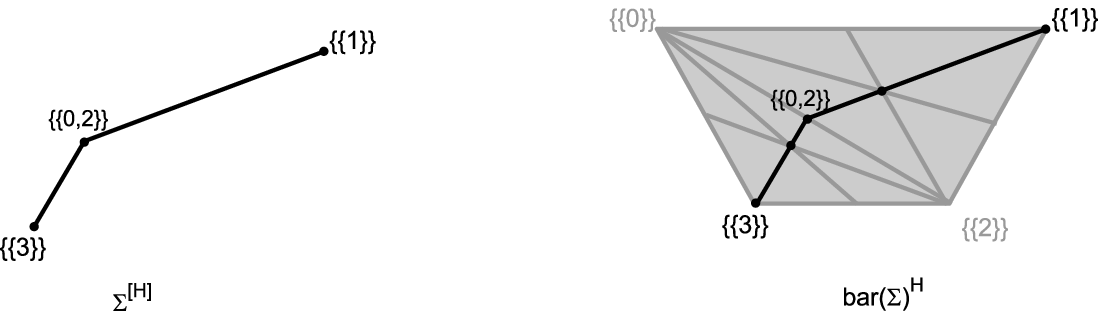}}
\fcaption{A continuation of the example from Figures~\ref{groupactionfig}
and \ref{groupaction2fig}. The barycentric subdivision of $\Sigma^{[H]}$ is
isomorphic to $\bar ( \Sigma )^H$.\label{barygroupfig}}
\end{figure}

\begin{theorem}\label{fpt}
Let $\Delta$ be a collapsible abstract simplicial complex. Let $G \circlearrowleft
\Delta$ be a group action on $\Delta$. Suppose that $G$ has a normal subgroup $G'$ which
is such that $\left| G' \right|$ is a prime power and $G / G'$ is cyclic.  Then,
\begin{eqnarray}
\chi ( \Delta^{[G]} ) = 1.
\end{eqnarray}
\end{theorem}

\begin{proof}
By Proposition~\ref{barycollapseprop}, $\bar ( \Delta )$ is collapsible. Therefore the
Euler characteristic of $\bar ( \Delta )^G \cong \bar ( \Delta^{[G]} )$ is $1$.  By
Proposition~\ref{baryeulerprop}, the Euler characteristic of $\Delta^{[G]}$ is likewise
equal to $1$.
\end{proof}

\chapter[Results on Decision-Tree Complexity]{Results on Decision-Tree Complexity}\label{resultschapter}

In this part of the text, we will give the proofs of three lower bounds on the
decision-tree complexity of graph properties, due to Kahn, Saks, Sturtevant, and Yao.
Then we will sketch (without proof) some more recent results.

Let
\begin{eqnarray}
h \colon \mathbf{G} ( V ) \to \{ 0, 1 \}
\end{eqnarray}
be a nontrivial monotone-increasing graph property.  The function $h$ satisfies two
conditions: it is increasing (meaning that if $Z$ is a subgraph of $Z'$ then $h ( Z )
\leq h ( Z' )$) and it is also isomorphism-invariant ($Y \cong Y' \Longrightarrow h ( Y )
= h ( Y' )$).  Proofs of evasiveness exploit the interaction between these two
conditions.

As we saw in \textit{\nameref{basicconceptschapter}}, the monotone-increasing condition
implies that $h$ determines a simplicial complex, $\Delta_h$, whose \hbox{simplices}
correspond to graphs $Z$ that satisfy $h ( Z ) = 0$.  The isomorphism-invariant property
implies that this complex $\Delta_h$ is highly symmetric. If $\sigma$ is any permutation
of $V$, and
\begin{eqnarray}
E \subseteq \left\{ \{ v, w \} \mid v, w \in V \right\}
\end{eqnarray}
is an edge set such that $h ( ( V, E ) ) = 0$, then the edge set
\begin{equation}
\sigma ( E ) = \left\{ \{ \sigma ( v ) , \sigma ( w ) \} \mid \{ v , w \} \in E \right\}
\end{equation}
also satisfies $h ( (V , \sigma ( E )) ) = 0$.  Thus there is an induced automorphism
$\sigma \colon \Delta_h \to \Delta_h$.

If $h$ were nonevasive, then $\Delta_h$ would be collapsible, and we could apply
fixed-point theorems to $\Delta_h$.  Corollary~\ref{lefschetzcorollary} would imply that
$\Delta_h$ must have a simplex which is stabilized by $\sigma$.  Therefore, we have the
following interesting result: if $h$ is a nonevasive graph property, then for any
permutation $\sigma \colon V \to V$ there must be a nontrivial $\sigma$-invariant graph
which does not satisfy $h$.  Figure~\ref{invariantgraphsfig} shows  what we can deduce
when $\left| V \right| = 9$ and $\sigma$ is chosen to be a cyclic permutation.

\begin{figure}[!t]
 \centerline{\includegraphics{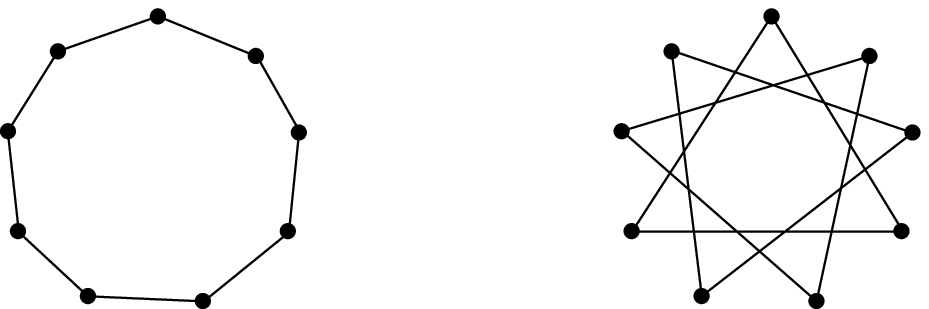}}
\fcaption{Let $V$ be a set of size $9$, and let $h$ be a nontrivial increasing graph
property.  If $h$ is not evasive, then at least one of the graphs above must fail to
satisfy $h$.\label{invariantgraphsfig}}
 \vspace*{-3pt}
\end{figure}

 \enlargethispage{6pt}

When we go further and consider the actions of finite groups on $\Delta_h$, we get
stronger results. Note that the entire symmetric group $\Sym ( V )$ acts on $\Delta_h$.
Unfortunately this group is too big for the application of any fixed-point theorems that
we have proved, and so we must restrict the action to some appropriate subgroup of $\Sym
( V )$. Making this choice of subgroup is a key step for many of the results that we will
discuss.

\section{Graphs of Order $p^k$}\label{mainresultsection}

\begin{theorem}[Kahn et~al. \cite{kss1984}]
Let $V$ be a finite set of order $p^k$, where $p$ is prime and $k \geq 1$.  Let
\begin{eqnarray}
h \colon \mathbf{G} (V) \to \{ 0, 1 \}
\end{eqnarray}
be a nontrivial monotone-increasing graph property.  Then, $h$ must be evasive.
\end{theorem}

 \removelastskip\pagebreak

\begin{proof}
Without loss of generality, we may assume that $V$ is the set of elements of the finite
field $\mathbb{F}_{p^k}$.  For any $a, b \in \mathbb{F}_{p^k}$ with $a \neq 0$, there is
a permutation of $V$ given by
\begin{eqnarray}
x \mapsto ax + b.
\end{eqnarray}
Let $G \subseteq \textnormal{Sym} ( V )$ be the group of all such permutations.  Let $G'
\subseteq G$ be the subgroup consisting of permutations of the form $x \mapsto x + b$.

We make the following observations:
\begin{enumerate}
\item \textbf{The subgroup $G'$ is an abelian group of
order $p^k$.}  It is  isomorphic to the additive group of $\mathbb{F}_{p^k}$.
\item \textbf{The subgroup $G'$ is normal.}
This is apparent from the fact that for any $x, a, b \in \mathbb{F}_{p^k}$, with $a \neq
0$,
\begin{eqnarray}
a^{-1} ( a x  + b )  = x + a^{-1}  b.
\end{eqnarray}
\item \textbf{The quotient group $G / G'$ is cyclic.}
The quotient group $G / G'$ is isomorphic to the multiplicative group of
$\mathbb{F}_{p^k}$, which is known to be cyclic (see Theorem IV.1.9 from \cite{lang}).
\item \label{transitivepairs} \textbf{The action of $G$ is transitive on
pairs of distinct elements $( x, x' ) \in V \times V$.} This is a consequence of the fact
that for any pairs $(x, x')$ and $(y, y')$ with $x \neq x'$ and $y \neq y'$, the system
of equations
\begin{eqnarray}
ax + b & = & y \\
a x' + b & = & y'
\end{eqnarray}
has a solution, with $a \neq 0$.
\end{enumerate}

Consider the group action
\begin{eqnarray}
G \circlearrowleft \Delta_h
\end{eqnarray}
Suppose that the graph property $h$ is nonevasive. By
Theorem~\ref{collapsibilitytheorem}, the simplicial complex $\Delta_h$ is
collapsible.\footnote{Technically, this is not true if $\Delta_h$ is empty, and so we
need to address that case separately.  If $\Delta_h$ is empty, then $h$ must be the
function that maps the empty graph to zero and all other graphs to $1$.  This graph
property is easily seen to be evasive.}  By Theorem~\ref{fptinit}, the set~$\left(
\Delta_h \right)^G$ is nonempty.  Therefore there is a nonempty $G$-invariant graph which
does not satisfy $h$.  But by property (\ref{transitivepairs}) above, the only nonempty
\hbox{$G$-invariant} graph is the complete graph. This makes $h$ a trivial graph
property, and thus we obtain a contradiction.

We conclude that $h$ must be an evasive graph property.
\end{proof}

\section{Bipartite Graphs}

Let $V$ be a finite set which is the disjoint union of two subsets, $Y$ and~$Z$. Then a
bipartite graph on $(Y, Z)$ is a graph whose edges are all elements of the set
\begin{eqnarray}\label{bipartiteedgepairs}
\left\{ \{ y, z \} \mid y \in Y, z \in Z \right\}\!.
\end{eqnarray}
A bipartite isomorphism between such graphs is a graph isomorphism which respects the
partition $(Y, Z)$.

Let $\mathbf{B} ( Y, Z)$ denote the set of all bipartite graphs on $(Y, Z)$. A
\textbf{bipartite graph property} is a function
\begin{eqnarray}\label{examplebipartiteprop}
f \colon \mathbf{B} ( Y, Z ) \to \{ 0, 1 \}
\end{eqnarray}
which respects bipartite isomorphisms. If this function is monotone increasing, it
determines a simplicial complex $\Delta_f$ whose vertices are elements of the
set~(\ref{bipartiteedgepairs}).

Naturally, we say that the bipartite graph property~(\ref{examplebipartiteprop}) is
evasive if its decision-tree complexity $D(f)$ is equal to $\left| Y \right| \cdot \left|
Z \right|$.  The following proposition can be proved by the same method that we used to
prove Theorem~\ref{collapsibilitytheorem}.

\begin{proposition}
Let $Y$ and $Z$ be disjoint finite sets, and let
\begin{eqnarray}
f \colon \mathbf{B} ( Y, Z ) \to \{ 0, 1 \}
\end{eqnarray}
be a monotone-increasing bipartite graph property which is not evasive.  If the complex
$\Delta_f$ is not empty, then it is collapsible.
\end{proposition}

Note that the complex $\Delta_f$ always has a group action,
\begin{eqnarray}
\left( \Sym ( Y ) \times \Sym (Z )  \right) \circlearrowleft \Delta_f.
\end{eqnarray}

{\makeatletter
\newtheoremstyle{nowthm}{4pt plus6pt minus4pt}{0pt}{\upshape}{0pt}{\bfseries}{}{.6em}
  {\rule{\textwidth}{.5pt}\par\vspace*{-1pt}\newline\thmname{#1}\thmnumber{\@ifnotempty{#1}{\hspace*{3.65pt}}{#2}$\!\!$}
  \thmnote{{\the\thm@notefont\bf (#3).}}}
\def\@endtheorem{\par\vspace*{-7.8pt}\noindent\rule{\textwidth}{.5pt}\vskip8pt plus6pt minus4pt}
\ignorespaces \makeatother

\begin{theorem}[Yao \cite{yao1988}]
Let $Y$ and $Z$ be disjoint finite sets, and let
\begin{eqnarray}
f \colon \mathbf{B} ( Y, Z ) \to \{ 0, 1 \}
\end{eqnarray}
be a nontrivial bipartite graph property which is monotone increasing.  Then, $f$ is
evasive.
\end{theorem}}

\begin{proof}
Let $\sigma \colon Y \to Y$ be a cyclic permutation of the elements of $Y$, and let $G
\subseteq \Sym ( Y )$ be the subgroup generated by $\sigma$. The edge set of any
$G$-invariant bipartite graph has the form
\begin{eqnarray}
H_S := \left\{  \{ y, z \} \mid y \in Y, z \in S \right\}
\end{eqnarray}
where $S$ is a subset of $Z$ (see Figure~\ref{invariantgraphfig}).  Since $f$ is
isomorphism-invariant and monotone-increasing, the behavior of $f$ on such graphs can be
easily described: there is some integer $k \in \{ 1, 2, \ldots, \left| Z \right| \}$ such
that
\begin{eqnarray}
( V , H_S ) \textnormal{ has property $f$} \Longleftrightarrow \left| S \right| > k.
\end{eqnarray}

\begin{figure}[!b]
 \centerline{\includegraphics{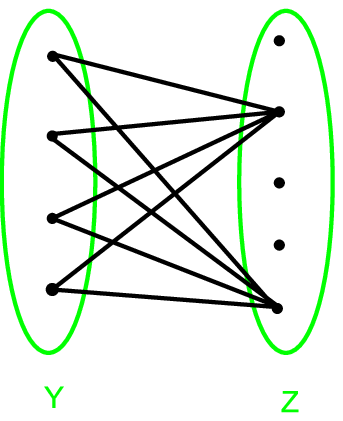}}
\fcaption{An example of a set $H_S$.\label{invariantgraphfig}}
\end{figure}

Let $\Delta = \Delta^{[G]}$.  The vertices of $\Delta^{[G]}$ are the sets of the form
\begin{eqnarray}
H_z := \left\{ \{ y, z \} \mid y \in Y \right\}\!,
\end{eqnarray}
with $z \in Z$, and the simplicies are precisely the subsets of $\{ H_z \mid z \in Z \}$
whose union forms a graph that does not have property $f$.  Thus we can calculate the
Euler characteristic directly:
\begin{eqnarray}
\chi ( \Delta^{[G]} )
& = & \sum_{j = 0}^{k-1} (-1)^j \binom{\left| Z \right|}{ j + 1} \\
& = & 1 + (-1)^{k-1} \binom{ \left| Z \right| - 1}{k}.
\end{eqnarray}

Suppose that $f$ is nonevasive.  Then $\Delta$ is collapsible and by Theorem~\ref{fpt},
\begin{eqnarray}
\chi ( \Delta^{[G]} ) & = & 1.
\end{eqnarray}
But this is possible only if $k = \left| Z \right|$ and $f$ is trivial.
\end{proof}

\section{A General Lower Bound}

Now we prove a lower bound on decision-tree complexity which applies to graphs of
arbitrary size. Our method of proof is based on \cite{kss1984}.

\begin{proposition}\label{generallowerbound}
Let $V$ be a finite set and let
\begin{eqnarray}
h \colon \mathbf{G} ( V ) \to \{ 0, 1 \}
\end{eqnarray}
be a nontrivial monotone-increasing graph property. Let $p$ be the largest prime that is
less than or equal to $\left| V \right|$.  Then,
\begin{eqnarray}
D ( h ) \geq \frac{p^2}{4}.
\end{eqnarray}
\end{proposition}

\begin{proof}
Assume that $\left| V \right| = n$. For any $r, s \geq 0$, let us write $K_r$ for the
complete graph on $\{ 1, 2, \ldots, r \}$, and let us write $K_{r, s}$ for the complete
bipartite graph on the sets $\{ 1, 2, \ldots, r \}$ and $\{ r+1 , \dots , r + s \}$.  For
any two graphs $H = (V, E)$ and $H' = (V' , E' )$, let us abuse notation slightly and
write $H \cup H'$ for the graph $(V \cup V', E \cup E')$.

For any $k \geq 1$, let $C_k$ denote the least decision-tree complexity that occurs for
nontrivial monotone-increasing graph properties on graphs of size $k$. We prove a lower
bound for $D ( h )$ in three cases.

\medskip\textbf{Case 1:} $\mathbf{h ( K_{1,n-1} ) = 0}.$ In this case, the function $h$
induces a nontrivial graph property $h'$ on the vertex set $\{ 2, 3, \ldots, n \}$, given
by
\begin{eqnarray}
h' ( P ) & = & h ( P \cup K_{1, n-1} ).
\end{eqnarray}
This function has decision-tree complexity at
least $C_{n-1}$, and therefore $D ( h ) \geq C_{n-1}$.

\medskip
\textbf{Case 2:} $\mathbf{h ( K_{n-1} ) = 1.}$ In this case the function $h$ induces a
nontrivial graph property $h'$ on the vertex set $\{ 2, 3, \ldots, n \}$ given by
\begin{eqnarray}
h' ( P ) & = & h ( P \cup K_1 ),
\end{eqnarray}
which is likewise nontrivial.  This function has decision-tree complexity at least
$C_{n-1}$, and so $D ( h ) \geq C_{n-1}$.

\medskip
\textbf{Case 3:} $\mathbf{h ( K_{1, n-1} ) = 1}$ \textbf{and} $\mathbf{h ( K_{n-1} ) =
0}$. Let $m = \lfloor n/2 \rfloor$.  The property $h$ induces a bipartite graph property
on the sets $\{1, 2, \ldots, m \}$ and $\{ m+1, m+2, \ldots, m \}$ defined by
\begin{eqnarray}
h' ( P ) = h ( P \cup K_m ).
\end{eqnarray}
Since $h ( K_m ) \leq h ( K_{n-1} ) = 0$ and $h ( K_m \cup K_{m,n-m} ) \geq h ( K_{1,
n-1} ) = 1$, the property $h'$ is nontrivial.  Therefore it has decision-tree complexity
at least $m (n-m)$.  The decision-tree complexity of $h$ is likewise bounded by $m ( n-m
) \geq (n-1)^2/4$.

\medskip
In all cases, we have
\begin{eqnarray}
D ( h ) \geq  \min \left\{ C_{n-1}, \frac{(n-1)^2}{4} \right\}\!.
\end{eqnarray}
The same reasoning shows that
\begin{eqnarray}
C_k \geq \min \left\{ C_{k-1} , \frac{(k-1)^2}{4} \right\}
\end{eqnarray}
for every $k \in \{ p+1, p+2 , \ldots, n-1 \}$.  Therefore
by induction,
\begin{eqnarray}
D ( h ) \geq \min \left\{ C_p , \frac{p^2}{4} \right\}\!.
\end{eqnarray}
The quantity $C_p$ is $\binom{p}{2}$, and the desired result follows.
\end{proof}

{\makeatletter
\newtheoremstyle{nowthm}{4pt plus6pt minus4pt}{0pt}{\upshape}{0pt}{\bfseries}{}{.6em}
  {\rule{\textwidth}{.5pt}\par\vspace*{-1pt}\newline\thmname{#1}\thmnumber{\@ifnotempty{#1}{\hspace*{3.65pt}}{#2}$\!\!$}
  \thmnote{{\the\thm@notefont\bf (#3).}}}
\def\@endtheorem{\par\vspace*{-7.8pt}\noindent\rule{\textwidth}{.5pt}\vskip8pt plus6pt minus4pt}
\ignorespaces \makeatother

\begin{theorem}[Kahn et~al. \cite{kss1984}]
\label{lowerboundtheorem} Let $C_n$ denote the least decision-tree complexity that occurs
among all nontrivial monotone-increasing graph properties of order $n$.  Then,
\begin{eqnarray}
C_n \geq \frac{n^2}{4} - o ( n^2 ).
\end{eqnarray}
\end{theorem}}

\begin{proof}
By the prime number theorem, there is a function $z ( n ) = o ( n )$ such that for any
$n$, the interval $[n -  z(n), n]$ contains a prime.\footnote{The prime number theorem
\cite{zagier1997} asserts that if $\pi( n )$ denotes the number of primes less than or
equal to $n$, then $\lim_{n \to \infty} \pi ( n ) \left( n / \ln n \right)^{-1} = 1$. If
there were an infinite number of linearly sized gaps between the primes, this limit could
not exist.} By Proposition~\ref{generallowerbound},
\begin{eqnarray}
C_n & \geq & \frac{(n - z(n))^2}{4} \\
& \geq & \frac{n^2}{4} - o ( n^2 ).
\end{eqnarray}
as desired.
\end{proof}

\section{A Survey of Related Results}

Much work on the decision-tree complexity of graph properties has followed the papers of
Kahn, Saks, Sturtevant, and Yao. We briefly sketch some of the newer results in this
area.

V. King proved a lower bound for properties of \textbf{directed} graphs.

{\makeatletter
\newtheoremstyle{nowthm}{4pt plus6pt minus4pt}{0pt}{\upshape}{0pt}{\bfseries}{}{.6em}
  {\rule{\textwidth}{.5pt}\par\vspace*{-1pt}\newline\thmname{#1}\thmnumber{\@ifnotempty{#1}{\hspace*{3.65pt}}{#2}$\!\!$}
  \thmnote{{\the\thm@notefont\bf (#3).}}}
\def\@endtheorem{\par\vspace*{-7.8pt}\noindent\rule{\textwidth}{.5pt}\vskip8pt plus6pt minus4pt}
\ignorespaces \makeatother

\begin{theorem}[King \cite{king1990}]
Let $C'_n$ denote the least decision-tree complexity
that occurs among all nontrivial monotone
\textbf{directed} graph properties of order $n$.  Then,
\begin{eqnarray}
C'_n \geq \frac{n^2}{2} - o ( n^2 ).
\end{eqnarray}
\end{theorem}}

\noindent Triesch \cite{triesch1994, triesch1996} proved multiple results about the
evasiveness of particular subclasses of monotone graph properties.

Korneffel and Triesch improved on the asymptotic bound of
\hbox{Theorem}~\ref{lowerboundtheorem} by using a different group action on the set of
vertices.  Let $V$ be a set of size $n$, and let $p$ be a prime that is close to $\left(
\frac{2}{5} \right) n$. Break the set $V$ up into disjoint subsets $V_1$, $V_2$, and
$V_3$, with $\left| V_1 \right| = \left| V_2 \right| = p$ and $\left| V_3 \right| = n -
2p$.   Let $\mathbf{P}$ be the class of tripartite graphs on $(V_1, V_2, V_3)$ which,
when taken together with the complete graphs on the sets $V_i$, do not satisfy property
$h$.  The abelian group
\begin{eqnarray}
G = \mathbb{Z} / p \mathbb{Z} \times \mathbb{Z} / p \mathbb{Z}
\times \mathbb{Z} / (n - 2p ) \mathbb{Z}
\end{eqnarray}
acts on the class $\mathbf{P}$ by cyclically permuting the elements of $V_1$, $V_2$,
and~$V_3$. From this action and some other arguments, the authors are able to prove the
following.

{\makeatletter
\newtheoremstyle{nowthm}{4pt plus6pt minus4pt}{0pt}{\upshape}{0pt}{\bfseries}{}{.6em}
  {\rule{\textwidth}{.5pt}\par\vspace*{-1pt}\newline\thmname{#1}\thmnumber{\@ifnotempty{#1}{\hspace*{3.65pt}}{#2}$\!\!$}
  \thmnote{{\the\thm@notefont\bf (#3).}}}
\def\@endtheorem{\par\vspace*{-7.8pt}\noindent\rule{\textwidth}{.5pt}\vskip8pt plus6pt minus4pt}
\ignorespaces \makeatother

\begin{theorem}[Korneffel and Triesch \cite{kt2010}]
Let $C_n$ denote the least decision-tree complexity that occurs among all nontrivial
monotone-increasing graph properties of order~$n$.  Then,
\begin{eqnarray}
C_n \geq \frac{8 n^2}{25} - o ( n^2 ). \hskip0.2in
\end{eqnarray}
\end{theorem}}

The work of Chakrabati et~al. \cite{cks2002} considers the
\textbf{subgraph containment property}.  For any finite graph $X$, let $h_{X,n}$ denote
the graph property for graphs of size $n$ which assigns a value of $1$ to a graph if and
only if it contains a subgraph isomorphic to $X$. This property is studied using another
group action. For appropriate values of $n$, the vertex set $V$ can be partitioned into
sets $V_1, \ldots , V_m$, where $\left| V_i \right| = q^{\alpha_i}$ for some prime power
$q$ which is greater than or equal to the number of vertices in $X$. Choose isomorphisms
$V_i \cong \mathbb{F}_{q^{\alpha_i}}$. Let $G$ be the group of permutations of $V$ that
is generated by the group $\mathbb{F}_{q^{\alpha_1}}^+ \times \ldots \times
\mathbb{F}_{q^{\alpha_m}}^+$ (acting on the sets $V_1 , \ldots , V_m$ in a component-wise
manner) and the group $\mathbb{F}_q^*$ (acting simultaneously on all the sets
$V_i$). If $h_X$ were nonevasive, then there would exist nontrivial $G$-invariant graphs
which do not satisfy $h_X$.  Such graphs would have a uniform structure and would
correspond simply to graphs on the set $\{ 1, 2, \ldots, m \}$.

With this reduction the authors are able to prove that $h_{X,n}$ is evasive for all $n$
within a set of positive density.  In general, the following asymptotic bound holds:
\begin{eqnarray}
D ( h_{X,n} ) \geq \frac{n^2}{2} - O ( n ).
\end{eqnarray}
This approach was further developed by Babai et al. \cite{bbkk2010}, who proved that
$h_{X,n}$ is evasive for almost all $n$, and that
\begin{eqnarray}
D ( h_{X, n } ) \geq \binom{n}{2} - O ( 1  ).
\end{eqnarray}

As one can observe from recent papers on evasiveness, advances in the strength of results
are paralleled by substantial increases in the difficulty of the proofs!    The
increase in difficulty has become fairly steep at this point.  Perhaps a new basic
insight, like the one in \cite{kss1984},  will be necessary to proceed further toward the
Karp conjecture.

\appendix

\chapter{Appendix}

\section{Long Exact Sequences of Homology Groups}\label{snakelemmaappendix}

The goal of this part of the appendix is to give a complete proof of the following
proposition. Our method is based on \cite{ash2007}.

\begin{proposition}\label{realsnakelemma}
Suppose that there is an exact sequence of complexes:
\begin{eqnarray}\label{bigcplxdiagram}
\scalebox{0.75}{\xymatrix{
&  0 \ar[d] & 0 \ar[d] & 0 \ar[d] &  \\
0 \ar[r] & X_m \ar[r]^F \ar[d]^d  & Y_m \ar[r]^G \ar[d]^d & Z_m \ar[r] \ar[d]^d & 0 \\
0 \ar[r] & X_{m-1} \ar[r]^F \ar[d]  & Y_{m-1} \ar[r]^G \ar[d] & Z_{m-1} \ar[r] \ar[d] & 0 \\
&  \vdots \ar[d] & \vdots \ar[d] & \vdots \ar[d] &  \\
0 \ar[r] & X_0 \ar[r]^F \ar[d]  & Y_0 \ar[r]^G \ar[d] & Z_0 \ar[r] \ar[d] & 0 \\
&  0 & 0 & 0 &  \\
}
}
\end{eqnarray}
Then, there exist homomorphisms $\gamma_n \colon H_n ( Z_\bullet ) \to H_{n-1} (
X_\bullet )$ for \hbox{$n \in \{ 1, 2, \ldots, m \}$} such that the sequence
\begin{eqnarray}\label{snakesequence}
\scalebox{0.75}{\xymatrix{ 0 \ar[r] & H_m ( X_\bullet ) \ar[r] & H_m ( Y_\bullet ) \ar[r] &
H_m ( Z_\bullet ) \ar[lldd]^{\gamma_m} \\
\\
& H_{m-1} ( X_\bullet ) \ar[r] & H_{m-1} ( Y_\bullet)
\ar[r] & H_{m-1} ( Z_\bullet) \ar[lldd]^{\gamma_{m-1}} \\
\\
& &  \\
& & \vdots \\
& &  & \ar[lldd]^{\gamma_1} \\
\\
 & H_0 ( X_\bullet ) \ar[r] & H_0 ( Y_\bullet ) \ar[r] &
H_0 ( Z_\bullet ) \ar[r] & 0\\
}}
\end{eqnarray}
is exact.
\end{proposition}

We begin by addressing the case in which $m = 1$.  Suppose that we have an exact sequence
of complexes
\begin{eqnarray}
\xymatrix{
& 0 \ar[d] & 0 \ar[d] & 0 \ar[d] \\
0 \ar[r] & X_1 \ar[r]^{F_1} \ar[d]^{d^X}  & Y_1 \ar[r]^{G_1} \ar[d]^{d^Y} & Z_1 \ar[r] \ar[d]^{d^Z} & 0 \\
0 \ar[r] & X_0  \ar[d] \ar[r]^{F_0}  & Y_0  \ar[d] \ar[r]^{G_0} & Z_0  \ar[d] \ar[r]  & 0 \\
& 0 & 0 & 0}
\end{eqnarray}
Then, $H_1 ( X_\bullet )$ is equal to $\ker d^X$, and $H_0 ( X_\bullet )$ is equal to the
group $X_0 / \im d^X$, which we denote by $\coker d^X$.   (The latter group is called the
``cokernel'' of $d^X$.)  Similar statements hold for $Y_\bullet$ and $Z_\bullet$.

Define a function
\begin{eqnarray}
\gamma \colon \ker d^Z \to \coker d^X
\end{eqnarray}
as follows.  Suppose that $z_1$ is an element of $\ker d^Z$.   Choose an element $y_1 \in
Y_1$ which maps to $z_1$.  The element $dy_1 \in Y_0$ maps to zero under $G_0$, and thus
we can find a (unique) element $x_0 \in X_0$ which maps to $dy_1$.  Let $\gamma ( x ) \in
\coker d^X$ be the coset containing $x_0$.

Note that the value of $\gamma ( x )$ does not depend on the choice of $y_1$, since if we
choose a different element $\overline{y}_1$ and obtain a different element
$\overline{x}_0 \in X_0$, then we will have $\overline{y}_1 - y_1 = F_1 ( x_1 )$ for some
$x_1$, and thus $\overline{x}_0 - x_0 = d x_1$, and $\overline{x}_0$ and $x_0$ will lie
in the same coset of $\coker d^X$. Note also that $\gamma ( x )$ is a homomorphism: if
$z_1 = z'_1 + z''_1$, then for any chosen pre-images $y_1$, $y'_1$, and $y''_1$, the
quantities $y_1$ and $(y'_1 + y''_1 )$ will differ by an element of $\im F_1$, and this
difference will likewise vanish when we map to $\coker d^X$.

Consider the sequence
\begin{eqnarray}\label{snakeseq}
\xymatrix{ 0 \ar[r] & \ker d^X \ar[r] & \ker d^Y \ar[r] &
\ker d^Z \ar[lldd]^{\gamma} \\
\\
& \coker d^X \ar[r] & \coker d^Y
\ar[r] & \coker d^Z \ar[r] & 0\\
}
\end{eqnarray}
It is easy to see that this sequence is a complex.  (The verification of this is left to
the reader.)  We wish to prove that the sequence is in fact exact.  We do this in six
steps.
\begin{enumerate}
\item \textbf{Exactness at $\ker d^X$.}  Immediate.
\item \textbf{Exactness at $\ker d^Y$.}  Suppose that
$y_1 \in \ker d^Y$ is an element that is killed by the map to $\ker d^Z$.  Then, there
exists an element $x_1 \in X_1$ which maps to $y_1$.  We must have $d x_1 = 0$ (since
otherwise $y_1$ could not be in the kernel of $d^Y$) and so $y_1$ lies in the image of
$\ker d^X$.
\item \textbf{Exactness at $\ker d^Z$.}  Suppose
that $z_1 \in \ker d^Z$ is such that $\gamma ( z_1 ) = 0$. Then, if we let $y_1$ and
$x_0$ be the elements chosen in the definition of $\gamma$, we must have $x_0 = d x_1$
for some \hbox{$x_1 \in X_1$}. The element $y'_1 := y_1 - F ( x_1 )$ maps to $z_1$ and
\hbox{satisfies} \hbox{$d y'_1 = 0$}. Therefore, $z_1$ is in the image of $[ \ker d^Y \to
\ker d^Z ]$.
\item \textbf{Exactness at $\coker d^X$.}  Suppose
that a coset of the form $x_0 + \im d^X \in \coker d^X$ maps to zero in $\coker d^Y$.
Then, there exists $y_1 \in Y_1$ such that $d y_1 = F ( x_0 )$.  Therefore the element
$z_1 := G ( y_1)$ maps to $(x_0 + \im d^X)$ under $\gamma$.
\item \textbf{Exactness at $\coker d^Y$.}  Suppose
that a coset of the form $y_0 + \im d^Y \in \coker d^Y$ maps to zero in $\coker d^Z$.
Then, there exists $z_1 \in Z_1$ such that $d z_1 = G ( y_0 )$.  If we let $y_1 \in Y_1$
be an element which maps to $z_1$, then we have $y_0 - d y_1 = F ( x_0 )$ for some $x_0$.
Therefore the coset $(x_0 + \im d^X) \in \coker d^X$ maps to $(y_0 + \im d^Y )$.
\item \textbf{Exactness at $\coker d^Z$.}  Immediate.
\end{enumerate}
We conclude that sequence~(\ref{snakeseq}) is indeed exact.

To prove Proposition~\ref{realsnakelemma} in general. we will need the following lemma,
which is a modified version of what we just proved.

\begin{lemma}\label{simplesnakelemma}
Suppose that the following is a diagram of maps of abelian groups.
\begin{eqnarray}
\xymatrix{
&  A_1 \ar[d]^{d^A}  \ar[r]^{f_1}  & B_1 \ar[d]^{d^B} \ar[r]^{g_1} & C_1 \ar[d]^{d^C} \ar[r] & 0 \\
0 \ar[r] & A_0  \ar[r]^{f_0} & B_0 \ar[r]^{g_0} & C_0 }
\end{eqnarray}
Suppose that the maps are compatible ($d \circ f = f \circ d$ and $d \circ g = g \circ
d$) and that the top and bottom rows are both exact.  Then, there exists a homomorphism
$\lambda \colon \ker f \to \coker d$ such that the sequence
\begin{eqnarray*}
\xymatrix{\ker d \ar[r] & \ker e \ar[r] &
\ker f \ar[r]^\lambda & \coker d \ar[r] & \coker e \ar[r]
& \coker f }
\end{eqnarray*}
is exact.
\end{lemma}

\begin{proof}
This follows by repeating the previous proof with steps (1) and (6) omitted.
\end{proof}

Now we are ready to prove Proposition~\ref{realsnakelemma} for arbitrary $m$. Take any $n
\in \{ 1, 2, \ldots, m \}$.  As we know, in the diagram
\begin{eqnarray}
\xymatrix{
& \coker d_{n+1}^X \ar[r] \ar[d] & \coker d_{n+1}^Y \ar[r] \ar[d] & \coker d_{n+1}^Z \ar[r] \ar[d] & 0 \\
0 \ar[r] & \ker d_{n-1}^X \ar[r] & \ker d_{n-1}^Y \ar[r] & \ker d_{n-1}^Z   }
\end{eqnarray}
induced by diagram~(\ref{bigcplxdiagram}), both rows are exact.  Therefore
Lemma~\ref{simplesnakelemma} implies that
\begin{eqnarray}
\xymatrix{ H_n ( X_\bullet ) \ar[r] & H_n ( Y_\bullet ) \ar[r] &  H_n ( Z_\bullet )  \ar[lld]^{\gamma_n}\\
H_{n-1} ( X_\bullet ) \ar[r] & H_{n-1} ( Y_\bullet ) \ar[r] & H_{n-1} ( Z_\bullet ) }
\end{eqnarray}
is an exact sequence for some homomorphism $\gamma_n$.  This completes the proof.

\section{Properties of Barycentric Subdivision}\label{subdivisionappendix}

This part of the appendix is a supplement to
\textit{\nameref{barycentricsubdivisionsection}}.  Our purpose here is to prove two
facts:
\begin{enumerate}
\item \label{euleritem} For any finite abstract simplicial complex $\Delta$,
the Euler characteristic of $\bar ( \Delta )$ is the same
as that of $ \Delta $.
\item If $\Delta$ is collapsible,
then $\bar ( \Delta )$ is also collapsible.
\end{enumerate}
A more extensive discussion of the relationship between collapsibility and barycentric
subdivision can be found in \cite{welker1999}.

We begin by considering the Euler characteristic property.  Recall that the Euler
characteristic of a subset $S$ of a simplicial complex $\Delta$ is given by
\begin{eqnarray}
\chi ( S ) = \sum_{n \geq 0} (-1)^n \left| \left\{
Q \in S , \dim ( Q ) = n \right\} \right|.
\end{eqnarray}
Let us define three basic abstract simplicial complexes. For any $n \geq 1$, let $\Pi_n$
denote the abstract simplicial complex consisting of the set of all nonempty subsets of
$\{ 0, 1, \ldots, n \}$. Let $\Theta_n = \Pi_n \smallsetminus \{ [0, 1, 2, \ldots , n ]
\}$ and $\Lambda_n = \Theta_n \smallsetminus \{ [ 0, 1, 2, \ldots, n-1 ] \}$. It is easy
to see that $\Pi_n$ and $\Lambda_n$ are both collapsible.

If $\Delta$ is an abstract simplicial complex and $t$ is an element not contained in the
vertex set of $\Delta$, then let us say that \textbf{cone} of $\Delta$ over~$t$, denoted
by $t \star \Delta$, is the simplicial complex that arises from adding to $\Delta$ all
sets of the form $\{ t \} \cup Q$ with $Q \in \Delta$ or $Q = \emptyset$.  Note that $n
\star \Pi_{n-1} = \Pi_n$.

\begin{lemma}\label{stareulerlemma}
Let $\Delta$ be a finite abstract simplicial complex, and let $t$ be an element that is
not contained in the vertex set of $\Delta$.  Then, $\chi ( t \star \Delta ) = 1$.
\end{lemma}

\begin{proof}
For any $Q \in \Delta$, the set
\begin{eqnarray}\label{twosimplexset}
\{ Q , Q \cup \Delta \}
\end{eqnarray}
has Euler characteristic zero.  The set $( t \star \Delta ) \smallsetminus \{ \{ t \} \}$
can be partitioned into such two-member sets. Therefore $\chi ( t \star \Delta ) = \chi (
\{ \{ t \} \} ) = 1$.
\end{proof}

\begin{proposition}\label{baryeulerappendixprop}
Let $\Sigma$ be finite abstract simplicial complex of dimension $n$.
Then, $\chi ( \bar ( \Sigma ) ) = \chi ( \Sigma )$.
\end{proposition}

\begin{proof}
We induct on $n$.  The base case ($n = 0$) is trivial.  Suppose that $n \geq 1$ and that
the statement is known to hold for all complexes of dimension less than $n$.

The Euler characteristic of $\Theta_n$ is
\begin{eqnarray}
\chi ( \Theta_n ) = \sum_{k = 0}^{n-1}
(-1)^k \binom{n+1}{k+1} = 1 - (-1)^n,
\end{eqnarray}
therefore by inductive assumption, $\chi ( \bar ( \Theta_n ) )$ is also equal to $1 -
(-1)^n$. The Euler characteristic of $\bar ( \Pi_n ) = [0, 1, \ldots , n] \star \bar (
\Theta_n )$ is $1$ by Lemma~\ref{stareulerlemma}. Therefore,
\begin{eqnarray}\label{pivstheta}
\chi ( \bar ( \Pi_n ) \smallsetminus \bar ( \Theta_n ) )
= 1 - (1 - (-1)^n ) = (-1)^n.
\end{eqnarray}

An easy consequence of equation~(\ref{pivstheta}) is that if $\Delta$ is an
\hbox{$n$-dimensional} abstract simplicial complex and $\Delta' \subset \Delta$ is a
subcomplex that arises from deleting a single $n$-simplex from $\Delta$, then
\begin{eqnarray}\label{singlesimplexdiff}
\chi ( \bar ( \Delta ) \smallsetminus \bar ( \Delta' ) )
& = & (-1)^n.
\end{eqnarray}
Let $\Sigma$ be a finite $n$-dimensional abstract simplicial complex.  Let $d$ be the
number of $n$-simplicies in $\Sigma$, and let $\Sigma^{(n-1)} \subseteq \Sigma$ be the
subcomplex that arises from deleting all $n$-simplicies. Then,
\begin{eqnarray}
\chi ( \bar ( \Sigma ) ) = \chi ( \bar ( \Sigma^{(n-1)} ))  + d \cdot (-1)^n.
\end{eqnarray}
Since $\chi ( \bar ( \Sigma^{(n-1)} ) ) = \chi ( \Sigma^{(n-1)} )$ by inductive
assumption, we have $\chi ( \bar ( \Sigma ) ) = \bar (\Sigma )$ as desired.
\end{proof}

Now we turn to collapsibility. If $\Sigma$ is an abstract simplicial complex and $\Sigma'
\subseteq \Sigma$ is a subcomplex, then let us say that \textbf{$\Sigma$ can be collapsed
onto $\Sigma'$} if there exists a sequence of elementary collapses (or equivalently, a
sequence of primitive elementary collapses),
\begin{eqnarray}
\Sigma = \Sigma_0, \Sigma_1 , \ldots, \Sigma_m = \Sigma'.
\end{eqnarray}

\begin{lemma}\label{conelemma}
If $\Sigma$ is a collapsible abstract simplicial complex and $t$ is not in the vertex set
of $\Sigma$, then the cone $ t \star \Sigma$ can be collapsed onto $\Sigma$.
\end{lemma}

\begin{proof}
There must exist a collapsing procedure which collapses $\Sigma$ to a single $0$-simplex
$\{ s \} \in \Sigma$.  The same procedure collapses $t \star \Sigma$ onto the subcomplex
\begin{eqnarray}
\Sigma \cup \{ t , s \} \cup \{ s \},
\end{eqnarray}
which can be collapsed onto $\Sigma$.
\end{proof}

Let $\bar ( \Pi_n )$ denote the barycentric subdivision of $\Pi_n$ (see
Definition~\ref{barycentricdef}).  The set $\bar ( \Pi_n )$ is the set of all
$\subset$-chains of nonempty subsets of $\{ 0, 1, \ldots, n \}$. The cone
\begin{eqnarray}
[0, 1, 2, \ldots, n ] \star
\bar ( \Lambda_n )
\end{eqnarray}
is a subcomplex of $\bar ( \Pi_n )$.

\begin{lemma}\label{conecollapselemma}
The abstract simplicial complex $\bar ( \Pi_n )$ can be collapsed onto $[0, 1, \ldots, n
] \star \bar ( \Lambda_n )$.
\end{lemma}

\begin{proof}
Let
\begin{eqnarray}
\Gamma = \left\{ Q \in \bar ( \Pi_n ) \mid
[0, 1, \ldots, n-1] \in S, [0, 1, \ldots, n ] \notin S \right\}.
\end{eqnarray}
Write the elements of $\Gamma$ as a sequence
\begin{eqnarray}
Q_1, Q_2, \ldots, Q_m \in \Gamma
\end{eqnarray}
so that $\dim ( Q_i ) \geq \dim ( Q_{i+1} )$.  Let
\begin{eqnarray}
\bar ( \Pi_n ) = \Sigma_0, \Sigma_1, \Sigma_2, \ldots , \Sigma_m
\end{eqnarray}
be the sequence of subcomplexes of $\bar ( \Pi_n )$ that arises from deleting the pairs
$(Q_i, Q_i \cup \{ [0, 1, \ldots , n ] \} )$ from $\bar ( \Pi_n )$ one at a time.  This
is a sequence of primitive elementary collapses, and its final term is $ [0, 1, \ldots,
n]  \star \bar ( \Lambda_n )$.
\end{proof}

\begin{proposition}\label{baryappendixprop}
Let $\Delta$ be an abstract simplicial complex of dimension $n \geq 0$. Suppose that
$\Delta'$ is a subcomplex of $\Delta$, and suppose that $\Delta$ can be collapsed onto
$\Delta'$.  Then, $\bar ( \Delta )$ can be collapsed onto $\bar ( \Delta' )$.
\end{proposition}

\begin{proof}
Again we induct on $n$.  The base case ($n = 0$) is immediate. Suppose that $n \geq 1$
and that the proposition is known to hold for all simplicial complexes of dimension
smaller than $n$.

Every collapsing sequence can be expanded into a sequence of primitive elementary
collapses. Therefore, to prove the proposition, it suffices to show that for any $k \geq
n$, the complex $\bar ( \Pi_k )$ can be collapsed onto $\bar ( \Lambda_k )$.  By
inductive assumption, we know this to be true for $k < n$, and so we need only prove it
for $k = n$.

By Lemma~\ref{conecollapselemma}, the complex $\bar ( \Pi_n )$ can be collapsed onto $[0,
1, \ldots, n ]  \star \bar ( \Lambda_n )$.  The complex $\Lambda_n$ is collapsible, and
therefore $\bar ( \Lambda_n )$ is collapsible by inductive assumption (since $\dim
\Lambda_n = n-1$).  By Lemma~\ref{conelemma}, the cone $ [0, 1, \ldots, n ]  \star \bar (
\Lambda_n )$ can be collapsed to $\bar ( \Lambda_n )$. This completes the proof.
\end{proof}

\begin{acknowledgements}
Many thanks to Yaoyun Shi for introducing me to the topic of evasiveness, and for his
help with the survey material in the introduction. Thanks also to an anonymous reviewer
whose comments helped me improve a previous draft. This work was supported by NSF grants
1017335 and 1216729.
\end{acknowledgements}

\bibliographystyle{ieeesort}
\bibliography{055}

\begin{thebibliography}{10}

\bibitem{ambainis2004}
A.~Ambainis, ``Quantum query complexity and lower bounds,'' in {\it Classical
  and New Paradigms of Computation and their Complexity Hierarchies{\rm ,
  vol.~23 of} Trends in Logic}, pp.~15--32, 2004.

\bibitem{ash2007}
R.~Ash, {\it Basic Abstract Algebra}.
\newblock Dover Publications, 2007.

\bibitem{bbkk2010}
L.~Babai, A.~Banerjee, R.~Kulkarni, and V.~Naik, ``Evasiveness and the
  distribution of prime numbers,'' in {\it Symposium on Theoretical Aspects of
  Computer Science}, pp.~71--82, 2010.

\bibitem{bbl1974}
M.~Best, P.~{Emde Boas}, and H.~Lenstra, ``A sharpened version of the
  Aanderaa-Rosenberg conjecture,'' {\it Stichting Mathematisch Centrum. Zuivere
  Wiskunde}, vol.~ZW 30/74, pp.~1--20, 1974.

\bibitem{bollobas1976}
B.~Bollobas, ``Complete subgraphs are elusive,'' {\it Journal of Combinatorial
  Theory, Series B}, vol.~21, no.~1, pp.~1--7, August 1976.

\bibitem{bcwz1999}
H.~Buhrman, R.~Cleve, R.~{de~Wolf}, and C.~Zalka, ``Bounds for small-error and
  zero-error quantum algorithms,'' in {\it Proceedings of the 40th Annual
  Symposium on Foundations of Computer Science}, pp.~358--368, 1999.

\bibitem{ck2007}
A.~Chakrabati and S.~Khot, ``Improved lower bounds on the randomized complexity
  of graph properties,'' {\it Random Structures and Algorithms}, vol.~30,
  no.~3, pp.~427--440,  2007.

\bibitem{cks2002}
A.~Chakrabati, S.~Khot, and Y.~Shi, ``Evasiveness of subgraph containment and
  related properties,'' {\it SIAM Journal on Computing}, vol.~31, no.~3,
  pp.~866--875,  2002.

\bibitem{ck2010}
A.~Childs and R.~Kothari, ``Quantum query complexity of minor-closed graph
  properties,'' in {\it Proceedings of the 28th International Symposium on
  Theoretical Aspects of Computer Science}, pp.~661--672, 2011.

\bibitem{duandko}
D.~Du and K.~Ko, {\it Theory of Computational Complexity}.
\newblock Wiley, 2000.

\bibitem{fkw2002}
E.~Friedgut, J.~Kahn, and A.~Wigderson, ``Computing graph properties by
  randomized subcube partitions,'' in {\it Randomization and Approximation
  Techniques in Computer Science{\rm , vol.~2483 of} Lecture Notes in Computer
  Science}, pp.~105--113, 2002.

\bibitem{groger1992}
H.~Groger, ``On the randomized complexity of monotone graph properties,'' {\it
  Acta Cybernetica}, vol.~10, no.~3, pp.~119--127,  1992.

\bibitem{hajnal1991}
P.~Hajnal, ``An $\omega(n^{\frac{4}{3}})$ lower bound on the randomized
  complexity of graph properties,'' {\it Combinatorica}, vol.~11, no.~2,
  pp.~131--143,  1991.

\bibitem{hatcher2002}
A.~Hatcher, {\it Algebraic Topology}.
\newblock Cambridge University Press, 2002.

\bibitem{hr1972}
R.~Holt and E.~Reingold, ``On the time required to detect cycles and
  connectivity in graphs,'' {\it Mathematical Systems Theory}, vol.~6,
  no.~1--2, pp.~103--106,  1972.

\bibitem{ht1974}
J.~Hopcroft and R.~Tarjan, ``Efficient planarity testing,'' {\it Journal of the
  ACM}, vol.~21, no.~4, pp.~549--568, October 1974.

\bibitem{jonsson2008}
J.~Jonsson, {\it Simplicial Complexes of Graphs}.
\newblock Springer-Verlag, 2008.

\bibitem{kss1984}
J.~Kahn, M.~Saks, and D.~Sturtevant, ``A topological approach to evasiveness,''
  {\it Combinatorica}, vol.~4, no.~4, pp.~297--306,  1984.

\bibitem{king1990}
V.~King, ``A lower bound for the recognition of digraph properties,'' {\it
  Combinatorica}, vol.~10, no.~1, pp.~53--59,  1990.

\bibitem{king1991}
V.~King, ``An $\omega(n^{5/4})$ lower bound on the randomized complexity of
  graph properties,'' {\it Combinatorica}, vol.~11, no.~1, pp.~23--32,  1991.

\bibitem{kirkpatrick1974}
D.~Kirkpatrick, ``Determining graph properties from matrix representations,''
  in {\it Proceedings of Sixth Annual ACM Symposium on Theory of Computing},
  pp.~84--90, 1974.

\bibitem{kk1980}
D.~Kleitman and D.~Kwiatowksi, ``Further results on the Aanderaa-Rosenberg
  conjecture,'' {\it Journal of Combinatorial Theory, Series B}, vol.~28,
  pp.~85--95, 1980.

\bibitem{kt2010}
T.~Korneffel and E.~Triesch, ``An asymptotic bound for the complexity of
  monotone graph properties,'' {\it Combinatorica}, vol.~30, no.~6,
  pp.~735--743,  2010.

\bibitem{kozlov2008}
D.~Kozlov, {\it Combinatorial Algebraic Topology}.
\newblock Berlin Heidelberg: Springer, 2008.

\bibitem{lang}
S.~Lang, {\it Algebra}.
\newblock Number 211 in Graduate Texts in Mathematics. Springer-Verlag, rev.
  3rd Edition, 2002.

\bibitem{ly2002}
L.~Lovasz and N.~Young, ``Lecture notes on evasiveness of graph properties,''
  arXiv:cs/0205031v1, May 2002.

\bibitem{longueville2013}
{Mark de~Longueville}, {\it A Course in Topological Combinatorics}.
\newblock Springer, 2013.

\bibitem{matousek2008}
J.~Matousek, {\it Using the Borsuk-Ulam Theorem}.
\newblock Springer-Verlag, corrected 2nd printing Edition, 2008.

\bibitem{mw1975}
E.~Milner and D.~Welsh, ``On the computational complexity of graph theoretical
  properties,'' in {\it Proceedings of the Fifth British Combinatorial
  Conference}, pp.~471--487, 1975.

\bibitem{munkres}
J.~Munkres, {\it Elements of Algebraic Topology}.
\newblock Addison-Wesley, 1984.

\bibitem{odonnel2005}
R.~O'Donnell, M.~Saks, O.~Schramm, and R.~Serveido, ``Every decision tree has
  an influential variable,'' in {\it Proceedings of the 46th Annual IEEE
  Symposium on Foundations of Computer Science}, pp.~31--39, 2005.

\bibitem{oliver1975}
R.~Oliver, ``Fixed-point sets of group actions on finite acyclic complexes,''
  {\it Commentarii mathematici Helvetici}, vol.~50, pp.~155--177, 1975.

\bibitem{rv1976}
R.~Rivest and J.~Vuillemin, ``On recognizing graph properties from adjacency
  matrices,'' {\it Theoretical Computer Science}, vol.~3, pp.~371--384, 1976.

\bibitem{rosenberg1973}
A.~Rosenberg, ``On the time required to recognize properties of graps: A
  problem,'' {\it SIGACT News}, vol.~5, no.~4, pp.~15--16, October 1973.

\bibitem{smith1941}
P.~Smith, ``Fixed-point theorems for periodic transformations,'' {\it American
  Journal of Mathematics}, vol.~63, no.~1, pp.~1--8,  1941.

\bibitem{syz2004}
X.~Sun, A.~Yao, and S.~Zhang, ``Graph properties and circular functions: How
  low can quantum query complexity go?,'' in {\it Proceedings of the IEEE
  Annual Conference on Computation Complexity}, pp.~286--293, 2004.

\bibitem{triesch1994}
E.~Triesch, ``Some results on elusive graph properties,'' {\it SIAM Journal on
  Computing}, vol.~23, no.~2, pp.~247--254,  1994.

\bibitem{triesch1996}
E.~Triesch, ``On the recognition complexity of some graph properties,'' {\it
  Combinatorica}, vol.~16, no.~2, pp.~259--268,  1996.

\bibitem{welker1999}
V.~Welker, ``Constructions preserving evasiveness and collapsibility,'' {\it
  Discrete Mathematics}, vol.~207, no.~1--3, pp.~243--255,  1999.

\bibitem{yao1988}
A.~Yao, ``Monotone bipartite graph properties are evasive,'' {\it SIAM Journal
  on Computing}, vol.~17, no.~3, pp.~517--520, June 1988.

\bibitem{yao1991}
A.~Yao, ``Lower bounds to randomized algorithms for graph properties,'' {\it
  Journal of Computer and System Sciences}, vol.~42, no.~3, pp.~267--287, June
  1991.

\bibitem{zagier1997}
D.~Zagier, ``Newman's short proof of the prime number theorem,'' {\it American
  Mathematical Monthly}, vol.~104, no.~8, pp.~705--708, October 1997.

\end{thebibliography}

\end{document}